\documentclass{amsart}

    \usepackage[utf8]{inputenc}
    \usepackage{amsfonts}
    \usepackage{amsmath}
    \usepackage{amsthm}
    \usepackage{amssymb}
    \usepackage{dsfont}
    \usepackage{mathtools} 
    \usepackage{bm} 
    \usepackage{enumitem} 
    \usepackage{stmaryrd} 
    \SetSymbolFont{stmry}{bold}{U}{stmry}{m}{n}
        
    \usepackage[style=alphabetic,maxbibnames=99]{biblatex}
        \DeclareFieldFormat{title}{\mkbibquote{#1}}
    
    \usepackage{hyperref}
    \usepackage[noabbrev,nameinlink,capitalise]{cleveref}
        \crefname{subsection}{Subsection}{Subsections}
        \crefname{subsection}{Subsection}{Subsections}
    
    \usepackage{subcaption}
    \usepackage{color}
    \usepackage[dvipsnames]{xcolor}
    \usepackage{graphicx}\graphicspath{{./images/}}
    \usepackage{copyrightbox}
    \usepackage{tikz}
        \usetikzlibrary{calc}
        \usetikzlibrary{arrows}
        \usetikzlibrary{shapes}
        \tikzset{every picture/.style = {line width=1pt}}
        \tikzset{vertex/.style = {shape=circle,draw,fill=white,inner sep=1pt, outer sep=1pt}}
        \tikzset{state/.style = {shape=ellipse,draw,inner sep=1pt, outer sep=1pt}}
        \tikzset{edge/.style = {->,> = latex'}}

    \theoremstyle{plain}
        \newtheorem{theorem}{Theorem}[section]
        \newtheorem*{witzel}{Witzel's Theorem}
        \newtheorem*{main}{Main Theorem}
        \newtheorem{corollary}[theorem]{Corollary}
        \newtheorem{lemma}[theorem]{Lemma}
        \newtheorem{proposition}[theorem]{Proposition}
        
    \theoremstyle{definition}
        \newtheorem{definition}[theorem]{Definition}
        \newtheorem{example}[theorem]{Example}
        \newtheorem{assumption}[theorem]{Assumption}
    \theoremstyle{remark}
        \newtheorem{remark}[theorem]{Remark}
        \newtheorem{question}[theorem]{Question}

    \newtheoremstyle{named}{}{}{\itshape}{}{\bfseries}{.}{.5em}{\thmnote{#3's }#1} 
        \theoremstyle{named}

    \title{Eventually Self-Similar Groups acting on Fractals}
    \author{Davide Perego \and Matteo Tarocchi}
    \date{}
    
    \thanks{Both authors are members of the Gruppo Nazionale per le Strutture Algebriche, Geometriche e le loro Applicazioni (GNSAGA) of the Istituto Nazionale di Alta Matematica (INdAM).\\
    The first author acknowledges support from the Grant QUALIFICA by Junta de Andalucía grant number QUAL21 005 USE, from the research grant PID2022-138719NA-I00 (Proyectos de Generación de Conocimiento 2022) financed by the Spanish Ministry of Science and Innovation, from the Swiss Government Excellence Scholarship and from Swiss NSF grant 200020-200400.}

    \address{Section de mathématiques, Université de Genève, rue du Conseil-Général 7-9, 1205 Genève, Switzerland}
    \email{\href{mailto:dperego9@gmail.com}{dperego9@gmail.com}}
    
    \address{Dipartimento di Matematica e Applicazioni, Universit\`a degli Studi di Milano-Bicocca, Ed. U5, Via R.Cozzi 55, 20125 Milano, Italy, EU}
    \email{\href{mailto:matteo.tarocchi.math@gmail.com}{matteo.tarocchi.math@gmail.com}}

    \newcommand{\f}{\mathfrak{f}}
    \newcommand{\F}{\mathfrak{F}}
    \newcommand{\g}{\mathfrak{g}}
    \newcommand{\G}{\mathfrak{G}}
    \newcommand{\R}{\mathcal{R}}
    \newcommand{\C}{\mathcal{C}}
    \newcommand{\E}{\mathcal{E}}
    \newcommand{\dpt}{\mathrm{D}}
    
    \newcommand{\Aut}{\mathrm{Aut}}
    \newcommand{\Iso}{\mathrm{Iso}}
    \newcommand{\id}{\mathrm{id}}

\begin{document}

\begin{abstract}
Generalizing work by Belk and Forrest, we develop almost expanding hyperedge replacement systems that build fractal topological spaces as quotients of edge shifts under certain ``gluing'' equivalent relations.
We define ESS groups, which are groups of homeomorphisms of these spaces that act as a finitary asynchronous transformations followed by self-similar ones, akin to the action of Scott-R\"{o}ver-Nekrashevych groups on the Cantor space.
We provide sufficient conditions for finiteness properties of such groups, which allow us to show that the airplane and dendrite rearrangement groups have type $F_\infty$, that a group combining dendrite rearrangements and the Grigorchuk group is finitely generated, and that certain ESS groups of edge shifts have type $F_\infty$ (partially addressing a question of Deaconu), in addition to providing new proofs of previously known results about several Thompson-like groups.
\end{abstract}

\maketitle


\section*{Introduction}

Scott-R\"{o}ver-Nekrashevych groups (developed in different times by multiple authors \cite{Scott,Rover,NekrashevychGroups}) are groups of certain homeomorphisms of the Cantor space that have a ``Thompson-like'' finitary asynchronous action followed by a self-similar action.
These are fascinating and exotic groups, for example due to their notable application \cite{SWZ19} as a source of simple groups with arbitrary finiteness properties.
Finiteness properties represent a current challenge in the field, as shown by the presence of papers devoted to the study of single such groups, such as \cite{RoverFinfty}, and by questions such as one in the Introduction of \cite{ExpandingMaps}, where Nekrashevych asks whether Scott-R\"{o}ver-Nekrashevych groups associated to iterated monodromy groups (introduced in \cite{IMG}, see also \cite[Chapter 5]{NekSSG}) have type $F_\infty$.

The family of \textit{rearrangement groups of fractals}, introduced in \cite{BF19}, is another family of ``Thompson-like" groups which recently gained interest (see e.g. \cite{QS_BF,BasilicaDense,IG,graphdiagramgroups,TBnotfinpres}) and provides a wide framework for understanding Thompson-like homeomorphisms of a variety of spaces.
Theorem 4.1 of \cite{BF19} provides sufficient conditions for the finiteness property $F_\infty$ of such groups, though many well-known groups of type $F_\infty$ belonging to this class do not satisfy it.

We generalize the construction of \textit{limit spaces} of edge replacement systems from \cite{BF19} in two distinct directions:
by considering hypergraphs instead of graphs and by relaxing the so-called ``expanding'' condition on the replacement system.
The former (already suggested in \cite{QS_BF}) allows us to build spaces whose self-similar pieces intersect at more than two points, such as the Sierpinski triangle and the Apollonian gasket (\cref{fig.sierpinski.triangle.and.apollonian.gasket}), while the latter allows the inclusion of isolated points, which is useful to realize natural spaces on which groups such as $QF$, $QT$, $QV$ and the Houghton groups act. The resulting spaces are \textit{finitely ramified fractals} in the sense of \cite{QS_BF}.
In \cref{sec.replacement.systems,,sec.topology} we rework some proofs about replacement systems in this larger setting of almost expanding hyperedge replacement systems.
Some results easily adapt, such as the rationality of the gluing relation proved in \cite{rationalgluing}, as we explain in \cref{sub.rational.gluing}.

Then, in \cref{sec.ESS}, we define eventually self-similar groups of homeomorphisms of these spaces (\textit{ESS groups}, in short), akin to the Scott-R\"{o}ver-Nekrashevych groups: they have a ``Thompson-like'' finitary asynchronous action followed by a self-similar action determined by a \textit{self-similar tuple}.
This is a tuple of groups, one for each color (type), such that the restriction of any element belongs to one of the groups, naturally generalizing self-similar groups to the setting of edge shifts.
A similar construction was described in \cite{SelfSimilarGroupoid}, with a significant difference: we never allow isomorphisms between cones of different colors (types).
The self-similar groupoid from \cite{SelfSimilarGroupoid} becomes a self-similar tuple when there are no homeomorphisms between cones of different colors (types), in which case the groups introduced in \cite{Deaconu} coincide with our eventually self-similar groups on edge shifts.

The family of ESS groups include rearrangement groups, for which the self-similar action is trivial, together with whole groups of homeomorphisms of certain fractals, such as the groups of homeomorphisms of the Sierpinski triangle and of the Apollonian gasket (see \cref{sub.triangle.gasket.homeo.groups}), and many new groups, such as an analogous of R\"{o}ver group acting on dendrites, where the self-similar action descends from Grigorchuk group (see \cref{sub.dendrites}).
Moreover, the family of ESS groups include the Scott-R\"{o}ver-Nekrashevych groups \cite{Scott,Rover,NekrashevychGroups} (and similar construction on edge shifts \cite{Deaconu}), the Houghton groups \cite{Houghton1978TheFC} and the Thompson-like groups $QF$, $QT$ and $QV$ \cite{QV}.
We suspect that, with a suitable metric, for connected limit spaces, ESS groups are groups of quasisymmetries (see \cite{QS_BF}).

We bring together strategies from \cite{SWZ19} and \cite{BF19} and we apply a general machinery developed in \cite{Witzel} to determine a general sufficient condition for finiteness properties of ESS groups, proving the following result.

\begin{main}
For nice enough replacement systems, ESS groups inherit finiteness properties from their building blocks.
More precisely, for an almost expanding replacement system $\R$ and a compatible self-similar tuple $\mathbb{G}$, assume that for every $m \in \mathbb{N}$ all but finitely many hypergraph expansions of $\R$ admit at least $m$ parallel simple $\pi$-contractions.
If the groups of the self-similar tuple have type $F_m$, then the ESS group $E_\R\mathbb{G}$ has type $F_k$, where $k$ depends on $m$ and $\R$.
In particular, if the replacement system is $\infty$-contractive and the groups of $\mathbb{G}$ have type $F_\infty$, then $E_\R\mathbb{G}$ has type $F_\infty$.
\end{main}

Here, a $\pi$-contraction is the inverse of a hyperedge expansion that may change the orientation of hyperedges, depending on the action of the self-similar tuple.

With $\pi$-contractions we will be able to circumvent a problem of \cite[Theorem 4.1]{BF19} described in \cite[Remark 4.6]{BF19}.
Thus, our theorem applies to certain rearrangement groups that do not satisfy the hypotheses of \cite[Theorem 4.1]{BF19}.
Moreover, our theorem applies to many new ESS groups of homeomorphisms of dendrites, as described in \cref{sub.dendrites}:
for every $n \geq 3$, one can build such a group $E_{\mathcal{D}_n}G$ acting on the Wa\.zewski dendrite $D_n$ for any self-similar group $G$ on the $(n-1)$-rooted tree. Such groups are still largely mysterious and it is the authors' opinion that they may provide interesting exotic examples.
We also show that many ESS groups on edge shifts inherit their finiteness properties from the groups of the self-similar tuple, which provides a partial answer to questions posed in \cite{Deaconu} after Corollary 5.5.

We observe two limits of these techniques.
On one hand, stating a general theorem about such a large family of groups requires sacrificing precision on evaluating the connectivity of the complexes, as showed by the Houghton groups in \cref{cor.example.known.results}(D).
On the other hand, there are Scott-R\"{o}ver-Nekrashevych groups with strong finiteness properties associated to self-similar groups with weak finiteness properties.
For example, the R\"{o}ver group from \cite{Rover} is the ESS group on the full shift $\{0,1\}^\omega$ with the Grigorchuk group as self-similar group.
It has type $F_\infty$ as shown in \cite{RoverFinfty}, even if the Grigorchuk group is not finitely presented.
For these reasons, we believe that a more detailed investigation of specific examples (such as the ones described in \cref{sub.dendrites}) could lead to a deeper understanding of interesting groups and a more general strategy for the study of finiteness properties.


\section{Background on Directed Hypergraphs and Edge Shifts}

Let us first introduce what we need about hypergraphs and edge shifts.

\subsection{Directed Hypergraphs and \texorpdfstring{$\pi$}{pi}-isomorphisms}

In this subsection we introduce the generalized notion of graph that we are going to use.

\subsubsection{Basic Definitions and Assumptions}

\begin{definition}
A \textbf{hypergraph} is a quintuple $(V,E,\lambda,C,c)$ where
\begin{itemize}
    \item $V$ is a set whose elements are called \textbf{vertices};
    \item $E$ is a set whose elements are called \textbf{hyperedges}, each with a given \textbf{order}, which is $\mathrm{ord}(e) \in \mathbb{N}_{\geq 2}$;
    \item $\lambda$ associates to each hyperedge $e$ a map $\lambda_e \colon \{1, \dots, \mathrm{ord}(e)\} \to V$, each of whose images is called a \textbf{boundary vertex} of $e$;
    \item $c \colon E \to C$ is an assignment of a \textbf{color} $c(e) \in C$ to each hyperedge.
\end{itemize}
If $\mathrm{ord}(e) = d$, we say that $e$ is a \textbf{$\boldsymbol{d}$-hyperedge}.
A hypergraph whose hyperedges all have order $2$ is called a \textbf{graph}.
\end{definition}

For example, \cref{fig.apollonian.gasket.expansion} portrays a hypergraph with four $3$-edges.
Note that, unlike graphs, picture alone fail to clarify the actual orientation of hyperedges.

If $c(e) = k \in C$ we say that $e$ is \textbf{$\boldsymbol{k}$-colored} or that it is \textbf{colored by $\boldsymbol{k}$}.
When $c(E)$ is a singleton, we can safely omit colors from the notation.

In general, for a hypergraph $\Gamma$ we denote by $V_\Gamma$, $E_\Gamma$ and $\lambda^\Gamma$ the sets of vertices and hyperedges and the boundary map of $\Gamma$, respectively.

Moreover, $\Gamma_2$ is a \textbf{subhypergraph} of $\Gamma_1$ if $V_{\Gamma_2} \subseteq V_{\Gamma_1}$, $E_{\Gamma_2} \subseteq E_{\Gamma_1}$ and the boundary and coloring maps of $\Gamma_2$ are the restrictions to $E_{\Gamma_2}$ of those of $\Gamma_1$.

\begin{assumption}
\label{ass.hypergraphs}
Throughout this work, we assume the following facts.
\begin{enumerate}
    \item Hypergraphs are \textbf{finite}, i.e., their sets of vertices and hyperedges are finite.
    \item Hypergraphs do not have \textbf{isolated vertices}, which are vertices that are not the boundary of any hyperedge.
\end{enumerate}
\end{assumption}

\begin{definition}
\label{def.incident.adjacent.isolated}
We say that an hyperedge $e$ is
\begin{itemize}
    \item \textbf{incident} on a vertex $v$ is $v$ is a boundary vertex of $e$;
    \item \textbf{adjacent} to an hyperedge $f$ if $e$ and $f$ share some boundary vertex;
    \item \textbf{isolated} if $e$ is not adjacent to any other vertex.
\end{itemize}
\end{definition}

\subsubsection{Hypergraph \texorpdfstring{$\pi$}{pi}-isomorphisms}

The notion of isomorphisms in the context of directed hypergraphs is the following.

\begin{definition}
Given two hypergraphs $A$ and $B$, a \textbf{hypergraph isomorphism} $\f \colon A \to B$ is a collection $(\f_V,\f_E)$ of two bijections $ \f_V \colon V_A \to V_B \text{ and } \f_E \colon E_A \to E_B $ that enjoy the following properties:
\begin{enumerate}
    \item the maps $\mathrm{ord}$ and $\f_E$ commute (i.e., $\f_E$ preserves the order of hyperedges);
    \item $\f_V (\lambda_e(i)) = \lambda_{\f_E(e)}(i)$ for all $i = 1, \ldots, \mathrm{ord}(e)$, for each $e \in E_A$ (i.e., the maps $\f_E$ and $\f_V$ preserve adjacency).
\end{enumerate}
\end{definition}

Note that, thanks to the second point of \cref{ass.hypergraphs}, each hypergraph isomorphism is entirely determined solely by the map $\f_E$.

Certain maps that we consider are almost hypergraph isomorphisms:
they are bijection of hyperedges that preserve adjacency up to some permutation of the boundary vertices at each hyperedge.
This is formalized by the following definition.

\begin{definition}
\label{def.pi.isomorphism}
Given two hypergraphs $A$ and $B$, a \textbf{$\boldsymbol{\pi}$-hypergraph isomorphism} (or simply \textbf{$\boldsymbol{\pi}$-isomorphism}) $\f \colon A \to B$ is a collection $(\f_V,\f_E,\pi)$ of two bijections
$ \f_V \colon V_A \to V_B \text{ and } \f_E \colon E_A \to E_B $
together with a map
\[ \pi \colon E_A \to \bigcup_{n \leq 1} \mathrm{Sym}(n),\, e \mapsto \pi_e \in \mathrm{Sym}(\mathrm{ord}(e)) \]
that enjoy the following properties:
\begin{enumerate}
    \item the maps $\mathrm{ord}$ and $\f_E$ commute (i.e., $\f_E$ preserves the order of hyperedges);
    \item $\f_V (\lambda^A_e(i)) = \lambda^B_{\f_E(e)}({\pi_e(i)})$ for all $i = 1, \ldots, \mathrm{ord}(e)$, for all $e \in E_A$ (i.e., the maps $\f_E$ and $\f_V$ preserve adjacency up to the permutations of the boundary vertices determined by $\pi$).
\end{enumerate}
\end{definition}

As examples, \cref{fig.dendrite.pi.isomorphism} depicts a $\pi$-graph isomorphism whose action simply switches the orientation of the edge named $d$ and \cref{fig.triangle.pi.isomorphism} depicts a $\pi$-hypergraph automorphism that essentially reflects around the vertical axis.

\begin{figure}
\centering
\begin{tikzpicture}
    \begin{scope}[xshift=-2.75cm]
    \node[vertex] (ll) at (-1.8,0) {};
    \node[vertex] (lc) at (-.6,0) {};
    \node[vertex] (lt) at (-.6,1.2) {};
    \node[vertex] (lb) at (-.6,-1.2) {};
    \node[vertex] (rc) at (.6,0) {};
    \node[vertex] (rt) at (.6,1.2) {};
    \node[vertex] (rb) at (.6,-1.2) {};
    \node[vertex] (rr) at (1.8,0) {};
    \draw[edge] (lc) to node[above]{$a$} (ll);
    \draw[edge] (lc) to node[above left]{$b$} (lt);
    \draw[edge] (lc) to node[left]{$c$} (lb);
    \draw[edge] (lc) to node[above]{$d$} (rc);
    \draw[edge] (rc) to node[above right]{$e$} (rt);
    \draw[edge] (rc) to node[right]{$f$} (rb);
    \draw[edge] (rc) to node[above]{$g$} (rr);
    \end{scope}
    \draw[-to] (-.15,0) -- node[above]{$\f$} (.15,0);
    \begin{scope}[xshift=2.75cm]
    \node[vertex] (ll) at (-1.8,0) {};
    \node[vertex] (lc) at (-.6,0) {};
    \node[vertex] (lt) at (-.6,1.2) {};
    \node[vertex] (lb) at (-.6,-1.2) {};
    \node[vertex] (rc) at (.6,0) {};
    \node[vertex] (rt) at (.6,1.2) {};
    \node[vertex] (rb) at (.6,-1.2) {};
    \node[vertex] (rr) at (1.8,0) {};
    \draw[edge] (lc) to node[above]{$\f(a)$} (ll);
    \draw[edge] (lc) to node[above left]{$\f(b)$} (lt);
    \draw[edge] (lc) to node[left]{$\f(c)$} (lb);
    \draw[edge] (rc) to node[above]{$\f(d)$} (lc);
    \draw[edge] (rc) to node[above right]{$\f(e)$} (rt);
    \draw[edge] (rc) to node[right]{$\f(f)$} (rb);
    \draw[edge] (rc) to node[above]{$\f(g)$} (rr);
    \end{scope}
    \node[align=center] at (6.3,0) {$\pi_d = (1\,2)$\\$\pi_i = \mathrm{id},\, \forall i \neq d$};
\end{tikzpicture}
\caption{A $\pi$-graph isomorphism.}
\label{fig.dendrite.pi.isomorphism}
\end{figure}
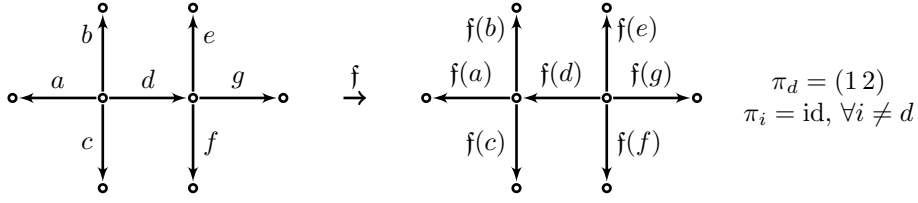

\begin{figure}
\centering
\begin{tikzpicture}
    \begin{scope}[xshift=-2.5cm]
    \node[vertex] (B3) at (90:1.8) {}; \draw (B3) node[left]{$t$};
    \node[vertex] (B1) at (210:1.8) {}; \draw (B1) node[below]{$l$};
    \node[vertex] (B2) at (330:1.8) {}; \draw (B2) node[below]{$r$};
    \node[vertex] (y) at ($(B3)!.5!(B1)$) {}; \draw (y) node[left]{$y$};
    \node[vertex] (z) at ($(B2)!.5!(B1)$) {}; \draw (z) node[below]{$z$};
    \node[vertex] (x) at ($(B2)!.5!(B3)$) {}; \draw (x) node[right]{$x$};
    \fill[black,opacity=.25] (B1.center) -- (z.center) -- (y.center) -- cycle; \node at (210:.9) {$L$};
    \draw (B1) to (z);
    \draw (z) to (y);
    \draw (y) to (B1);
    \fill[black,opacity=.25] (z.center) -- (B2.center) -- (x.center) -- cycle; \node at (330:.9) {$R$};
    \draw (z) to (B2);
    \draw (B2) to (x);
    \draw (x) to (z);
    \fill[black,opacity=.25] (y.center) -- (x.center) -- (B3.center) -- cycle; \node at (90:.9) {$T$};
    \draw (y) to (x);
    \draw (x) to (B3);
    \draw (B3) to (y);
    \end{scope}
    \draw[-to] (-.15,.45) -- node[above]{$\g$} (.15,.45);
    \begin{scope}[xshift=2.5cm]
    \node[vertex] (B3) at (90:1.8) {}; \draw (B3) node[left]{$\g(t)$};
    \node[vertex] (B1) at (210:1.8) {}; \draw (B1) node[below]{$\g(r)$};
    \node[vertex] (B2) at (330:1.8) {}; \draw (B2) node[below]{$\g(l)$};
    \node[vertex] (y) at ($(B3)!.5!(B1)$) {}; \draw (y) node[left]{$\g(x)$};
    \node[vertex] (z) at ($(B2)!.5!(B1)$) {}; \draw (z) node[below]{$\g(z)$};
    \node[vertex] (x) at ($(B2)!.5!(B3)$) {}; \draw (x) node[right]{$\g(y)$};
    \fill[black,opacity=.25] (B1.center) -- (z.center) -- (y.center) -- cycle; \node at (210:.9) {$\g(R)$};
    \draw (B1) to (z);
    \draw (z) to (y);
    \draw (y) to (B1);
    \fill[black,opacity=.25] (z.center) -- (B2.center) -- (x.center) -- cycle; \node at (330:.9) {$\g(L)$};
    \draw (z) to (B2);
    \draw (B2) to (x);
    \draw (x) to (z);
    \fill[black,opacity=.25] (y.center) -- (x.center) -- (B3.center) -- cycle; \node at (90:.9) {$\g(T)$};
    \draw (y) to (x);
    \draw (x) to (B3);
    \draw (B3) to (y);
    \end{scope}
    \node[align=center] at (6,.45) {$\pi_L=(1\,2)$\\$\pi_R=(1\,2)$\\$\pi_T=(1\,2)$};
    \node[align=center] at (-2.5,-2.25) {$L=(l,z,y)$\\$R=(z,r,x)$\\$T=(y,x,t)$};
    \node[align=center] at (2.5,-2.25) {$\g(L)=(\g(l),\g(z),\g(y))$\\$\g(R)=(\g(z),\g(r),\g(x))$\\$\g(T)=(\g(y),\g(x),\g(t))$};
\end{tikzpicture}
\caption{A $\pi$-hypergraph automorphism.}
\label{fig.triangle.pi.isomorphism}
\end{figure}
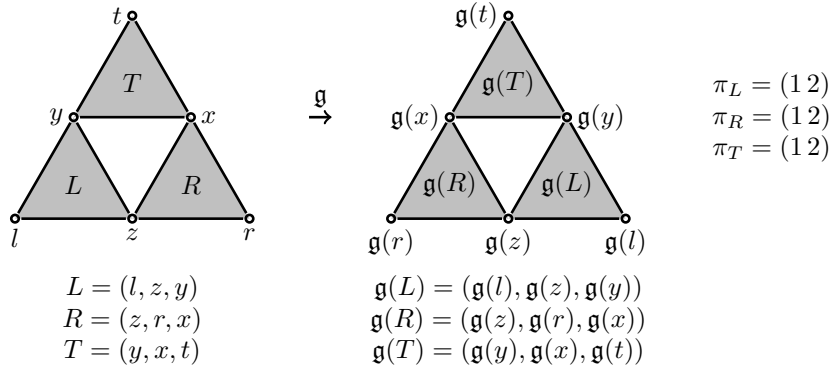

Note that, when $\pi_e$ is trivial for all $e$, a $\pi$-hypergraph isomorphism is a hypergraph isomorphism.
The converse almost holds, but not entirely:
in the presence of hyperedges $e$ with repeated boundary vertices (such as loops in graphs), any $\pi_e$ that permutes the repeated vertices and is otherwise trivial corresponds to the same hypergraph isomorphism.

\begin{proposition}
\label{prop.pi.isomorphism.composition}
Given a $\sigma$-hypergraph isomorphism $f \colon B \to C$ and a $\tau$-hypergraph isomorphism $g \colon A \to B$, their composition $f \circ g$ is a $\pi$-hypergraph isomorphism $A \to C$ with $\pi_e \coloneq \sigma_{g(e)} \circ \tau_e$.
\end{proposition}

\begin{proof}
Clearly $f_V \circ g_V \colon V_A \to V_C$ and $f_E \circ g_E \colon E_A \to E_C$ are bijections, as they are compositions of bijections.
Since $f_E$ and $g_E$ commute with $\mathrm{ord}$, their composition does too.

Let $e \in E_A$.
Using the point (2) of \cref{def.pi.isomorphism}, we have that
\[ f_V \circ g_V \circ \lambda^A_e(i) = f_V \circ \lambda^B_{g_E(e)} \circ \tau_e(i) = \lambda^C_{f_E \circ g_E(e)} \circ \sigma_{g_E (e)} \circ \tau_e(i), \]
which is precisely what we needed in order to prove that $f \circ g$ is a $\pi$-hypergraph isomorphism for the map $\pi_e = \sigma_{g_E(e)} \circ \tau_e$.
\end{proof}

\subsection{Edge Shifts}
\label{sub.edge.shifts}

Here we briefly recall the basics of (one-sided) \textit{edge shifts}, which are spaces made of certain infinite sequences of symbols, arising from symbolic dynamics.

\begin{definition}
\label{def.edge.shift}
Let $\Gamma = (V,E,\lambda)$ be a finite graph.
\begin{itemize}
    \item The \textbf{edge shift} $\boldsymbol{\Omega(\Gamma)}$ is the set of all infinite walks on $\Gamma$, i.e.,
    \[ \Omega(\Gamma) \coloneq \{ e_1 e_2 \dots \mid \lambda_{e_n}(2) = \lambda_{e_{n+1}}(1) \}. \]
    \item For a vertex $s \in V$, the \textbf{initial edge shift} $\boldsymbol{\Omega_s(\Gamma)}$ is the set of infinite walks on $\Gamma$ starting at $s$, i.e.,
    \[ \Omega_s(\Gamma) \coloneq \{ e_1 e_2 \dots \mid \lambda_{e_1}(1) = s \text{ and } \lambda_{e_n}(2) = \lambda_{e_{n+1}}(1) \}. \]
\end{itemize}
We refer to the set $E$ of edges of $\Gamma$ as the \textbf{alphabet} of $\Omega_s(\Gamma)$ and $\Omega(\Gamma)$ and we will often denote it by $\mathbb{A}$.
The \textbf{language} of $\Omega(\Gamma)$, usually denoted by $\boldsymbol{\mathbb{L}}$, is the set of finite walks on $\Gamma$ and the \textbf{language} $\boldsymbol{\mathbb{L}_s}$ of $\Omega_s(\Gamma)$ is the set of those finite walks on $\Gamma$ that start at $s$.
\end{definition}

\begin{example}
\label{ex.full.shift}
Given $n \geq 2$, consider a graph $\Gamma=(V,E,\lambda)$ with $V=\{s\}$ and $E=\{e_1 \dots e_n\}$ (thus $\lambda_e(1)=\lambda_e(2)=s$ for all $e \in E$).
The \textbf{full shift} on $n$ elements is the edge shift on $\Gamma$ starting at $s$, and we denote it by $\boldsymbol{E^\omega}$.
Observe that it simply corresponds to the set
\[ E^\omega = \{ e_1 e_2 \dots \mid e_i \in E \} \]
of all infinite words in the alphabet $E$.
\end{example}

\begin{example}
\label{ex.bad.shift}
Consider the graph $\Gamma$ depicted in \cref{fig.bad.shift.graph}.
The edge shift $\Omega(\Gamma)$ is the set of all infinite words in the alphabet $\{a,b,c,d\}$ such that each $a$, $b$ and $d$ is followed by either $c$ or $d$ and each $c$ is followed by either $a$ or $b$.
The initial shift $\Omega_{\text{\textcolor{red}{red}}}(\Gamma)$ is the subset of $\Omega(\Gamma)$ consisting of those words that start either $c$ or $d$.
\end{example}

\begin{figure}
\centering
\begin{tikzpicture}
    \node[vertex] (b) at (-1.333,0) {\textcolor{blue}{blue}};
    \node[vertex] (r) at (1.333,0) {\textcolor{red}{red}};
    \draw[edge] (b) to[out=45,in=135,looseness=.9] node[above]{$a$} (r);
    \draw[edge] (b) to[out=0,in=180] node[above]{$b$} (r);
    \draw[edge] (r) to[out=225,in=-45,looseness=.9] node[above]{$c$} (b);
    \draw[edge] (r) to[out=30,in=330,looseness=7.5] node[right]{$d$} (r);
\end{tikzpicture}
\caption{A graph for the edge shift described in \cref{ex.bad.shift}.}
\label{fig.bad.shift.graph}
\end{figure}
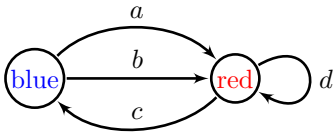

$E^\omega$ is the product of countably many copies of $E$.
When endowed with the product topology, the full shift $E^\omega$ is famously a \textbf{Cantor space} (i.e., a perfect, compact, totally disconnected, and metrizable space).
In general, edge shifts are compact, metrizable and totally disconnected, but they may not be perfect.

\begin{definition}
\label{def.cones}
Fix an edge shift.
Given an element $x$ of its language, its \textbf{cone} is the subspace $T_x$ of those sequences of the edge shift that have $x$ as a prefix.
\end{definition}

The set of cones of an edge shift is a basis of clopen sets for its topology.
When an edge shift is a Cantor space (i.e., when it does not feature isolated points), cones are Cantor spaces themselves.
Otherwise there are cones which consist of a sole isolated point.

\begin{definition}
\label{def.cone.colors}
If $w = x_1 \dots x_n$ belongs to the language of an edge shift $\Omega(\Gamma)$, we say that its \textbf{color} is $\boldsymbol{c(w)} = \lambda_{x_n}(2)$.
Moreover, the homeomorphism
\[ \boldsymbol{C_w} \colon T_w \to \Omega_{c(w)}(\Gamma),\, x\alpha \mapsto \alpha \]
is called the \textbf{color map} of $w$.
\end{definition}
The presumed ambiguity between this notion of color and the one about hypergraphs will be clear later on.
In practice, for a shift $\Omega_s(\Gamma)$, the possible colors of an element of $\mathbb{L}_s$ correspond to the vertices of $\Gamma$ that are reachable from the vertex $s$.
The color map $C_w$ determines a bijection between the set of infinite sequences starting with $w$ and the space of suffixes that may follow $w$.

\begin{definition}
\label{def.canonical.homeomorphism}
Assume that two elements $x$ and $y$ of the language of $\Omega_s(\Gamma)$ have the same color.
Then the homeomorphism
\[ \chi_{x,y} = C_y^{-1} \circ C_x \colon T_x \to T_y,\, x\alpha \mapsto y\alpha \]
is referred to as the \textbf{canonical homeomorphism} between $T_x$ and $T_y$.
\end{definition}


\section{Limit Spaces of Hyperedge Replacement Systems}
\label{sec.replacement.systems}

In this section we define our class of rewriting systems, the \textit{hyperedge replacement systems} (\cref{sub.replacement.systems}).
Every such object determines an edge shift called \textit{symbol space} (\cref{sub.symbol.space}) and a \textit{gluing relation} (\cref{sub.gluing.relation}) on it.
The gluing relation is an equivalence relation when the replacement system satisfies certain natural conditions, in which case we say that it is \textit{expanding} and the quotient of the symbol space on the gluing relation is called \textit{limit space} (\cref{sub.limit.space}).

\subsection{Hyperedge Replacement systems}
\label{sub.replacement.systems}

\begin{definition}
A \textbf{replacement rule of order $\boldsymbol{d}$ and color $\boldsymbol{k}$} is a rewriting rule $e \xrightarrow{k,d} R$, where $e$ is a $k$-colored $d$-edge with pairwise distinct boundary vertices and $R$ is a hypergraph equipped with $d$ distinct vertices $\lambda_R(1), \dots, \lambda_R(d)$, which are called \textbf{boundary vertices} of $R$.
The hypergraph $R$ is called a \textbf{$\boldsymbol{k}$-colored replacement hypergraph of order $\boldsymbol{d}$}.
\end{definition}

Given a replacement rule $e \xrightarrow{k,d} R$ of order $d$ and color $k$, any $k$-colored $d$-edge $e$ of a hypergraph $\Gamma$ can be \textbf{expanded} by replacing it with the hypergraph $R$, identifying the boundary vertices $\lambda_e(i)$ of $e$ with the boundary vertices $\lambda_R(i)$ of $R$.
Performing such a replacement rule on $\Gamma$ is called an \textbf{hyperedge expansion} of $\Gamma$ and we denote the resulting hypergraph by $\Gamma \triangleleft e$.

\begin{definition}
\label{def.consistent.replacement.rules}
Given a fixed set of colors $C$, a set $\R$ of replacement rules is \textbf{consistent} if the following conditions are satisfied.
\begin{enumerate}
    \item For every $k \in C$, among all of the replacement hypergraphs $R$ of the replacement rules of $\R$, each $k$-colored hyperedge has the same order $d_k$.
    \item For every $k \in C$, there is exactly one replacement rule $e_k \xrightarrow{k,d_k} R_k$, where $d_k$ is as above.
\end{enumerate}
\end{definition}

In practice, a consistent set of replacement rules allows to replace any $k$-colored $d_k$-edge in a unique way.
This prompts the following definition.

\begin{definition}
\label{def.replacement.system}
Given a fixed set of colors $C$, a \textbf{replacement system} is a consistent set of replacement rules $\R$ together with a hypergraph $\Gamma_0$, called the \textbf{base hypergraph}, that has the same set of colors $C$ and such that every $k$-colored hyperedge of $\Gamma$ has order $d_k$, where $d_k$ is from condition (1) of \cref{def.consistent.replacement.rules}.
\end{definition}

\begin{example}
\label{ex.replacement.systems}
We present four main examples that will be useful throughout this section to highlight a few different properties of replacement systems.

\cref{fig.airplane.replacement.system} depicts the airplane replacement system $\mathcal{A}$, which features two colors and only employs graphs instead of more general hypergraphs.

\cref{fig.houghton.replacement.system} depicts the Houghton replacement systems $\mathcal{H}_n$,
which will be example of almost expanding but not expanding replacement systems.

\cref{fig.sierpinski.triangle.replacement.system} depicts the Sierpinski triangle replacement system $\mathcal{ST}$, which employs hypergraphs instead.

\cref{fig.apollonian.gasket.replacement.system} dpictes the Apollonian gasket replacement system $\mathcal{AG}$, which has the same replacement hypergraph as $\mathcal{ST}$ but yields very different groups.
\end{example}

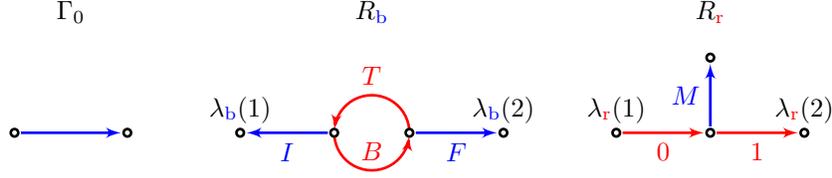
\begin{figure}
\centering
\begin{tikzpicture}
    \node at (0,1.6) {$\Gamma_0$};
    \node[vertex] (l) at (-.75,0) {};
    \node[vertex] (r) at (.75,0) {};
    \draw[edge,blue] (l) to node[above]{} (r);
    \begin{scope}[xshift=4cm]
    \draw[edge,red,domain=5:175] plot ({.5*cos(\x)}, {.5*sin(\x)});
    \draw (90:.5) node[above,red] {$T$};
    \draw[edge,red,domain=185:355] plot ({.5*cos(\x)}, {.5*sin(\x)});
    \draw (270:.5) node[above,red] {$B$};
    \node at (0,1.6) {$R_{\text{\textcolor{blue}{b}}}$};
    \node[vertex] (l) at (-1.75,0) {}; \draw (-1.75,0) node[above]{$\lambda_{\text{\textcolor{blue}{b}}}(1)$};
    \node[vertex] (cl) at (-.5,0) {};
    \node[vertex] (cr) at (.5,0) {};
    \node[vertex] (r) at (1.75,0) {}; \draw (1.75,0) node[above]{$\lambda_{\text{\textcolor{blue}{b}}}(2)$};
    \draw[edge,blue] (cl) to node[below]{$I$} (l);
    \draw[edge,blue] (cr) to node[below]{$F$} (r);
    \end{scope}
    \begin{scope}[xshift=8.5cm]
    \node at (0,1.6) {$R_{\text{\textcolor{red}{r}}}$};
    \node[vertex] (l) at (-1.25,0) {}; \draw (-1.25,0) node[above]{$\lambda_{\text{\textcolor{red}{r}}}(1)$};
    \node[vertex] (r) at (1.25,0) {}; \draw (1.25,0) node[above]{$\lambda_{\text{\textcolor{red}{r}}}(2)$};
    \node[vertex] (c) at (0,0) {};
    \node[vertex] (ct) at (0,1) {};
    \draw[edge,red] (l) to node[below]{$0$} (c);
    \draw[edge,red] (c) to node[below]{$1$} (r);
    \draw[edge,blue] (c) to node[left]{$M$} (ct);
    \end{scope}
\end{tikzpicture}
\caption{The airplane replacement system $\mathcal{A}$.}
\label{fig.airplane.replacement.system}
\end{figure}

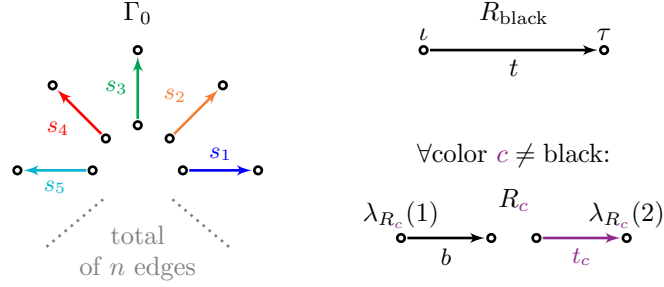
\begin{figure}
\centering
\begin{tikzpicture}
    \node at (0,2.1) {$\Gamma_0$};
    \node[vertex] (1i) at (0:.6) {};
    \node[vertex] (1o) at (0:1.6) {};
    \node[vertex] (2i) at (45:.6) {};
    \node[vertex] (2o) at (45:1.6) {};
    \node[vertex] (3i) at (90:.6) {};
    \node[vertex] (3o) at (90:1.6) {};
    \node[vertex] (4i) at (135:.6) {};
    \node[vertex] (4o) at (135:1.6) {};
    \node[vertex] (5i) at (180:.6) {};
    \node[vertex] (5o) at (180:1.6) {};
    \draw[edge,blue] (1i) to node[above]{\small$s_1$} (1o);
    \draw[edge,Orange] (2i) to node[above left]{\small$s_2$} (2o);
    \draw[edge,Green] (3i) to node[left]{\small$s_3$} (3o);
    \draw[edge,red] (4i) to node[below left]{\small$s_4$} (4o);
    \draw[edge,Turquoise] (5i) to node[below]{\small$s_5$} (5o);
    \draw[dotted,gray] (-40:.6) to (-40:1.6);
    \draw[dotted,gray] (220:.6) to (220:1.6);
    \node[gray,align=center] at (270:1.1) {total\\of $n$ edges};
    \begin{scope}[xshift=5cm,yshift=1.6cm]
    \node at (0,.5) {$R_{\text{black}}$};
    \node[vertex] (bi) at (0-1.2,0) {};
    \node[vertex] (bt) at (1.2,0) {};
    \draw (bi) node[above]{$\iota$};
    \draw (bt) node[above]{$\tau$};
    \draw[edge,black] (bi) to node[below]{$t$} (bt);
    \end{scope}
    \begin{scope}[xshift=5cm,yshift=-.9cm]
    \node at (0,1.1) {$\forall \text{color } \textcolor{Plum}{c} \neq \text{black}$:};
    \node at (0,.5) {$R_{\textcolor{Plum}{c}}$};
    \node[vertex] (ci) at (-1.5,0) {};
    \node[vertex] (c1) at (-.3,0) {};
    \node[vertex] (c2) at (.3,0) {};
    \node[vertex] (ct) at (1.5,0) {};
    \draw (ci) node[above]{$\lambda_{R_{\textcolor{Plum}{c}}}(1)$};
    \draw (ct) node[above]{$\lambda_{R_{\textcolor{Plum}{c}}}(2)$};
    \draw[edge,black] (ci) to node[below]{\small$b$} (c1);
    \draw[edge,Plum] (c2) to node[below]{\small{$t_c$}} (ct);
    \end{scope}
\end{tikzpicture}
\caption{The Houghton replacement systems $\mathcal{H}_n$.}
\label{fig.houghton.replacement.system}
\end{figure}

\begin{figure}
\centering
\begin{subfigure}[t]{.45\textwidth}\centering
$\Gamma_0$\\
\begin{tikzpicture}
    \useasboundingbox (-2,-1.3) rectangle (2,2.1);
    %
    %
    \node[vertex] (t) at (90:1.8) {}; \draw (t) node[left]{$t$};
    \node[vertex] (l) at (210:1.8) {}; \draw (l) node[left]{$l$};
    \node[vertex] (r) at (330:1.8) {}; \draw (r) node[right]{$r$};
    \fill[black,opacity=.25] (t.center) -- (l.center) -- (r.center) -- cycle; \node at (0,0) {$X$};
    \draw (t) to (l);
    \draw (l) to (r);
    \draw (r) to (t);
\end{tikzpicture}\\
$X = (l,r,t)$\\
\vspace{2\baselineskip}
\end{subfigure}
\begin{subfigure}[t]{.45\textwidth}\centering
$R$\\
\begin{tikzpicture}
    \useasboundingbox (-2,-1.3) rectangle (2,2.1);
    %
    %
    \node[vertex] (B3) at (90:1.8) {}; \draw (B3) node[left]{$\lambda_R(3)$};
    \node[vertex] (B1) at (210:1.8) {}; \draw (B1) node[left]{$\lambda_R(1)$};
    \node[vertex] (B2) at (330:1.8) {}; \draw (B2) node[right]{$\lambda_R(2)$};
    \node[vertex] (y) at ($(B3)!.5!(B1)$) {}; \draw (y) node[left]{$y$};
    \node[vertex] (z) at ($(B2)!.5!(B1)$) {}; \draw (z) node[below]{$z$};
    \node[vertex] (x) at ($(B2)!.5!(B3)$) {}; \draw (x) node[right]{$x$};
    \fill[black,opacity=.25] (B1.center) -- (z.center) -- (y.center) -- cycle; \node at (210:.9) {$1$};
    \draw (B1) to (z);
    \draw (z) to (y);
    \draw (y) to (B1);
    \fill[black,opacity=.25] (z.center) -- (B2.center) -- (x.center) -- cycle; \node at (330:.9) {$2$};
    \draw (z) to (B2);
    \draw (B2) to (x);
    \draw (x) to (z);
    \fill[black,opacity=.25] (y.center) -- (x.center) -- (B3.center) -- cycle; \node at (90:.9) {$3$};
    \draw (y) to (x);
    \draw (x) to (B3);
    \draw (B3) to (y);
\end{tikzpicture}\\
$1 = (\lambda_R(1),z,y)$\\
$2 = (z,\lambda_R(2),x)$\\
$3 = (y,x,\lambda_R(3))$
\end{subfigure}
\caption{The Sierpinski triangle replacement system $\mathcal{ST}$.}
\label{fig.sierpinski.triangle.replacement.system}
\end{figure}
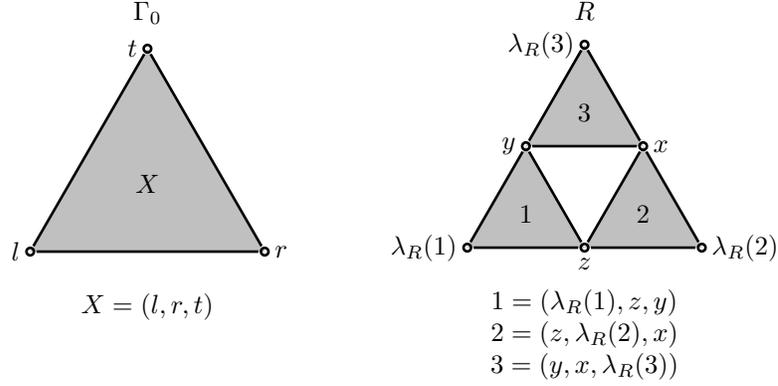

\begin{figure}
\centering
\begin{subfigure}[t]{.45\textwidth}\centering
$\Gamma_0$\\
\vspace{.3\baselineskip}
\begin{tikzpicture}
    %
    \coordinate (Ct) at (0,1.5) {};
    \coordinate (Cc) at (0,0) {};
    \coordinate (Cb) at (0,-1.5) {};
    \draw[fill=black,fill opacity=.25] (Cc) circle (1.5);
    \draw[fill=white] ($(Ct)!.5!(Cc)$) circle (1.5/2);
    \draw[fill=white] ($(Cc)!.5!(Cb)$) circle (1.5/2);
    \node at (180:1) {$L$};
    \node at (0:1) {$R$};
    \node[vertex] (Vt) at (Ct) {};
    \node[below] at (Ct) {$t$};
    \node[vertex] (Vc) at (Cc) {};
    \node[below] at (Cc) {$c$};
    \node[vertex] (Vb) at (Cb) {};
    \node[above] at (Cb) {$b$};
\end{tikzpicture}\\
$L = (t,b,c)$\\
$R = (t,c,b)$\\
\end{subfigure}
\begin{subfigure}[t]{.45\textwidth}\centering
$R$\\
\begin{tikzpicture}
    %
    \node[vertex] (B3) at (90:1.8) {}; \draw (B3) node[left]{$\lambda_R(3)$};
    \node[vertex] (B1) at (210:1.8) {}; \draw (B1) node[left]{$\lambda_R(1)$};
    \node[vertex] (B2) at (330:1.8) {}; \draw (B2) node[right]{$\lambda_R(2)$};
    \node[vertex] (y) at ($(B3)!.5!(B1)$) {}; \draw (y) node[left]{$y$};
    \node[vertex] (z) at ($(B2)!.5!(B1)$) {}; \draw (z) node[below]{$z$};
    \node[vertex] (x) at ($(B2)!.5!(B3)$) {}; \draw (x) node[right]{$x$};
    \fill[black,opacity=.25] (B1.center) -- (z.center) -- (y.center) -- cycle; \node at (210:.9) {$1$};
    \draw (B1) to (z);
    \draw (z) to (y);
    \draw (y) to (B1);
    \fill[black,opacity=.25] (z.center) -- (B2.center) -- (x.center) -- cycle; \node at (330:.9) {$2$};
    \draw (z) to (B2);
    \draw (B2) to (x);
    \draw (x) to (z);
    \fill[black,opacity=.25] (y.center) -- (x.center) -- (B3.center) -- cycle; \node at (90:.9) {$3$};
    \draw (y) to (x);
    \draw (x) to (B3);
    \draw (B3) to (y);
\end{tikzpicture}\\
$1 = (\lambda_R(1),z,y)$\\
$2 = (z,\lambda_R(2),x)$\\
$3 = (y,x,\lambda_R(3))$
\end{subfigure}
\caption{The Apollonian gasket replacement system $\mathcal{AG}$.}
\label{fig.apollonian.gasket.replacement.system}
\end{figure}
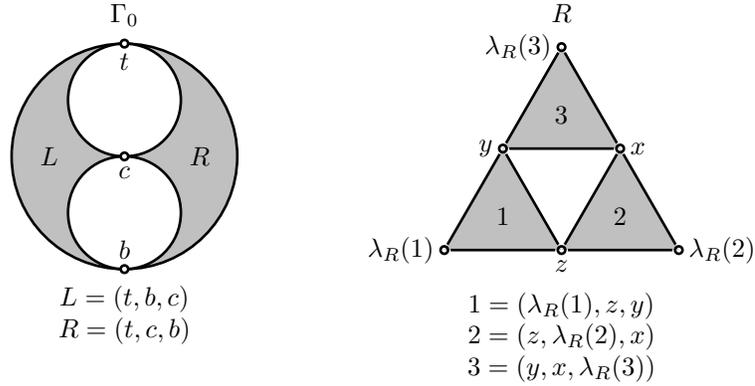

\subsection{The Symbol Space}
\label{sub.symbol.space}

Henceforth, fix a replacement system $\R = (\{ R_k \}_{k \in C}, \Gamma_0)$.

\begin{definition}
\label{def.hypergraph.expansion}
A \textbf{hypergraph expansion} of $(\R,\Gamma_0)$ is a hypergraph $\Gamma$ that is obtained via a sequence of hyperedge expansions of the base hypergraph $\Gamma_0$, i.e.,
\[ \Gamma = \Gamma_0 \triangleleft e_0 \cdots \triangleleft e_m. \]
\end{definition}

As examples, \cref{fig.airplane.expansion,,fig.apollonian.gasket.expansion} show (hyper)graph expansions of the airplane and Apollonian gasket replacement systems $\mathcal{A}$ and $\mathcal{AG}$, respectively.

\begin{figure}
\centering
\begin{tikzpicture}[scale=1.6667,font=\small]
    \draw[edge,red,domain=5:85] plot ({.5*cos(\x)}, {.5*sin(\x)});
    \draw[edge,red,domain=95:175] plot ({.5*cos(\x)}, {.5*sin(\x)});
    \draw[edge,red,domain=185:355] plot ({.5*cos(\x)}, {.5*sin(\x)});
    \draw[edge,red,domain=5:175] plot ({.225*cos(\x)-1.125}, {.225*sin(\x)});
    \draw[edge,red,domain=185:265] plot ({.225*cos(\x)-1.125}, {.225*sin(\x)});
    \draw[edge,red,domain=275:355] plot ({.225*cos(\x)-1.125}, {.225*sin(\x)});
    \draw (45:.5) node[above right,red] {$T0$};
    \draw (135:.5) node[above left,red] {$T1$};
    \draw (270:.5) node[above,red] {$B$};
    \draw (-1.125,.225) node[above,red] {$IB$};
    \draw (-1.5,-.125) node[below,red] {$IT0$};
    \draw (-.75,-.125) node[below,red] {$IT1$};
    \node[vertex] (l) at (-1.75,0) {};
    \node[vertex] (ll) at (-1.35,0) {};
    \node[vertex] (lc) at (-1.125,-.225) {};
    \node[vertex] (lb) at (-1.125,-.55) {};
    \node[vertex] (lr) at (-.875,0) {};
    \node[vertex] (cl) at (-.5,0) {};
    \node[vertex] (cc) at (0,.5) {};
    \node[vertex] (ct) at (0,1.25) {};
    \node[vertex] (cr) at (.5,0) {};
    \node[vertex] (r) at (1.75,0) {};
    \draw[edge,blue] (cr) to node[below]{$F$} (r);
    \draw[edge,blue] (cc) to node[above right]{$TM$} (ct);
    \draw[edge,blue] (ll) to node[above]{$IF$} (l);
    \draw[edge,blue] (lr) to node[above]{$II$} (cl);
    \draw[edge,blue] (lc) to node[below left]{$ITM$} (lb);
\end{tikzpicture}
\caption{Graph expansion of $\mathcal{A}$ corresponding to $\Gamma_0 \triangleleft I \triangleleft B$.}
\label{fig.airplane.expansion}
\end{figure}
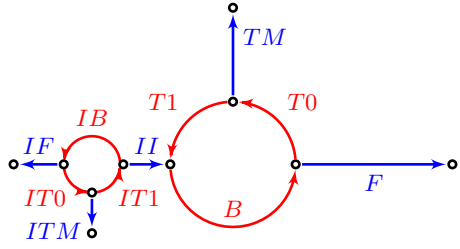

\begin{figure}
\centering
\begin{minipage}[c]{.625\textwidth}
\centering
\begin{tikzpicture}[scale=4/3]
    \coordinate (Ct) at (0,1.5) {};
    \coordinate (Cc) at (0,0) {};
    \coordinate (Cb) at (0,-1.5) {};
    \coordinate (XL) at (-1,0) {};
    \coordinate (Cx) at ($(XL)+(-45:.5)$) {};
    \coordinate (Cy) at ($(XL)+(45:.5)$) {};
    \coordinate (Cz) at ($(XL)+(180:.5)$) {};
    \draw[fill=black,fill opacity=.25] (Cc) circle (1.5);
    \draw[fill=white] ($(Ct)!.5!(Cc)$) circle (1.5/2);
    \draw[fill=white] ($(Cc)!.5!(Cb)$) circle (1.5/2);
    \draw[fill=white] (XL) circle (1.5/3);
    \node at (140:1.8) {$L1$};
    \node at (220:1.8) {$L2$};
    \node (missingcell) at (170:2.5) {$L3$};
    \draw[-stealth,dotted] (missingcell) to[out=0,in=180] (180:.5);
    \node at (0:1) {$R$};
    \node[vertex] (Vt) at (Ct) {};
    \node[above left] at (Ct) {$t$};
    \node[vertex] (Vc) at (Cc) {};
    \node[below] at (Cc) {$c$};
    \node[vertex] (Vb) at (Cb) {};
    \node[below] at (Cb) {$b$};
    \node[vertex] (Vx) at (Cx) {};
    \node[below right] at (Cx) {$x$};
    \node[vertex] (Vy) at (Cy) {};
    \node[above right] at (Cy) {$y$};
    \node[vertex] (Vz) at (Cz) {};
    \node[left] at (Cz) {$z$};
\end{tikzpicture}
\end{minipage}
\begin{minipage}[c]{.325\textwidth}
$L1 = (t,z,y)$\\
$L2 = (z,b,x)$\\
$L3 = (y,x,c)$\\
$R = (t,c,b)$
\end{minipage}
\caption{Hypergraph expansion of $\mathcal{AG}$ corresponding to $\Gamma_0 \triangleleft L$.}
\label{fig.apollonian.gasket.expansion}
\end{figure}

Observe that hyperedges of a hypergraph expansion $\Gamma$ correspond to sequences of elements of $E_{\Gamma_0}$ and $E_{R_k}$ ($k \in C$).
This correspondence is built inductively:
if $e$ is a $k$-colored hyperedge of $\Gamma$, the hyperedges of $\Gamma \triangleleft e$ correspond to
\[ \{ x \in E_\Gamma \mid x \neq e \} \cup \{ e y \mid y \in E_{R_k} \}. \]
In practice, we can only append to $e$ a letter $x$ that belongs $E_{c(e)}$.

Moreover, note that certain hyperedges of distinct hypergraph expansions are naturally identified:
if $x$ and $e$ are distinct hyperedges of $\Gamma$, then $x$ is an hyperedge of both $\Gamma$ and $\Gamma \triangleleft e$.

All of this is formalized by the notion of edge shift that we described in \cref{sub.edge.shifts}:
the set of all hyperedges that appear among all hypergraph expansions $(\R,\Gamma_0)$ is the language of the edge shift $\Omega_{\mathrm{s}}(\Gamma_C)$, where $\Gamma_C$ is defined as follows:
\begin{itemize}
    \item the set of vertices is the set $C$ of colors, with an additional symbol $\mathrm{s}$;
    \item given $k \in C$, there is a $2$-edge from $\mathrm{s}$ to $k$ for each $k$-colored hyperedge of the base hypergraph $\Gamma_0$;
    \item given $k_1, k_2 \in C$, there is a $2$-edge from $k_1$ to $k_2$ for each $k_2$-colored hyperedge of the replacement hypergraph $R_{k_1}$.
\end{itemize}

\begin{example}
The graph $\Gamma_C$ for the airplane replacement system (\cref{fig.airplane.replacement.system}) is depicted in \cref{fig.airplane.shift.graph}.
The word $IBM$ belongs to the language of the shift on $\Gamma_C$, and indeed it is an edge of multiple graph expansions of $\mathcal{A}$ (for example, the one depicted in \cref{fig.airplane.expansion}).
The word $IBF$ instead does not belong to the shift, and indeed there can be no edge associated to such code in any graph expansion of $\mathcal{A}$.
\end{example}

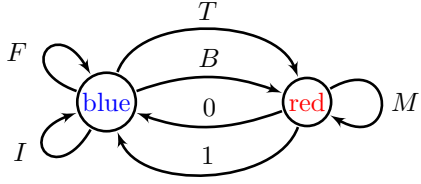
\begin{figure}
\centering
\begin{tikzpicture}
    \node[vertex] (b) at (-1.333,0) {\textcolor{blue}{blue}};
    \node[vertex] (r) at (1.333,0) {\textcolor{red}{red}};
    \draw[edge] (b) to[out=70,in=110,looseness=.9] node[above]{$T$} (r);
    \draw[edge] (b) to[out=20,in=160,looseness=1] node[above]{$B$} (r);
    \draw[edge] (r) to[out=200,in=-20,looseness=1] node[above]{$0$} (b);
    \draw[edge] (r) to[out=250,in=-70,looseness=.9] node[above]{$1$} (b);
    \draw[edge] (b) to[out=240,in=200,looseness=8] node[left=.125cm]{$I$} (b);
    \draw[edge] (b) to[out=160,in=120,looseness=8] node[left=.125cm]{$F$} (b);
    \draw[edge] (r) to[out=30,in=330,looseness=7.5] node[right]{$M$} (r);
\end{tikzpicture}
\caption{The graph $\Gamma_C$ for the airplane replacement system $\mathcal{A}$.}
\label{fig.airplane.shift.graph}
\end{figure}

With \cref{def.edge.shift} in mind, this construction prompts the following definition.

\begin{definition}
\label{def.alphabet.language.symbol.space}
Fix a replacement system $(\R,\Gamma_0)$.
\begin{itemize}
    \item Its \textbf{alphabet} is the set $\boldsymbol{\mathbb{A}_{(\R,\Gamma_0)}}$ of hyperedges of the base and replacement hypergraphs $\Gamma_0$ and $R_k$ ($k \in C$).
    \item Its \textbf{language} is the set $\boldsymbol{\mathbb{L}_{(\R,\Gamma_0)}}$ of all hyperedges that appear among all hypergraph expansions of $(\R,\Gamma_0)$.
    It corresponds to a set of finite words that can be read as finite walks on $\Gamma_C$.
    \item Its \textbf{symbol space} is the set $\boldsymbol{\Omega_{(\R,\Gamma_0)}}$ of those infinite sequences in the alphabet $\mathbb{A}_{(\R,\Gamma_0)}$ whose prefixes belong to $\mathbb{L}_{(\R,\Gamma_0)}$.
    It corresponds to the edge shift $\Omega_{\mathrm{s}}(\Gamma_C)$.
\end{itemize}
We denote them simply by $\mathbb{A}$, $\mathbb{L}$ and $\Omega$ when the replacement system is understood.
If $e = x_1 \dots x_d \in \mathbb{L}_{(\R,\Gamma_0)}$ then $d$ is called the \textbf{depth} of $e$ and is denoted by $\dpt(e)$.
\end{definition}

\subsection{The Gluing Relation}
\label{sub.gluing.relation}

\begin{definition}
\label{def.gluing.relation}
The \textbf{gluing relation} of the replacement system $(\R,\Gamma_0)$ is the binary relation on the symbol space $\Omega$ defined by setting $\alpha \sim \beta$ when every two prefixes of $\alpha$ and $\beta$ of equal length represent hyperedges that share at least one boundary vertex.
\end{definition}

By definition, the gluing relation is always reflexive and symmetric, but may not be transitive in general.
One could consider the transitive closure, but many of the nice properties that we expect (which are described later) are not guaranteed when doing so (see for example \cite[Remark 1.23]{BF19}).
It is instead convenient to focus our attentions on those replacement systems that satisfy the natural condition that we expose below in \cref{def.expanding}, which generalizes \cite[Definition 1.8]{BF19}.

\begin{definition}
\label{def.isolated.colors}
A color $k$ is \textbf{isolated} if all $k$-colored hyperedges in $\mathbb{L}$ (i.e., all hyperedges among all hypergraph expansions) are isolated (\cref{def.incident.adjacent.isolated}).
\end{definition}

For example, all of the colors in the Houghton replacement systems $\mathcal{H}_n$ are isolated.
In most examples, either each hypergraph expansion is connected or all colors are isolated, two exceptions being the replacement systems for the Thompson-like groups $QF$ and $QT$ from \cite{QV} (see \cite{RearrConj} for a construction of these groups with replacement systems).

\begin{definition}
\label{def.trivial.replacement.hypergraph}
A replacement rule $e \xrightarrow{k,d} R$ is \textbf{trivial} if the hypergraph $R$ consists solely of a $k$-colored $d$-edge (which must then have the $d$ boundary vertices of $R$ as its boundary vertices).
We then simply say that the color $k$ itself is \textbf{trivial}.
\end{definition}

The reason why we say that these replacement rules are trivial comes from the fact that performing an expansion based on such a rule on a hypergraph can at most change the order of the boundary vertices of one of its hyperedges and nothing else.
In particular, if the $k$-colored replacement rule of a replacement system is trivial then each $k$-colored hyperedge, when regarded as a finite sequence of symbols in the alphabet $\mathbb{A}_{(\R,\Gamma_0)}$, is the prefix of a unique element of the symbol space.
Thus, words starting with a $k$-colored prefix are forced to a unique infinite sequence.

\begin{definition}
\label{def.expanding}
A replacement system $(\R,\Gamma_0)$ is \textbf{almost expanding} if
\begin{itemize}
    \item for each of its non-isolated colors, in its replacement hypergraph no two boundary vertices are adjacent;
    \item for each of its isolated colors, either it is trivial or in its replacement hypergraph no two boundary vertices are adjacent.
\end{itemize}
An almost expanding replacement system is \textbf{expanding} if no two boundary vertices are adjacent in any replacement hypergraph.
\end{definition}

From our \cref{ex.replacement.systems}, the replacement systems $\mathcal{A}$, $\mathcal{ST}$ and $\mathcal{AG}$ are expanding, whereas each $\mathcal{H}_n$ is only almost expanding.

We will soon show that, when a replacement system is almost expanding, the gluing relation is an equivalence relation, which will allow us to take the quotient $\Omega / \sim$.
In order to see this, we need to understand the behaviour of vertices.

Precisely as we did in \cref{sub.symbol.space} for hyperedges, vertices that appear among the hypergraph expansions of a replacement system $(\R,\Gamma_0)$ also correspond to certain finite words.
Indeed, if we let $\mathbb{V}_{(\R,\Gamma_0)}$ (or simply $\mathbb{V}$) denote the set of vertices of the base and replacement hypergraphs, then each vertex corresponds uniquely to a word of the form $l v$ where $l \in \mathbb{L} \cup \{ \varepsilon \}$ (i.e., $l$ is some hyperedge of a hypergraph expansion or the empty word) and $v$ is a vertex of the $c(l)$-colored replacement hypergraph $R_{c(l)}$ or, when $l = \varepsilon$ is the empty word $\varepsilon$, a vertex of the base hypergraph $\Gamma_0$.

We say that a sequence $\alpha \in \Omega$ \textbf{represents a vertex} $v \in \mathbb{V}$ if all but finitely many prefixes of $\alpha$ are incident on $v$.

\begin{remark}[Proposition 1.22(1) of \cite{BF19}]
Each vertex has at least one representative.
Indeed, recall that by \cref{ass.hypergraphs} our graphs do not have isolated vertices, so there exists some hyperedge $l \in \mathbb{L}$ that is adjacent to $v$.
Moreover, since the boundary vertices of the replacement hypergraphs are not isolated, there are arbitrarily long words starting with $l$ that are adjacent to $v$, as desired.
\end{remark}

It is useful to distinguish a special type of sequences that can occur in almost expanding but not expanding replacement systems.

\begin{definition}
For an almost expanding replacement system, a sequence of $\Omega$ is \textbf{isolated} if it has a prefix whose color is trivial.
\end{definition}

\begin{remark}
\label{rmk.isolated.sequences}
In an almost expanding replacement system (\cref{def.expanding}), a sequence $\alpha = x_1 x_2 \dots \in \Omega$ is isolated if and only if there exists some $m \in \mathbb{N}$ such that each element of $\mathbb{L}$ starting with $x_1 \dots x_m$ is an isolated hyperedge with the same boundary vertices.
In this case, $\alpha$ is the sole sequence starting with $x_1 \dots x_m$.
As for non-isolated sequences, none of its prefixes is colored by a trivial color.

Moreover, note that expanding replacement systems never feature trivial colors, so they cannot have isolated sequences.
\end{remark}

The next two propositions, which generalize \cite[Proposition 1.22]{BF19}, show how the gluing relation can only be non-trivial on points that represent vertices.

\begin{proposition}
\label{prop.gluing.vertices}
Suppose that the replacement system is almost expanding.
\begin{enumerate}
    \item Each non-isolated sequence of $\Omega$ represents at most one vertex.
    \item Each isolated sequence of $\Omega$ represents vertices that are only adjacent between themselves.
    \item Every isolated sequence is not related to any other sequence under the gluing relation.
    \item Two distinct non-isolated sequences $\alpha, \beta \in \Omega$ are related under the gluing relation if and only if they represent the same vertex.
\end{enumerate}
\end{proposition}

\begin{proof}
For (1), suppose that a non-isolated sequence $\alpha$ represents the vertices $v_1$ and $v_2$.
This means that every prefix of $\alpha$ is incident on both $v_1$ and $v_2$.
As noted in \cref{rmk.isolated.sequences}, no prefix of $\alpha$ is colored by a trivial color.
Then, by the definition of almost expanding (\cref{def.expanding}), an expansion of a prefix of $\alpha$ that is incident on both $v_1$ and $v_2$ cannot produce hyperedges that are incident on both vertices, so we must have that $v_1 = v_2$.

For (2), we noted in \cref{rmk.isolated.sequences} that, eventually, prefixes of isolated sequences are hyperedges with the same boundary vertices.
Since such prefixes are colored by an isolated color (by \cref{def.expanding}), such vertices cannot be adjacent to any other vertex.

For (3), assume that $\alpha \sim \beta$ and that $\alpha = x_1 x_2 \dots$ is an isolated sequence.
Then, as noted in \cref{rmk.isolated.sequences}, $\alpha$ has a prefix $x_1 \dots x_m$ colored by a trivial color and every subsequent prefix $x_1 \dots x_{m+i}$ (for every $i \geq 0$) is an isolated hyperedge with the same boundary vertices as $x_1 \dots x_m$.
Hence, since $\alpha \sim \beta$, every prefix $y_1 \dots y_{m+i}$ of $\beta$ shares at least one boundary vertex with $x_1 \dots x_{m+i}$.
But $x_1 \dots x_{m+i}$ is an isolated hyperedge, so it must be the same as $y_1 \dots y_{m+i}$ for all $i \geq 0$.
Thus, we must have $\alpha = \beta$, as needed.

Finally, for (4), let $\alpha$ and $\beta$ be distinct non-isolated sequences.
If $\alpha \sim \beta$, by \cref{def.gluing.relation} we have that each pair of prefixes $x_1 \dots x_m$ and $y_1 \dots y_m$ share at least one boundary vertex $v_m$ for all $m \in \mathbb{N}$.
For the proof in this direction, we only need to show that the sequence $v_m$ eventually stabilizes.
Since $\alpha \neq \beta$, there exists $m \in \mathbb{N}$ such that $x_1 \dots x_m \neq y_1 \dots y_m$.
By \cref{def.expanding}, hyperedges of replacement hypergraphs are adjacent to at most one boundary vertex.
Thus, if $x_1 \dots x_m$ and $y_1 \dots y_m$ share multiple boundary vertices, their successors $x_1 \dots x_{m+1}$ and $y_1 \dots y_{m+1}$ only share one, so we can assume that $x_1 \dots x_m$ and $y_1 \dots y_m$ share a unique boundary vertex.
Then, if $x_1 \dots x_m$ and $y_1 \dots y_m$ share the boundary vertex $v_m$, then their successors $x_1 \dots x_{m+1}$ and $y_1 \dots y_{m+1}$ cannot share any boundary vertex that is not $v_m$ itself, so we are done.

Conversely, assume that two distinct non-isolated sequences $\alpha = x_1 x_2 \dots$ and $\beta = y_1 y_2 \dots$ represent some common vertex $v$.
Assume that $v$ has depth $l$, i.e., $v$ first appears after performing $l$ expansions of the base hypergraph (possibly zero).
Then $x_1 \dots x_l = y_1 \dots y_l$.
We claim that, since $\alpha$ and $\beta$ both represent $v$, further prefixes $x_1 \dots x_m$ and $y_1 \dots y_m$ (where $m > l$) must be adjacent, which implies $\alpha \sim \beta$.
The claim is indeed true since, in general, if $x_1 \dots x_m$ and $y_1 \dots y_m$ are not adjacent, then $x_1 \dots x_{m+1}$ and $y_1 \dots y_{m+1}$ cannot be adjacent either.
\end{proof}

As an easy consequence of \cref{prop.gluing.vertices}, we have that the gluing relations of almost expanding replacement systems are transitive.
Indeed, assume that $\alpha \sim \beta$ and that $\beta \sim \gamma$.
Thanks to point (3) of \cref{prop.gluing.vertices}, there is nothing to prove whenever one of the three sequences is isolated.
When the sequences are non-isolated, point (4) immediately concludes the proof.
We have thus proved the following essential fact.

\begin{proposition}
The gluing relation of an almost expanding replacement system is an equivalence relation.
\end{proposition}

\subsection{The Limit Space}
\label{sub.limit.space}

From here on, we endow the symbol space $\Omega$ with the subspace topology inherited from that of the full shift $\mathbb{A}^\omega$.

\begin{definition}
\label{def.limit.space}
Given an almost expanding replacement system, its \textbf{limit space} is defined as the quotient space $\Omega / \sim$.
For each sequence $\alpha \in \Omega$, the corresponding point of the limit space is denoted by $\boldsymbol{\llbracket \alpha \rrbracket}$.
A representative of a point $p$ of the limit space is called an \textbf{address} for $p$.
\end{definition}

For example, the limit spaces of the airplane, Sierpinski triangle and Apollonian gasket replacement systems (\cref{fig.airplane.replacement.system,fig.sierpinski.triangle.replacement.system,fig.apollonian.gasket.replacement.system}) are depicted in \cref{fig.airplane.limit.space,fig.sierpinski.triangle.and.apollonian.gasket}.
The limit space of the $n$-th Houghton replacement system $\mathcal{H}_n$ (\cref{fig.houghton.replacement.system}) is the disjoint union of $n$ copies of $\{ 1/k \mid k \in \mathbb{N} \} \cup \{0\}$.

\begin{figure}
\centering
\copyrightbox[b]{\includegraphics[width=.425\textwidth]{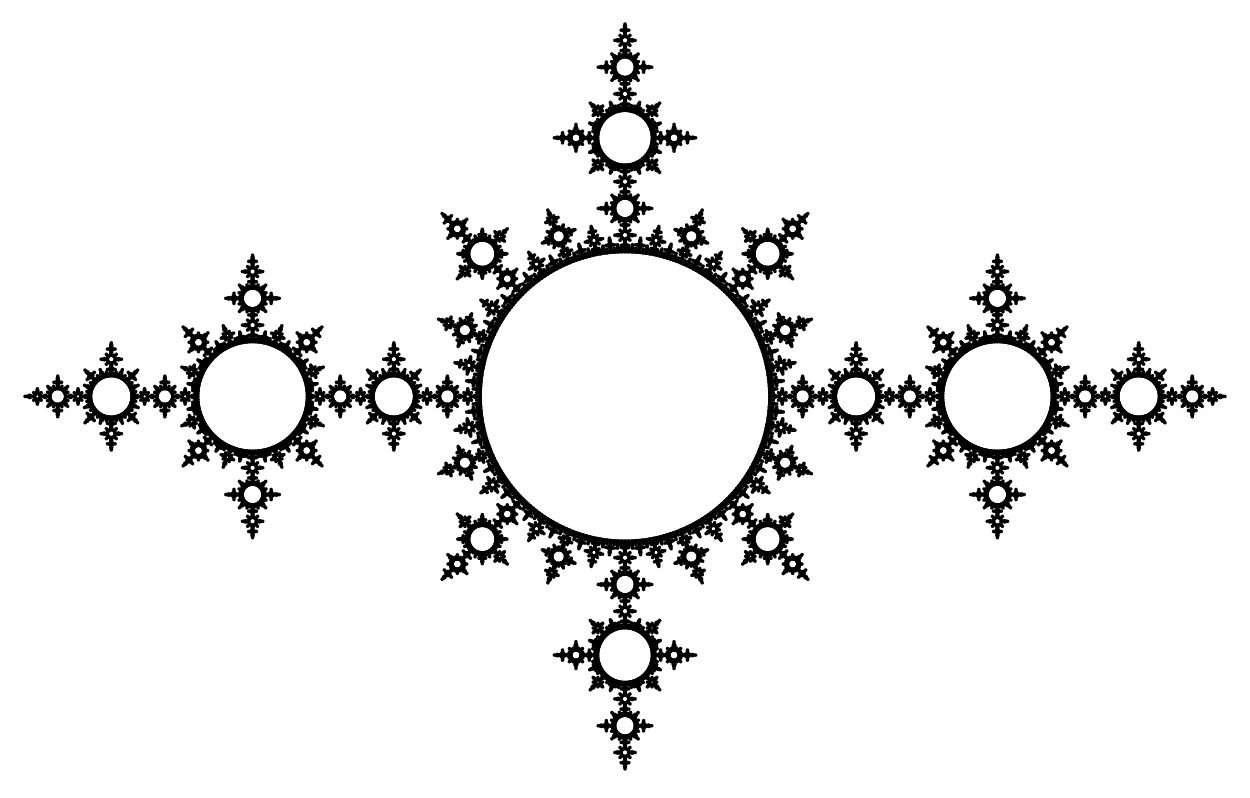}}{Image by James Belk and Bradley Forrest, from \cite{BF19}}
\caption{The limit space of the replacement system $\mathcal{A}$.}
\label{fig.airplane.limit.space}
\end{figure}

\begin{figure}
\centering
\copyrightbox[b]{\includegraphics[width=.3\textwidth]{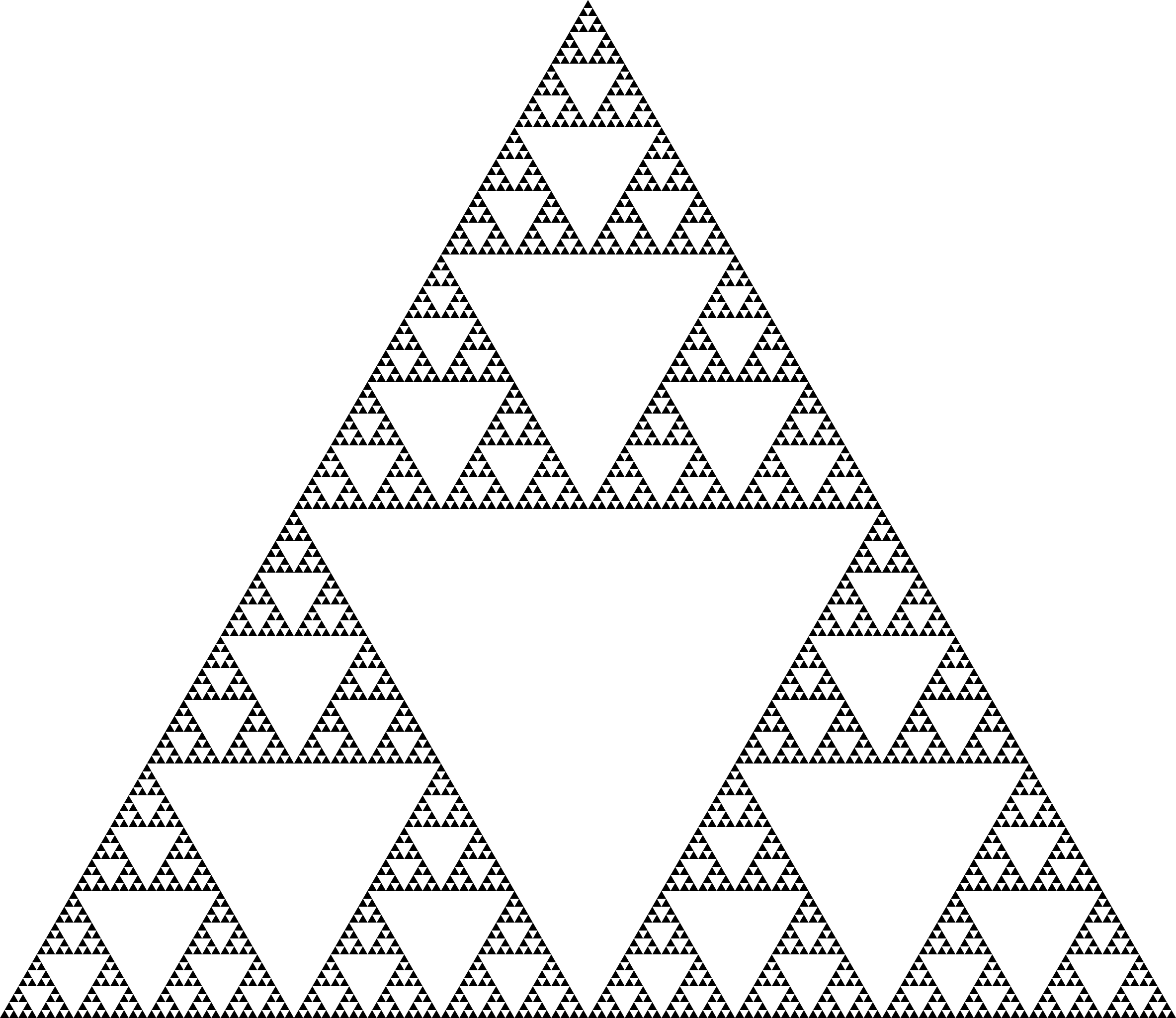}}{Image in public domain, available at \href{https://commons.wikimedia.org/wiki/File:SierpinskiTriangle.svg}{this link}}
\hspace{1cm}
\copyrightbox[b]{\includegraphics[width=.3\textwidth]{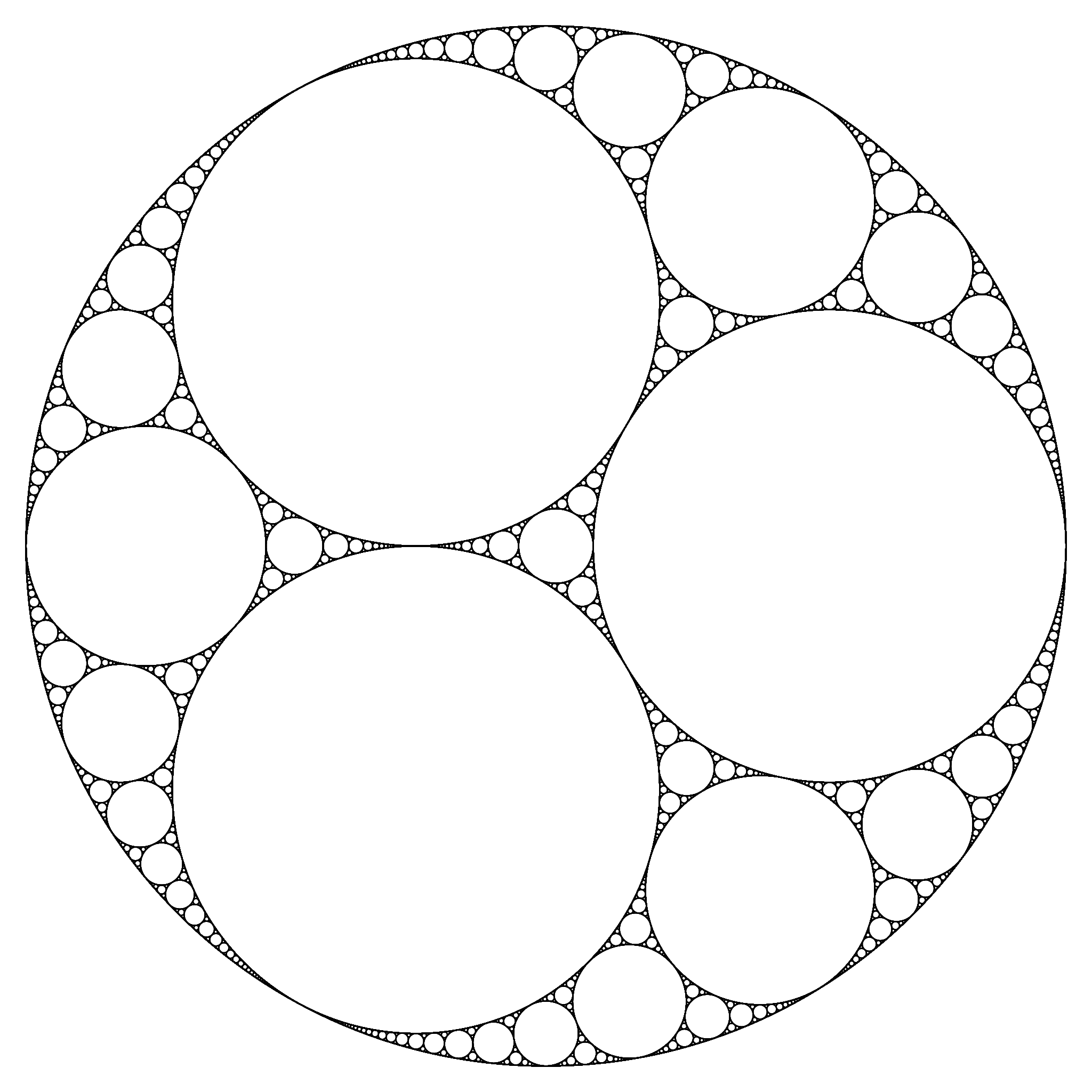}}{Image by Todd Stedl, \href{https://creativecommons.org/licenses/by-sa/4.0}{CC BY-SA 4.0}, available at \href{https://commons.wikimedia.org/wiki/File:ApollonianGasket-15_32_32_33.svg}{this link}}
\caption{The limit spaces of the Sierpinski triangle replacement system $\mathcal{ST}$ and of the Apollonian gasket replacement system $\mathcal{AG}$.}
\label{fig.sierpinski.triangle.and.apollonian.gasket}
\end{figure}

Being the quotient of a compact space, the limit space is compact.
Then, in order to show that the limit space of an almost expanding replacement system is metrizable, it suffices to prove that it is Hausdorff (for example by \cite[Corollary 23.2]{Willard}).
The proof of this is almost identical to that of \cite[Theorem 1.25]{BF19}, which works for limit spaces of expanding $2$-edge replacement systems, so we will not include it here.
The only additional point of the proof concerns isolated points, whose treatment is trivial since they are themselves open.

In order to make sure that the proof of \cite[Theorem 1.25]{BF19} works in our setting, however, we need to show that singletons and cells are closed and that cell interiors are open, which we will do in \cref{lem.topology.limit.space}.
Once these facts are proven, we will have shown the following result.

\begin{theorem}
\label{thm.limit.space.is.compact.metrizable}
The limit space of an almost expanding replacement system is a compact and metrizable topological space.
\end{theorem}

\section{The topology of the limit space}
\label{sec.topology}

Throughout this section, we will assume that the replacement system is almost expanding (\cref{def.expanding}).

\subsection{Cells of the Limit Space}

Each point of the limit space corresponds to some (possibly more than one) sequence $\alpha$, in which we indicate the point by $\llbracket \alpha \rrbracket$.
With a slight abuse of notation, we use the same notation for hyperedges $e \in \mathbb{L}$, as explained right below.

\begin{definition}
\label{def.cell}
A \textbf{cell} of the limit space is a subset of the form
\[ \boldsymbol{\llbracket e \rrbracket} = \{ \llbracket \alpha \rrbracket \mid e \text{ is a prefix of } \alpha \} = \{ \llbracket \alpha \rrbracket \mid \alpha \in T_e \} \]
for some hyperedge $e \in \mathbb{L}$.
\end{definition}

\cref{fig.cells.airplane.sierpinski.triangle.and.apollonian.gasket} portrays examples of cells for replacement systems from \cref{ex.replacement.systems}.

\begin{figure}
\centering
\copyrightbox[b]{\includegraphics[width=.32\textwidth]{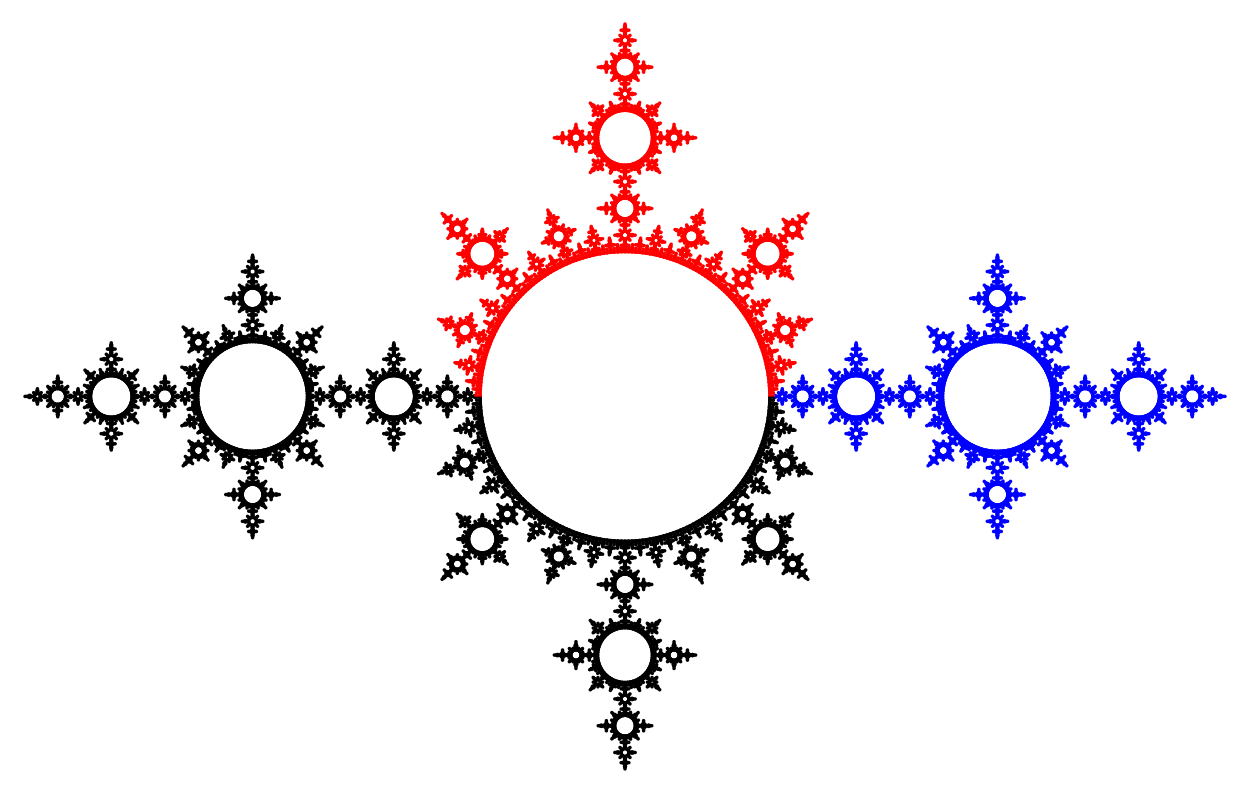}}{Image by James Belk and Bradley Forrest, from \cite{BF19}, modified by the authors}
\hfill
\copyrightbox[b]{\includegraphics[width=.28\textwidth]{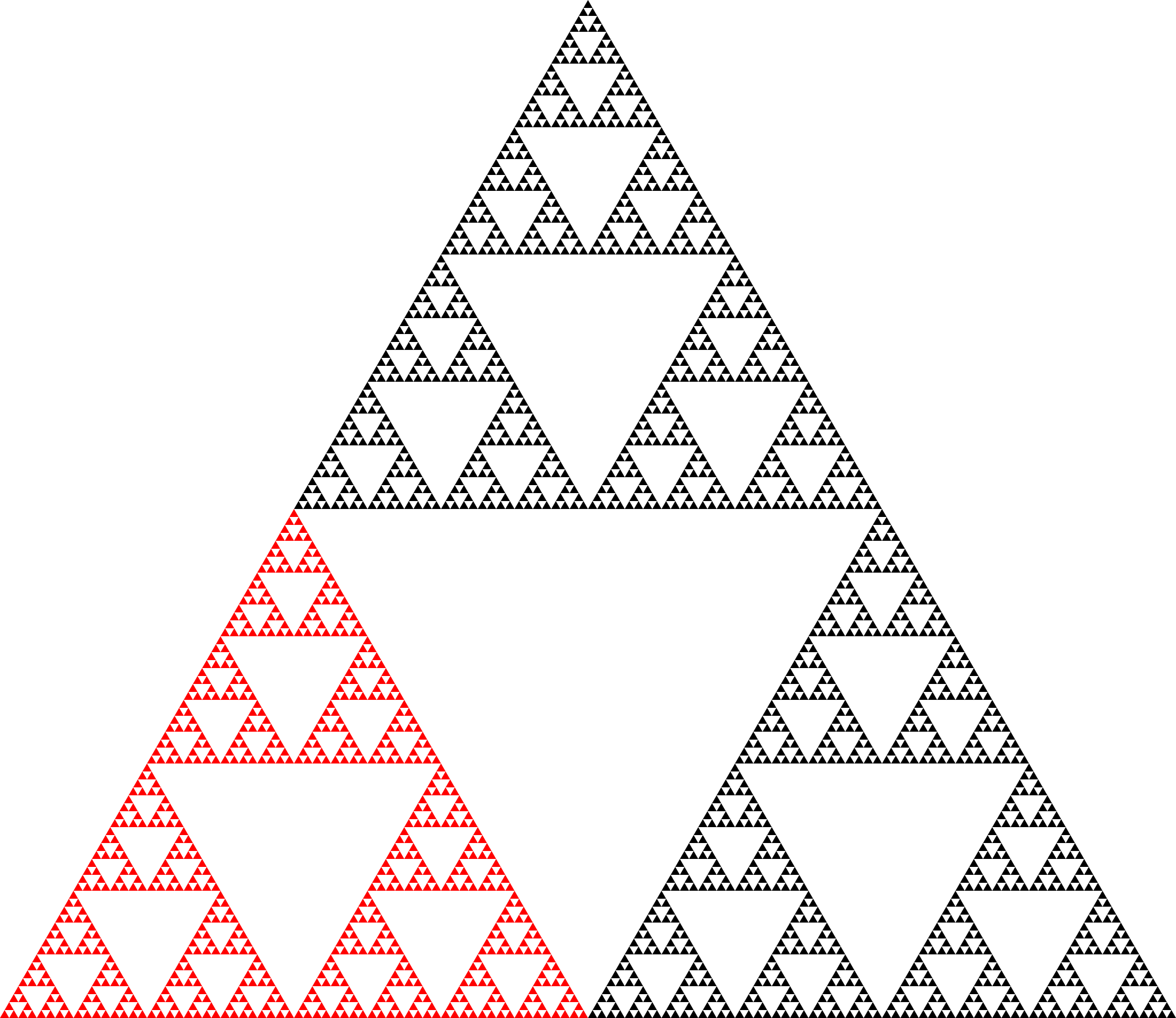}}{Image in public domain, available at \href{https://commons.wikimedia.org/wiki/File:SierpinskiTriangle.svg}{this link}, modified by the authors}
\hfill
\copyrightbox[b]{\includegraphics[width=.28\textwidth]{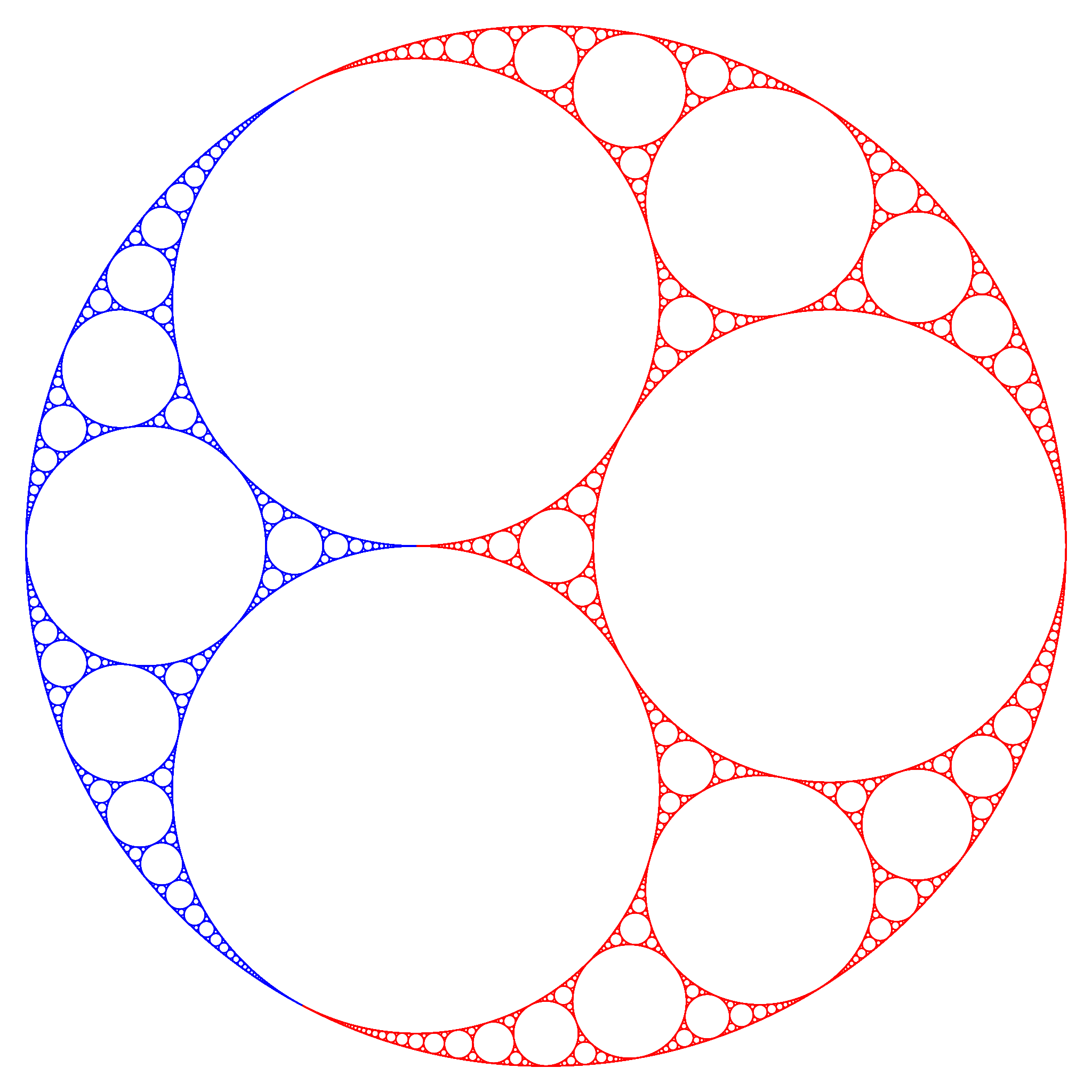}}{Image by Todd Stedl, \href{https://creativecommons.org/licenses/by-sa/4.0}{CC BY-SA 4.0}, available at \href{https://commons.wikimedia.org/wiki/File:ApollonianGasket-15_32_32_33.svg}{this link}, modified by the authors}
\caption{Cells for the replacement systems $\mathcal{A}$, $\mathcal{ST}$ and $\mathcal{AG}$.}
\label{fig.cells.airplane.sierpinski.triangle.and.apollonian.gasket}
\end{figure}

Clearly, if $e$ is a prefix of $f$, then $\llbracket e \rrbracket \supseteq \llbracket f \rrbracket$.
In fact, more can be shown:

\begin{proposition}
\label{prop.cells.contain.points}
Consider hyperedges $e, f \in \mathbb{L}$ such that none of the two is a prefix of the other (equivalently, $e$ and $f$ appear in some common hypergraph expansion).
Then the cells $\llbracket e \rrbracket$ and $\llbracket f \rrbracket$ both include some point that is not included in the other.
\end{proposition}

\begin{proof}
There exist $\alpha \in T_e$ and $\beta \in T_f$, where $T_e$ and $T_f$ are the cones of $e$ and $f$, respectively (recall \cref{def.cones}).
If one of them, say $\alpha$, is isolated, then by \cref{def.trivial.replacement.hypergraph}(3) $\llbracket \alpha \rrbracket$ has $\alpha$ as its unique address, so it cannot belong to $\llbracket f \rrbracket$.
Assume instead $T_e$ and $T_f$ do not feature isolated sequences (i.e., no element of $\mathbb{L}$ starting with $e$ or with $f$ is colored by a trivial color).
Then it is easy to see that, since the replacement system is almost expanding (\cref{def.expanding}), there exist $x_1, x_2, y_1, y_2 \in \mathbb{A}$ for which $e x_1 x_2, f y_1 y_2 \in \mathbb{L}$ are not adjacent.
In particular, sequences starting with $ex$ or $fy$ cannot represent the same vertex, so by \cref{def.trivial.replacement.hypergraph}(4) such sequences correspond to distinct points in the limit space.
\end{proof}

\begin{lemma}
\label{lem.trivial.cells}
A cell $\llbracket e \rrbracket$ is a singleton if and only if $e$ is colored by a trivial color (\cref{def.trivial.replacement.hypergraph}).
In this situation, we say that the cell is \textbf{trivial}.
\end{lemma}

\begin{proof}
If $e$ is colored by a trivial color then there is a unique sequence in $\Omega$ that starts with $e$, so $\llbracket e \rrbracket$ is a singleton.
Since the replacement system is almost expanding (\cref{def.expanding}), the color of $e$ (and thus of each hyperedge that has $e$ as a prefix) is isolated.
Hence, the unique sequence in $\Omega$ starting with $e$ is the sole address of the singleton $\llbracket e \rrbracket$.

Conversely, assume that the color $k$ of the hyperedge $e$ is not trivial.
Then, since the replacement system is almost expanding (\cref{def.expanding}), the replacement hypergraph $R_k$ does not feature hyperedges that are incident on multiple boundary vertices.
In particular, since there are at least two boundary vertices, $R_k$ features at least two distinct hyperedges $x, y \in \mathbb{A}$.
By \cref{prop.cells.contain.points}, $\llbracket ex \rrbracket$ and $\llbracket ey \rrbracket$ feature distinct points, so $\llbracket e \rrbracket$ is not a singleton.
\end{proof}

Each cell has one or more \textbf{boundary points}, which are those points that represent the boundary vertices of $e$.
The removal of the boundary points from a cell $\llbracket e \rrbracket$ produces the \textbf{cell interior}, which we denote by $\boldsymbol{\| e \|}$.
Note that $\| e \|$ is empty if and only if $\llbracket e \rrbracket$ is a trivial cell (in the sense of \cref{lem.trivial.cells}).

In general, the cell interior $\|e\|$ is not the same as the topological interior of the cell $\llbracket e \rrbracket$, which we instead denote by $\boldsymbol{\llparenthesis e \rrparenthesis}$.

\begin{remark}
\label{rmk.cell.and.topological.interior}
For each hyperedge $e \in \mathbb{L}$, the following inclusions always hold
\[ \| e \| \subseteq \llparenthesis e \rrparenthesis \subseteq \llbracket e \rrbracket \]
and $\llparenthesis e \rrparenthesis$ features some (possibly all, possibly none) of the boundary points of $\llbracket e \rrbracket$.

More precisely, by definition the boundary points of a cell $\llbracket e \rrbracket$ never belong to the cell interior $\| e \|$.
Instead, a boundary point $p$ belongs to the topological interior $\llparenthesis e \rrparenthesis$ of the cell $\llbracket e \rrbracket$ if and only if each address of $p$ has $e$ as a prefix, i.e., if $p$ is not ``glued'' to any cell contained in the complement of $\|e\|$.
For example, consider the cell $\llbracket X 1 \rrbracket$ of the Sierpinski triangle, which is depicted in \cref{fig.cells.airplane.sierpinski.triangle.and.apollonian.gasket}:
its topological interior $\llparenthesis X 1 \rrparenthesis$ includes the leftmost vertex of the triangle, while the cell interior $\| X 1 \|$ does not.
This can be easily proved using the upcoming \cref{rmk.stars.are.basis}.
\end{remark}

\subsection{Points of the Limit Space}

In this subsection, we classify points of limit spaces into three types. Let us start with isolated points.

\begin{lemma}
\label{lem.isolated.points}
A point of the limit space is isolated if and only if it has an address that is an isolated sequence, which is then unique by \cref{prop.gluing.vertices}(3).
\end{lemma}

\begin{proof}
Suppose that a sequence $\alpha$ is isolated (\cref{rmk.isolated.sequences}).
Then, by \cref{def.expanding}, $\alpha$ admits a prefix $x_1 \dots x_m$ that is an isolated hyperedge whose color is trivial.
Then the cell $\llbracket x_1 \dots x_m \rrbracket$ is the singleton $\llbracket \alpha \rrbracket$ by \cref{lem.trivial.cells}.
The preimage of $\llbracket \alpha \rrbracket$ is the singleton $\{ \alpha \}$, which is an open subset of the symbol space, so the singleton $\llbracket \alpha \rrbracket$ is isolated.

Conversely, if $\alpha = x_1 x_2 \dots$ is not an isolated sequence, then by \cref{def.expanding} for each $m \in \mathbb{N}$ the color of each prefix $x_1 \dots x_m$ admits at least a vertex $v_m$ that is not a boundary vertex.
Such vertices correspond to distinct points $p_m$ of the limit space that converge to $\llbracket \alpha \rrbracket$, which is thus not an isolated point.
\end{proof}

Thanks to this Lemma and in light of \cref{prop.gluing.vertices}, it makes sense to distinguish the points of a limit space in the three following types.

\begin{definition}
\label{def.singular.isolated.regular.points}
There are three distinct types of points in a limit space:
\begin{itemize}
    \item \textbf{singular point} are points whose addresses represent one vertex;
    \item \textbf{isolated points} are points whose addresses are isolated sequences;
    \item \textbf{regular points} are points whose addresses do not represent any vertex.
\end{itemize}
\end{definition}

\subsection{Topological Properties of the Limit Space}

The following result generalizes Lemma 1.26 of \cite{BF19} to our setting and it is all we need to show \cref{thm.limit.space.is.compact.metrizable}, as explained in \cref{sub.limit.space}.

\begin{lemma}
\label{lem.topology.limit.space}
The limit space of an almost expanding replacement system enjoys the following properties.
\begin{enumerate}
    \item Each singleton is closed.
    \item Each cell is closed.
    \item Each cell interior is open.
\end{enumerate}
\end{lemma}

\begin{proof}
Each singular point represents exactly one vertex and each isolated point represents a finite amount of vertices (the boundary vertices of the trivially-colored isolated hyperedge that represent the point by \cref{lem.isolated.points}).
Exactly as observed in \cite[Lemma 1.26]{BF19}, note that a point $p$ that represents a vertex $v$ has the following preimage under the quotient map:
\[ \bigcap_{i > m} \left( \bigcup \left\{ \llbracket e \rrbracket \mid \dpt(e)=i \text{ and } e \text{ is incident on } v \right\} \right) \subseteq \Omega. \]
Since each cell is closed and each union of the previous equation is finite, this set is closed.
Then, for a singular and isolated point $p$, the singleton $\{p\}$ is the continuous preimage of a closed set, so it is closed.
The preimage of a regular point $p$ is a single element of $\Omega$, so the singleton $\{p\}$ is closed too.

Statement (2) is proved simply by noting that the preimage of $\llbracket e \rrbracket$ under the quotient map is the union of the cone $T_e$ (\cref{def.cones}) with the preimages under the quotient map of the (finitely many) points representing the boundary vertices of $e$.
This is closed because it is the union of finitely many closed set.

A similar reasoning proves statement (3).
The preimage under the quotient map of a cell interior $\| e \|$ is the cone $T_e$ minus the preimages of the points representing the boundary vertices of $e$.
Since cones are open and preimages of singletons are closed, the preimage of $\|e\|$ is open, thus $\|e\|$ itself is open.
\end{proof}

It is easy to check the following useful facts about limit spaces.
\begin{enumerate}
    \item If $e$ is a prefix of $f$ then $\llbracket e \rrbracket \supseteq \llbracket f \rrbracket$ and the converse holds as soon as $\llbracket e \rrbracket$ is not a trivial cell (in the sense of \cref{lem.trivial.cells}).
    \item If neither $e$ nor $f$ is a prefix of the other, then the topological interiors $\llparenthesis e \rrparenthesis$ and $\llparenthesis f \rrparenthesis$ are disjoint.
\end{enumerate}

\begin{remark}
The connectedness of the limit space is inherited by that of the base and replacement hypergraphs.
More precisely, we have the following facts, the first two of which are precisely the same of \cite[Remark 1.27]{BF19}.
\begin{enumerate}
    \item If each replacement hypergraph is connected, then every cell is connected.
    In particular, in this case the limit space has finitely many connected components, the amount of which is precisely the number of connected components of the base hypergraph.
    \item The limit space is totally disconnected if and only if, for every color, its replacement hypergraph $R$ admits an expansion in which the boundary vertices of $R$ all lie in distinct connected components.
    \item Since a point of the limit space is isolated if and only if it descends from an isolated sequence (\cref{lem.isolated.points}), limit spaces of expanding replacement systems do not have isolated points.
\end{enumerate}
\end{remark}

\begin{remark}
\label{rmk.stars.are.basis}
Reading along the lines of the proof of \cite[Theorem 1.25]{BF19} (which promptly extends to a proof of our \cref{thm.limit.space.is.compact.metrizable} as explained in \cref{sub.limit.space}), we find that a basis for the topology of the limit space is given by the set of all $\boldsymbol{d}$-\textbf{star} centered in $p$, which are sets of the following form:
\[ \boldsymbol{S_d (p)} = \bigcup \{ \| e \| \mid \dpt(e) = d \text{ and } p \in \llbracket e \rrbracket \} \cup \{ p \}. \]

With this in mind, one can easily show that a point $p$ belongs to the topological interior $\llparenthesis e \rrparenthesis$ of a cell if and only if each address of $p$ has $e$ as a prefix, which proves \cref{rmk.cell.and.topological.interior}.
\end{remark}

\subsection{Rationality of the Gluing Relation}
\label{sub.rational.gluing}

In \cite{rationalgluing}, the same authors proved that the gluing relation of an expanding edge replacement system from \cite{BF19} (limited to $2$-edges) is \textbf{rational}, i.e., it is recognized by a finite state automaton.
The paper explicitly describes a construction of the automaton based on the replacement system.
Essentially, the automaton works by keeping track of the type of adjacency between equal-length finite prefixes of two sequences.

It is not hard to see that the construction of \cite{rationalgluing} can be extended to work with the expanding \textit{hyperedge} replacement systems introduced here.
To do so, one needs to assign additional colors that keep track of whether an hyperedge $(v_1, \dots, v_n)$ has repetitions among its boundary vertices, i.e., whether there are $i \neq j$ such that $v_i = v_j$ (this generalizes the assignment of special colors to loops explained in \cite[Subsection 1.1.1]{rationalgluing}).
Moreover, it is easy to see that the construction also works with \textit{almost expanding} hyperedge replacement systems without any modification, since there is no hyperedge adjacency to check for isolated hyperedges.

Ultimately, although we will not fully develop the details here, the same argument of \cite{rationalgluing} shows that the following fact.

\begin{theorem}
The gluing relation of an almost expanding replacement system is rational.
\end{theorem}

For example, \cref{fig.sierpinski.triangle.gluing.automaton,} depict the automaton that recognizes the gluing relations of the replacement system $\mathcal{ST}$.

\begin{figure}
\centering
\begin{tikzpicture}[scale=.85,font=\small]
    \node[state] (start) at (-2.5,4) {$q_{-1}$};
    \node[state] (q0) at (1,4) {$q_0$};
    \node[state] (2-3) at (-5,.5) {$q_1 \big(\begin{smallmatrix} 1 & 2 & 3 \\ 0 & 3 & 0 \end{smallmatrix}\big)$};
    \node[state] (3-2) at (-3,-.5) {$q_1 \big(\begin{smallmatrix} 1 & 2 & 3 \\ 0 & 0 & 2 \end{smallmatrix}\big)$};
    \node[state] (2-1) at (-1,.5) {$q_1 \big(\begin{smallmatrix} 1 & 2 & 3 \\ 0 & 1 & 0 \end{smallmatrix}\big)$};
    \node[state] (1-2) at (1,-.5) {$q_1 \big(\begin{smallmatrix} 1 & 2 & 3 \\ 2 & 0 & 0 \end{smallmatrix}\big)$};
    \node[state] (3-1) at (3,.5) {$q_1 \big(\begin{smallmatrix} 1 & 2 & 3 \\ 0 & 0 & 1 \end{smallmatrix}\big)$};
    \node[state] (1-3) at (5,-.5) {$q_1 \big(\begin{smallmatrix} 1 & 2 & 3 \\ 3 & 0 & 0 \end{smallmatrix}\big)$};
    \draw[edge] (-3.5,4) to (start);
    \draw[edge] (start) to node[above]{\small$(X,X)$} (q0);
    \draw[edge] (q0) to[out=-10,in=45,looseness=9] node[right]{\small$(i,i)$} (q0);
    \draw[edge] (q0) to[out=255,in=0] (0,3) to node[above]{\small$(3,2)$} (-3.5,3) to[out=180,in=90] (2-3);
    \draw[edge] (q0) to[out=260,in=0] (0,2.25) to node[above]{\small$(2,3)$} (-2,2.25) to[out=180,in=90] (3-2);
    \draw[edge] (q0) to[out=265,in=0] (0,1.5) to node[above]{\small$(1,2)$} (-.5,1.5) to[out=180,in=90] (2-1);
    \draw[edge] (q0) to node[below right]{\small$(2,1)$} (1-2);
    \draw[edge] (q0) to[out=275,in=90] node[right]{\small$(1,3)$} (3-1);
    \draw[edge] (q0) to[out=280,in=180] (2,3) to node[above]{\small$(3,1)$} (3.5,3) to[out=0,in=90] (1-3);
    \draw[edge] (2-3) to[out=250,in=290,looseness=8] node[below]{\small$(2,3)$} (2-3);
    \draw[edge] (1-2) to[out=250,in=290,looseness=8] node[below]{\small$(1,2)$} (1-2);
    \draw[edge] (1-3) to[out=250,in=290,looseness=8] node[below]{\small$(1,3)$} (1-3);
    \draw[edge] (2-1) to[out=250,in=290,looseness=8] node[below]{\small$(2,1)$} (2-1);
    \draw[edge] (3-1) to[out=250,in=290,looseness=8] node[below]{\small$(3,1)$} (3-1);
    \draw[edge] (3-2) to[out=250,in=290,looseness=8] node[below]{\small$(3,2)$} (3-2);
\end{tikzpicture}
\caption{The gluing automaton for $\mathcal{ST}$.}
\label{fig.sierpinski.triangle.gluing.automaton}
\end{figure}
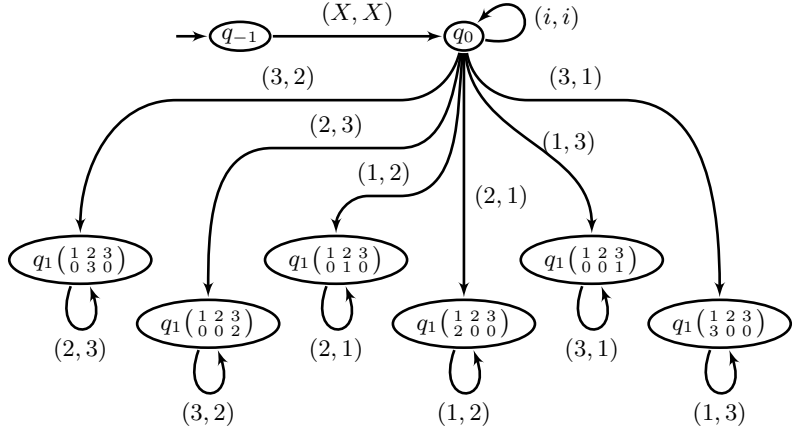

For the generalization from graphs to hypergraphs, the main distinction from \cite{rationalgluing} is that $q_1$-states have the form
\[ q_1 \big(\begin{smallmatrix} c(a) & 1 & \dots & n \\ c(b) & i_1 & \dots & i_n \end{smallmatrix}\big), \]
signifying that the first hyperedge $x_1 \dots x_m a$ is an $n$-edge whose $j$-th boundary vertex coincides with the $i_{j}$-th boundary vertex of the second hyperedge $y_1 \dots y_m b$, for $j = 1, \dots, n$ (in absence of multiple colors, such as in the previous examples, the first column is omitted).

\subsection{Homeomorphisms of the Limit Space}

The fact that limit spaces come equipped with a natural coding through the symbol space encourage us to consider those homeomorphisms of limit spaces that arise from homeomorphisms of the symbol space.
In fact, as shown here in \cref{prop.homeo.symbol.space.acts.on.limit.space}, this defines a faithful action by homeomorphisms of the group of homeomorphisms of the symbol space on the limit space.

\begin{lemma}
\label{lem.cell.contains.isolated.regular.point}
Each cell contains an isolated point or a regular point.
\end{lemma}

\begin{proof}
Consider a cell $\llbracket e\rrbracket$.
If it contains an isolated point there is nothing to prove, so assume that it does not.
Then, by \cref{lem.isolated.points}, sequences of $\Omega$ starting with $e$ cannot be isolated.
Thus, as noted in \cref{rmk.isolated.sequences}, finite words starting with $e$ cannot be colored by trivial colors.
It follows that, since the replacement system is almost expanding (\cref{def.expanding}), the colors appearing among finite words starting with $e$ are such that no two boundary vertices are adjacent in their replacement hypergraphs.
In particular, this implies that $\llbracket e \rrbracket$ contains points whose addresses are sequences that represent no vertices,
which are regular points by their very definition (\cref{def.singular.isolated.regular.points}).
\end{proof}

\begin{proposition}
The set consisting of the regular and the isolated points of a limit space is dense in the limit space.
\end{proposition}

\begin{proof}
Thanks to \cref{lem.cell.contains.isolated.regular.point}, we only need to show that every open subset of a limit space includes some cell.
To see this, let $U$ be open in the limit space.
Since the limit space is equipped with the quotient topology, the preimage $\mathfrak{U}$ of $U$ under the quotient map is open in the symbol space $\Omega$.
Being a subspace of the Cantor space $\mathbb{A}^\omega$, a basis of open sets for the topology of $\Omega$ is given by the collection of \textit{cones} of $\Omega$, which are the sets of the form
$\mathfrak{C}_w = \{ w \alpha \in \Omega \}$
for some finite prefix $w \in \mathbb{L}$.
In particular, $\mathfrak{U}$ includes some cone $\mathfrak{C}_w$.
Then, by the very definition of cells (\cref{def.cell}), clearly $U$ must include the cell $\llbracket w \rrbracket$, so we are done.
\end{proof}

Note that regular and isolated points have trivial equivalence classes under the gluing relation (the converse is untrue, as there may be singular points with trivial equivalence class).
Then, in particular, the previous proposition shows that the set of points that are only glued to themselves is dense in the limit space.
This allows us to show the following simple yet essential fact.

\begin{proposition}
\label{prop.homeo.symbol.space.acts.on.limit.space}
A homeomorphism $\Phi$ of the symbol space $\Omega$ that preserves the gluing relation $\sim$ (i.e., $x \sim y \iff \Phi(x) \sim \Phi(y)$) descends to a homeomorphism $\phi$ of the limit space.
Moreover, if $\Phi$ and $\Psi$ descend to $\phi = \psi$ in this way, then $\Phi = \Psi$, meaning that this correspondence describes a faithful action of the group of gluing-preserving homeomorphisms of $\Omega$ on the limit space.
\end{proposition}

\begin{proof}
Since the limit space $X$ is the quotient of $\Omega$ under the gluing relation $\sim$, the fact that a gluing preserving homeomorphism $\Phi$ of $\Omega$ descends to a homeomorphism of the limit space is a general fact of quotient topological spaces that follows from the universal property of the quotient map.
The fact that the set of points of a limit space that are only glued to themselves is dense in the limit space implies that this action is faithful.
\end{proof}


\section{Eventually Self-similar Groups of Homeomorphisms of Fractals}
\label{sec.ESS}

The groups treated in this paper have a finitary asynchronous action followed by a self-similar one.
This section builds towards the definition of \textit{diagrams} that represent these homeomorphisms.
Such diagrams encode the asynchronous part of the action as a $\pi$-hypergraph isomorphism and the self-similar part as hyperedge labels which need to agree with $\pi$ in order to preserve the gluing relation.

\subsection{Self-similar Groups of Homeomorphisms}

Our goal in this subsection is to describe self-similar isomorphisms between cells.
We will first recall two classical notions of \textit{automorphism} of edge shifts and \textit{self-similarity}, which does not involve the gluing relation and is only based on the symbol space.
We will then define \textit{color automorphisms}, which are the automorphisms of edge shifts that preserve the ``internal'' gluing relation, and \textit{cell isomorphisms}, which are certain maps defined from a cell to another that restrict to color automorphisms.
This will ultimately allow to build self-similar groups of homeomorphisms of cells.

\subsubsection{Automorphisms of Edge Shifts}

As commonly done in the theory of automorphism groups of trees, one can identify $\Omega(\Gamma)$ with the boundary of the vertex-colored (\cref{def.cone.colors}) forest with vertices $\mathbb{L}$ and edges between words $w$ and $wx$, for all $x \in \mathbb{A}$ and $w, wx \in \mathbb{L}$.
Each forest automorphism then induces a homeomorphism of $\Omega(\Gamma)$.
For short, we will refer to such homeomorphisms simply as \textbf{automorphisms} of $\Omega(\Gamma)$ and we will write $\Aut(\Omega)$ for the group of automorphisms of $\Omega$.

As commonly done in the theory of automorphism groups of trees, one can identify $\Omega(\Gamma_c)$ with the boundary of the vertex-colored (\cref{def.cone.colors}) tree with vertices $\mathbb{L}_c$ and edges between words $w$ and $wx$, for all $x \in \mathbb{A}$ and $w, wx \in \mathbb{L}_c$.
Each tree automorphism then induces a homeomorphism of $\Omega(\Gamma_c)$.
For short, we will refer to such homeomorphisms simply as \textbf{automorphisms} of $\Omega(\Gamma_c)$ and we will write $\Aut(\Omega_c)$ for the group of automorphisms of $\Omega_c$.

Each automorphism $\phi$ of $\Omega(\Gamma_c)$ determines and is determined by the permutation of vertices of the tree, which is a self-bijection of $\mathbb{L}_c$ that preserves word-lengths and the prefix relation (i.e., the relation $\leq$ defined by setting $v \leq w$ when $v$ is a prefix of $w$).
We will denote it by $\boldsymbol{\widehat\phi}$.

\subsubsection{Self-similar Groups of Homeomorphisms of Edge Shifts}

\begin{definition}
\label{def.state}
Given $\phi \in \Aut \left(\Omega_s(\Gamma)\right)$ and an element $w=x_1 x_2 \ldots x_n \in \mathbb{L}_s$ of color $c$, the \textbf{state} of $\phi$ at $w$ is the map
\[ \phi|w \colon \Omega_c(\Gamma) \to \Omega_c(\Gamma),\, C_{\hat{\phi}(w)} \circ \phi \circ \left(C_w\right)^{-1}, \]
which is, in other words, the transformation performed by $\phi$ on the suffix of any $\alpha \in T_w$ after its prefix $w$.
\end{definition}

\begin{definition}
\label{def.self.similar.tuple}
Consider an edge shift $\Omega(\Gamma)$.
A \textbf{self-similar tuple} for $\Omega(\Gamma)$ is a tuple $(G_c \mid c \in V_\Gamma)$ of groups such that
\begin{itemize}
    \item $G_c \leq \Aut(\Omega_c(\Gamma))$ for each $c \in V_\Gamma$ and
    \item $\forall c \in V_\Gamma$ and $\forall g \in G_c$ one has $g|w \in G_{c(w)}$ for every $w \in \mathbb{L}$ of color $c$.
\end{itemize}
\end{definition}

\subsubsection{Color and Cell Automorphisms}

Fix an almost expanding replacement system $\R$ (\cref{def.replacement.system,,def.expanding}).
For each color $c$, let $\R_c$ be the replacement system obtained from $\R$ by replacing the base hypergraph with a sole hyperedge colored by $c$ with pairwise distinct boundary vertices.
Its limit space is homeomorphic to every $c$-colored cell of $\R$, possibly except for additional gluings on the boundary vertices.
Denote by $\sim_c$ the gluing relation of $\R_C$, which is an equivalence relation on $\Omega_c(\Gamma)$.

\begin{definition}
\label{def.color.automorphism}
For each color $c$ of $\R$, a $c$-\textbf{automorphism} is an element $\phi \in \Aut(\Omega_c(\Gamma))$ that preserves the gluing relation of $\R_c$, i.e., such that $\alpha \sim_c \beta$ if and only if $\phi(\alpha) \sim_c \phi(\beta)$.
We denote by $\boldsymbol{\Aut(\R_c)}$ the set of such maps.
\end{definition}

\begin{example}
\label{ex.color.automorphism}
For the airplane replacement system $\mathcal{A}$ (\cref{fig.airplane.replacement.system}), the following map $\phi_\mathcal{A}$ is an automorphism for the color blue:
\[ \phi_\mathcal{A}(I \alpha) = F \alpha,\, \phi_\mathcal{A}(B \alpha) = T \alpha,\, \phi_\mathcal{A}(F \alpha) = I \alpha,\, \phi_\mathcal{A}(T \alpha) = B \alpha. \]

For the Sierpinski triangle or Apollonian gasket replacement systems $\mathcal{ST}$ and $\mathcal{AG}$ (\cref{fig.sierpinski.triangle.replacement.system,,fig.apollonian.gasket.replacement.system}), which share the same replacement hypergraph, the following maps $\phi$ and $\rho$ are automorphisms of the sole color of the replacement systems:
\begin{align*}
\phi(1 \alpha) = 2 \phi(\alpha),\, \phi(2 \alpha) = 1 \phi(\alpha),\, \phi(3 \alpha) = 3 \phi(\alpha),
\\
\rho(1 \alpha) = 2 \rho(\alpha),\, \rho(2 \alpha) = 3 \rho(\alpha),\, \rho(3 \alpha) = 1 \rho(\alpha).
\end{align*}
Note that $\phi$ and $\rho$ are a reflection and a rotation of the whole limit space.
\end{example}

It is easy to see that $\Aut(\R_c)$ is a group and its elements are homeomorphisms of the limit space of $\R_c$.
This is immediately shown as an application of \cref{prop.homeo.symbol.space.acts.on.limit.space} to the replacement system $\R_c$.

Let us denote by $\boldsymbol{\partial w}$ the topological boundary of the cell $\llbracket w \rrbracket$, which is $\llbracket w \rrbracket \setminus \llparenthesis w \rrparenthesis$.
Note that, by \cref{rmk.cell.and.topological.interior}, $\partial w$ coincides with the set of those boundary points of $\llbracket w \rrbracket$ that also belong to some cell that does not include or that is not included in $\llbracket w \rrbracket$.
For instance, consider the cell $\llbracket X 1 \rrbracket$ of the replacement system $\mathcal{ST}$, which is depicted in \cref{fig.cells.airplane.sierpinski.triangle.and.apollonian.gasket}:
its boundary $\partial X 1$ includes the two rightmost vertices of the triangle, but not the leftmost one.
Furthermore, note that $\partial w$ is empty if and only if $w$ is an isolated hyperedge, which in particular is the case for isolated colors.

The points of $\partial w$ are thus singular points, which correspond to some (but not necessarily all) boundary vertices of the hyperedge $w$ in any hypergraph expansion featuring $w$.
We will often treat those points as vertices without further notice.

Color automorphisms have no restrictions on how boundary points are mapped.
For example, it is easy to build color automorphisms of cells of the dendrite replacement systems (which will be considered in \cref{sub.dendrites}, see \cref{fig.dendrite.replacement.systems}) that map boundary vertices to other vertices.

\begin{definition}
\label{def.cell.automorphism}
A \textbf{cell isomorphism} between two cells $\llbracket v \rrbracket$ and $\llbracket w \rrbracket$ of the same color is a map
\[ \llbracket v \rrbracket \to \llbracket w \rrbracket,\, \llbracket v \alpha \rrbracket \mapsto \llbracket w \phi(\alpha) \rrbracket, \]
where $\phi \in \Aut(\R_c)$ restricts to a bijection $\partial v \to \partial w$.
We denote by $\Iso(\llbracket v \rrbracket, \llbracket w \rrbracket)$ the set of isomorphisms from $\llbracket v \rrbracket$ to $\llbracket w \rrbracket$.
\end{definition}

It is easily seen that $\Aut(\llbracket w \rrbracket)$ is a group and it naturally embeds into $\Aut(\R_c)$.

\begin{example}
Let us give two examples, both pertaining to our \cref{ex.replacement.systems}.

In the airplane replacement system $\mathcal{A}$, consider the cells $\llbracket F \rrbracket$ and $\llbracket I I \rrbracket$ (both present in \cref{fig.airplane.expansion}).
The boundary $\partial F$ has a sole point, so any cell automorphism of $\llbracket F \rrbracket$ must fix it, and it is then easy to see that $\llbracket F \rrbracket$ has no non-trivial cell automorphisms.
The boundary $\partial II$ instead has two points, corresponding to the boundary points of $\llbracket II \rrbracket$, so the map $\llbracket II\alpha \rrbracket \mapsto \llbracket II\phi(\alpha) \rrbracket$ (wit $\phi_\mathcal{A}$ as described in \cref{ex.color.automorphism}) is a cell automorphism of $\llbracket II \rrbracket$ and it switches the two points of $\partial II$.

In the Sierpinski triangle replacement system $\mathcal{ST}$, consider the cells $\llbracket X 1 \rrbracket$ and $\llbracket X 2 \rrbracket$:
each of their boundaries consists of two points, one of which is in common, say $\partial \llbracket X 1 \rrbracket = \{ y,z \}$ and $\partial \llbracket X 2 \rrbracket = \{ x,z \}$.
Consider the color automorphisms $\phi$ and $\rho$ described in \cref{ex.color.automorphism}:
the maps $\llbracket X 1 \alpha \rrbracket \mapsto \llbracket X 2 \phi(\alpha) \rrbracket$ and $\llbracket X 1 \alpha \rrbracket \mapsto \llbracket X 2 \rho(\alpha) \rrbracket$ are cell isomorphisms $\llbracket X 1 \rrbracket \to \llbracket X 2 \rrbracket$, since the first maps $y$ to $x$ and fixes $z$ and the second maps $y$ to $z$ and $z$ to $x$.
\end{example}

\begin{remark}
\label{rmk.cell.automorphisms}
Consider a cell automorphism $\llbracket w \alpha \rrbracket \mapsto \llbracket w \phi(\alpha) \rrbracket$ and let $\alpha = x \alpha'$.
The first-level states $\phi|x$ induce cell automorphisms $\llbracket w x \alpha' \rrbracket \mapsto \llbracket w \hat{\phi}(x) \phi|x (\alpha') \rrbracket$
that agree on the topological boundaries $\partial wx$ of the subcells $\llbracket wx \rrbracket$.
\end{remark}

\subsubsection{Self-similar Groups of Homeomorphisms of Cells}

\begin{definition}
\label{def.compatible.self.similar.tuple}
A self-similar tuple $(G_c \mid c \in C)$ is \textbf{compatible} with a replacement system $\R$ if $G_c \leq \Aut(\R_c)$ for each color $c$.
\end{definition}

\subsection{Eventually Self-Similar Groups of Homeomorphisms of the Limit Space}

Throughout this section, fix a self-similar tuple $(G_c \mid c \in C)$ that is compatible with a replacement system $\R$.

\subsubsection{Definition of Diagrams}
\label{sub.diagrams}

For the next definition, recall \cref{def.pi.isomorphism,def.cell.automorphism}.

\begin{definition}
\label{def.diagram}
A \textbf{diagram $\boldsymbol{(D,R,\f,l)}$} consists of a $\pi$-hypergraph isomorphism $\f \colon D \to R$ between two hypergraph expansions $D$ and $R$ together with a labeling of the hyperedges of the domain hypergraph $l \colon E_D \to \bigcup_{c \in \mathrm{C}} G_c, e \mapsto l_e$ such that
\begin{enumerate}
    \item $l_e \in G_c$ if and only if $e$ is colored by $c$;
    \item the map $\llbracket e \alpha \rrbracket \mapsto \llbracket \f_E(e) l_e(\alpha) \rrbracket$ is a cell isomorphism $\llbracket e \rrbracket \to \llbracket \f_E(e) \rrbracket$;
    \item for each $e \in E_D$, the cell isomorphism $\llbracket e \alpha \rrbracket \mapsto \llbracket \f_E(e) l_e(\alpha) \rrbracket$ agrees with the map $\lambda_{\f_E(e)} \circ \pi_e \circ \lambda_e^{-1}$ on $\partial e$.
\end{enumerate}
\end{definition}

As examples, \cref{fig.apollonian.gasket.generators} depicts two diagrams based on the Apollonian gasket replacement system, where the self-similar tuple is the group (because there is a unique color) generated by $\phi$ and $\rho$ from \cref{ex.color.automorphism}.

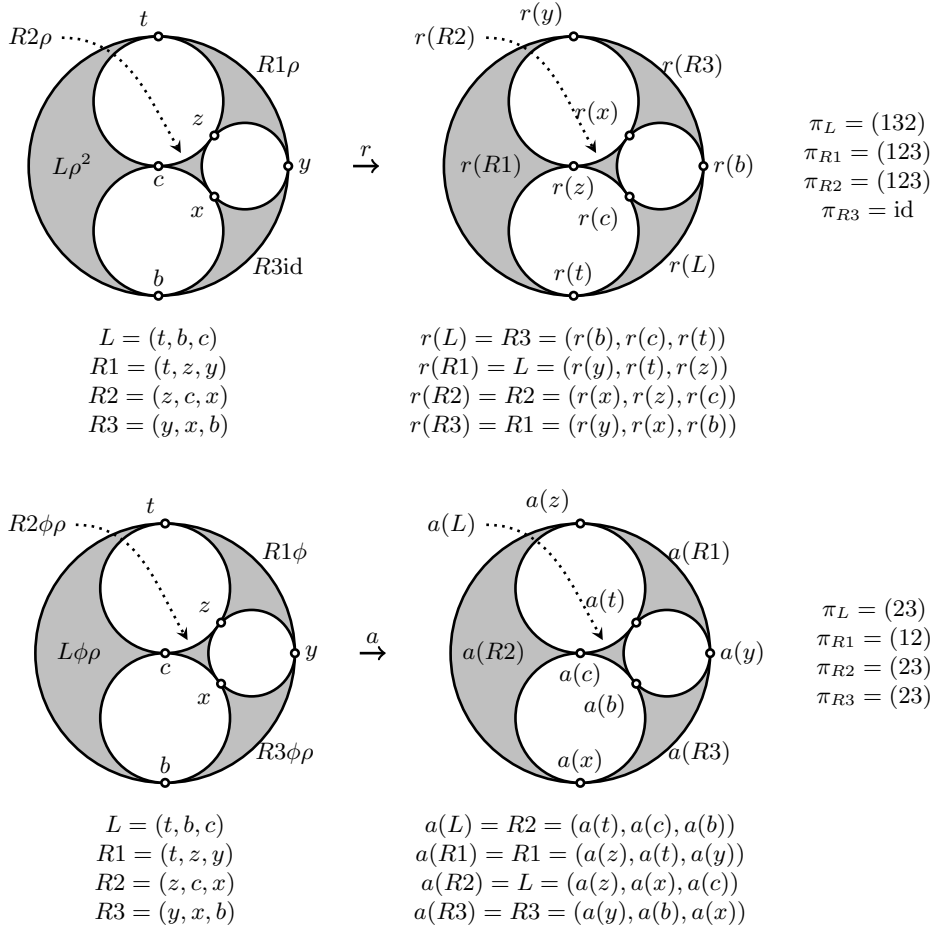
\begin{figure}
\centering
\begin{subfigure}[t]{\textwidth}\centering
\begin{tikzpicture}[scale=1.15,font=\small]
    \begin{scope}[xshift=-2.4cm]
    \coordinate (Ct) at (0,1.5) {};
    \coordinate (Cc) at (0,0) {};
    \coordinate (Cb) at (0,-1.5) {};
    \coordinate (XR) at (1,0) {};
    \coordinate (Cx) at ($(XR)+(225:.5)$) {};
    \coordinate (Cy) at ($(XR)+(0:.5)$) {};
    \coordinate (Cz) at ($(XR)+(135:.5)$) {};
    \draw[fill=black,fill opacity=.25] (Cc) circle (1.5);
    \draw[fill=white] ($(Ct)!.5!(Cc)$) circle (1.5/2);
    \draw[fill=white] ($(Cc)!.5!(Cb)$) circle (1.5/2);
    \draw[fill=white] (XR) circle (1.5/3);
    \node at (40:1.8) {$R1\rho$};
    \node at (320:1.8) {$R3\id$};
    \node (missingcell) at (135:2.1) {$R2\rho$};
    \draw[-stealth,dotted] (missingcell) to[out=0,in=120] (30:.3);
    \node at (180:1) {$L\rho^2$};
    \node[vertex] (Vt) at (Ct) {};
    \node[above left] at (Ct) {$t$};
    \node[vertex] (Vc) at (Cc) {};
    \node[below] at (Cc) {$c$};
    \node[vertex] (Vb) at (Cb) {};
    \node[above] at (Cb) {$b$};
    \node[vertex] (Vx) at (Cx) {};
    \node[below left] at (Cx) {$x$};
    \node[vertex] (Vy) at (Cy) {};
    \node[right] at (Cy) {$y$};
    \node[vertex] (Vz) at (Cz) {};
    \node[above left] at (Cz) {$z$};
    \node[align=center] at (0,-2.5) {
        $L = (t,b,c)$\\
        $R1 = (t,z,y)$\\
        $R2 = (z,c,x)$\\
        $R3 = (y,x,b)$
        };
    \end{scope}
    \draw[-to] (-.15,0) -- node[above]{$r$} (.15,0);
    \begin{scope}[xshift=2.4cm]
    \coordinate (Ct) at (0,1.5) {};
    \coordinate (Cc) at (0,0) {};
    \coordinate (Cb) at (0,-1.5) {};
    \coordinate (XR) at (1,0) {};
    \coordinate (Cx) at ($(XR)+(225:.5)$) {};
    \coordinate (Cy) at ($(XR)+(0:.5)$) {};
    \coordinate (Cz) at ($(XR)+(135:.5)$) {};
    \draw[fill=black,fill opacity=.25] (Cc) circle (1.5);
    \draw[fill=white] ($(Ct)!.5!(Cc)$) circle (1.5/2);
    \draw[fill=white] ($(Cc)!.5!(Cb)$) circle (1.5/2);
    \draw[fill=white] (XR) circle (1.5/3);
    \node at (40:1.8) {$r(R3)$};
    \node at (320:1.8) {$r(L)$};
    \node (missingcell) at (135:2.1) {$r(R2)$};
    \draw[-stealth,dotted] (missingcell) to[out=0,in=120] (30:.3);
    \node at (180:.95) {$r(R1)$};
    \node[vertex] (Vt) at (Ct) {};
    \node[above left] at (Ct) {$r(y)$};
    \node[vertex] (Vc) at (Cc) {};
    \node[below] at (Cc) {$r(z)$};
    \node[vertex] (Vb) at (Cb) {};
    \node[above] at (Cb) {$r(t)$};
    \node[vertex] (Vx) at (Cx) {};
    \node[below left] at (Cx) {$r(c)$};
    \node[vertex] (Vy) at (Cy) {};
    \node[right] at (Cy) {$r(b)$};
    \node[vertex] (Vz) at (Cz) {};
    \node[above left] at (Cz) {$r(x)$};
    \node[align=center] at (0,-2.5) {
        $r(L) = R3 = (r(b),r(c),r(t))$\\
        $r(R1) = L = (r(y),r(t),r(z))$\\
        $r(R2) = R2 = (r(x),r(z),r(c))$\\
        $r(R3) = R1 = (r(y),r(x),r(b))$
        };
    \end{scope}
    \node[align=center] at (5.8,0) {
        $\pi_L = (132)$\\
        $\pi_{R1} = (123)$\\
        $\pi_{R2} = (123)$\\
        $\pi_{R3} = \id$
        };
\end{tikzpicture}
\end{subfigure}
\\
\
\\
\begin{subfigure}[t]{\textwidth}\centering
\begin{tikzpicture}[scale=1.15,font=\small]
    \begin{scope}[xshift=-2.4cm]
    \coordinate (Ct) at (0,1.5) {};
    \coordinate (Cc) at (0,0) {};
    \coordinate (Cb) at (0,-1.5) {};
    \coordinate (XR) at (1,0) {};
    \coordinate (Cx) at ($(XR)+(225:.5)$) {};
    \coordinate (Cy) at ($(XR)+(0:.5)$) {};
    \coordinate (Cz) at ($(XR)+(135:.5)$) {};
    \draw[fill=black,fill opacity=.25] (Cc) circle (1.5);
    \draw[fill=white] ($(Ct)!.5!(Cc)$) circle (1.5/2);
    \draw[fill=white] ($(Cc)!.5!(Cb)$) circle (1.5/2);
    \draw[fill=white] (XR) circle (1.5/3);
    \node at (40:1.8) {$R 1 \phi$};
    \node at (320:1.8) {$R 3 \phi \rho$};
    \node (missingcell) at (135:2.1) {$R 2 \phi \rho$ };
    \draw[-stealth,dotted] (missingcell) to[out=0,in=120] (30:.3);
    \node at (180:1) {$L \phi \rho$};
    \node[vertex] (Vt) at (Ct) {};
    \node[above left] at (Ct) {$t$};
    \node[vertex] (Vc) at (Cc) {};
    \node[below] at (Cc) {$c$};
    \node[vertex] (Vb) at (Cb) {};
    \node[above] at (Cb) {$b$};
    \node[vertex] (Vx) at (Cx) {};
    \node[below left] at (Cx) {$x$};
    \node[vertex] (Vy) at (Cy) {};
    \node[right] at (Cy) {$y$};
    \node[vertex] (Vz) at (Cz) {};
    \node[above left] at (Cz) {$z$};
    \node[align=center] at (0,-2.5) {
        $L = (t,b,c)$\\
        $R 1 = (t,z,y)$\\
        $R 2 = (z,c,x)$\\
        $R 3 = (y,x,b)$
        };
    \end{scope}
    \draw[-to] (-.15,0) -- node[above]{$a$} (.15,0);
    \begin{scope}[xshift=2.4cm]
    \coordinate (Ct) at (0,1.5) {};
    \coordinate (Cc) at (0,0) {};
    \coordinate (Cb) at (0,-1.5) {};
    \coordinate (XR) at (1,0) {};
    \coordinate (Cx) at ($(XR)+(225:.5)$) {};
    \coordinate (Cy) at ($(XR)+(0:.5)$) {};
    \coordinate (Cz) at ($(XR)+(135:.5)$) {};
    \draw[fill=black,fill opacity=.25] (Cc) circle (1.5);
    \draw[fill=white] ($(Ct)!.5!(Cc)$) circle (1.5/2);
    \draw[fill=white] ($(Cc)!.5!(Cb)$) circle (1.5/2);
    \draw[fill=white] (XR) circle (1.5/3);
    \node at (40:1.8) {$a(R 1)$};
    \node at (320:1.8) {$a(R 3)$};
    \node (missingcell) at (135:2.1) {$a(L)$};
    \draw[-stealth,dotted] (missingcell) to[out=0,in=120] (30:.3);
    \node at (180:1) {$a(R2)$};
    \node[vertex] (Vt) at (Ct) {};
    \node[above left] at (Ct) {$a(z)$};
    \node[vertex] (Vc) at (Cc) {};
    \node[below] at (Cc) {$a(c)$};
    \node[vertex] (Vb) at (Cb) {};
    \node[above] at (Cb) {$a(x)$};
    \node[vertex] (Vx) at (Cx) {};
    \node[below left] at (Cx) {$a(b)$};
    \node[vertex] (Vy) at (Cy) {};
    \node[right] at (Cy) {$a(y)$};
    \node[vertex] (Vz) at (Cz) {};
    \node[above left] at (Cz) {$a(t)$};
    \node[align=center] at (0,-2.5) {
        $a(L) = R 2 = (a(t),a(c),a(b))$\\
        $a(R1) = R 1 = (a(z),a(t),a(y))$\\
        $a(R2) = L = (a(z),a(x),a(c))$\\
        $a(R3) = R 3 = (a(y),a(b),a(x))$
        };
    \end{scope}
    \node[align=center] at (5.8,0) {
        $\pi_{L} = (23)$\\
        $\pi_{R1} = (12)$\\
        $\pi_{R2} = (23)$\\
        $\pi_{R3} = (23)$
        };
\end{tikzpicture}
\end{subfigure}
\caption{Two diagrams $r$ and $a$ of $E_\mathcal{AG}G$, where $G = \langle \rho, \phi \rangle$.}
\label{fig.apollonian.gasket.generators}
\end{figure}

Fix a diagram $\F = (D,R,\f,l)$.
Given any element $\alpha$ of the limit space $\Omega$, there is a unique hyperedge $a$ of $D$ such that $\alpha = a \alpha'$.
We then define an action of $\F$ on $\Omega$ as:
\[ \alpha = a\alpha' \xmapsto{\F} \f_E(a) l_a(\alpha'). \]
This map is bijective, because $\f_E$ is a bijection from $E_D$ to $E_R$ and each $l_a$ is a self-bijection of the cone $T_a$.
It is continuous because it maps cones to cones.

\begin{remark}
\label{rmk.diagram.boundary.points}
Condition (2) of \cref{def.diagram} ensures that condition (3) makes sense.
In turn, condition (3) has the following important consequence.
Let $a$ be an hyperedge of the domain hypergraph $D$ of a diagram $(D,R,\f,l)$ and consider a point $\llbracket \alpha \rrbracket \in \partial a$, say $\alpha = a \alpha'$.
Then $\llbracket \alpha \rrbracket$ represents some boundary vertex $v$ of the hyperedge $a$ and thus
\[ \llbracket \F(\alpha) \rrbracket = \llbracket \f_E(a) l_a(\alpha') \rrbracket = \lambda_{\f_E(a)} \circ \pi_a \circ \lambda_a^{-1} (v) = \f_V \circ \lambda_a \left( \lambda_a^{-1} (v) \right) = \f_V(v), \]
where the first equality holds by definition of the action of $\F$, the second by condition (3) of \cref{def.diagram} and the third because $\f$ is a $\pi$-hypergraph isomorphism.
\end{remark}

\begin{lemma}
Let $(G_c \mid c \in \mathrm{C})$ be a self-similar tuple compatible with $\mathcal{R}$.
Then each diagram defines a homeomorphism of the limit space.
\end{lemma}

\begin{proof}
We need to show that the self-homeomorphisms of the symbol space $\Omega$ defined by a diagram $\F$ preserves the gluing relation, so by \cref{prop.homeo.symbol.space.acts.on.limit.space} it descends to a homeomorphism of the limit space.
Given a diagram $\F = (D,R,\f,l)$, consider elements $\alpha = a \alpha'$ and $\beta = b \beta'$ of $\Omega$, where $a$ and $b$ are hyperedges of $D$.

First suppose that $\alpha \sim \beta$.
If $a=b$, say $a$ is colored by $c$, then $l_a \in \Aut(\mathcal{R}_c)$ by \cref{def.compatible.self.similar.tuple} and condition (1) of \cref{def.diagram}.
Hence $\F(\alpha) = \F(a \alpha') = \f_E(a) l_a(\alpha') \sim \f_E(a) l_a(\beta') = \F(a \beta') = \F(\beta)$ by \cref{def.color.automorphism}.

If instead $a \neq b$, since $\alpha \sim \beta$ we have that $a$ and $b$ are adjacent hyperedges of the domain hypergraph $D$ and the sequences $\alpha$ and $\beta$ represent the same singular point $p$ of the limit space (\cref{def.singular.isolated.regular.points}).
Being a singular point that belongs to both $\llbracket \alpha \rrbracket$ and $\llbracket \beta \rrbracket$, $p$ corresponds to a unique vertex $v$ of $D$ on which both $a$ and $b$ are incident.
Using \cref{rmk.diagram.boundary.points} on both $\alpha$ and $\beta$ we have that $\llbracket \F(\alpha) \rrbracket = \f_V(v)$ and $\llbracket \F(\beta) \rrbracket = \f_V(v)$, so $\F(\alpha) \sim \F(\beta)$ as needed.

To prove the converse, suppose now that $\alpha \not\sim \beta$.
If $a = b$, then $\F(\alpha) = \f_E(a) l_a(\alpha') \not\sim \f_E(a) l_a(\beta') = \F(\beta)$ because $l_a \in G_{c}$ and the self-similar tuple is compatible.
Suppose then that $a \neq b$.
If either $\alpha \notin \partial a$ then $\F(\alpha) \notin \partial \f_E(a)$, so it is impossible that $\F(\alpha) \sim \F(\beta)$, and the same holds if $\beta \notin \partial b$.
Assume then that $\alpha \in \partial a$ and $\beta \in \partial b$ and let $v$ and $w$ be the vertices of $D$ that they represent, respectively.
If by contradiction $\F(\alpha) \sim \F(\beta)$, then we can apply \cref{rmk.diagram.boundary.points} to show that $\f_V(v)=\f_V(w)$.
But this implies that $v=w$, so $\alpha \sim \beta$, a contradiction.
\end{proof}

\subsubsection{Equivalence of Diagrams}
\label{sub.diagram.equivalence}

We define two rewriting rules that modify a diagram without changing its ``meaning''.

A \textbf{permutation} is a rewriting $\boldsymbol{(D,R,\f,l) \rightarrow (D,R,\f',l)}$ where the $\f$, $\pi$-hypergraph isomorphism, is replaced by a $\f'$ obtained by modifying the permutation $\pi$ on vertices that are not points on the topological boundary of any cell.

An \textbf{expansion} is a rewriting $\boldsymbol{(D,R,\f,l) \rightarrow (D \triangleleft e, R \triangleleft \f_E(e), \f \triangleleft e, l \triangleleft e)}$ where the entries of the second diagram are obtained as follows:
\begin{itemize}
    \item $D \triangleleft e$ and $R \triangleleft \f_E(e)$ are hypergraph expansions (\cref{def.hypergraph.expansion}).
    \item $\f \triangleleft e$ is a $\pi \triangleleft e$-hypergraph isomorphism $D \triangleleft e \to R \triangleleft \f_E(e)$ defined as follows:
    \begin{itemize}
        \item $\f_V \triangleleft e$ agrees with $\f_V$ on $V_D$;
        \item $\f_E \triangleleft e$ agrees with $\f_E$ on $E_D \setminus e$ and maps each $ex$ to $\f_E(e)\hat{l_e}(x)$;
        \item $\pi \triangleleft e$ agrees with $\pi$ on $E_D \setminus e$ and maps each $ex$ to any permutation $\sigma$ such that $\lambda_{\f_E} \triangleleft e (ex)) \circ \sigma \circ \lambda_{ex}^{-1}$ agrees with $\llbracket ex\alpha \rrbracket \mapsto \llbracket \f_E \triangleleft e (ex) (l \triangleleft e)_{ex} \rrbracket$ on $\partial ex$.
    \end{itemize}
    \item $l \triangleleft e$ agrees with $l$ on $E_D$ and maps $ex$ to $l_{ex} \coloneq l_e | x$.
\end{itemize}
If $\F$ is the initial diagram, it will be convenient to write $\F \triangleleft e$ to denote such an expansion.
In essence, we expand both the domain and range hypergraphs of $g$ in their hyperedges $e$ and $g_E(e)$, we permute the new hyperedges according to the label of $e$ in $g$ and we update the labelings and the map $\pi$.
For example, \cref{fig.grig.diagrams} depicts a diagram along with two of its expansions (further details in \cref{sub.dendrites}).

\begin{definition}
Two diagrams $(D,R,\f,l)$ and $(D',R',\f',l')$ are \textbf{equivalent} when they differ by a finite sequence of permutations and expansions.
\end{definition}

Using the fact that each $l_e$ induces a cell automorphism of $\llbracket \f_E(e) \rrbracket$ that is compatible with $\pi_e$ on $\partial e$ (see \cref{rmk.diagram.boundary.points}) together with \cref{rmk.cell.automorphisms}, it is not hard to show that permutations and expansions of a diagram are diagrams themselves, i.e., they satisfy \cref{def.diagram}, so this is a well defined equivalence relation on the set of all diagrams.
Moreover, diagram equivalence fully characterizes the action of two diagrams on the limit space, as is showed in the following Lemmas.

\begin{lemma}
\label{lem.equivalence.works}
Equivalent diagrams have the same action.
\end{lemma}

\begin{proof}
Permutations clearly do not change the action of a diagram.
As for expansions, let $\F$ be a diagram and $\F \triangleleft e$ be an expansion of $\F$.
Every sequence $\alpha$ in the symbol space has some hyperedge $a$ of $D$ as a prefix, so let $\alpha = a \alpha'$.
Then, by definitions, $\F(\alpha) = \f_E(a) l_a(\alpha')$ and we distinguish two cases for $\F \triangleleft e (\alpha)$.
When $a \neq e$ we have
\[ \F \triangleleft e (a \alpha') = (\f_E \triangleleft e) (a) (l \triangleleft e)_a (\alpha') = \f_E(a) l_a (a \alpha'), \]
as needed.
When instead $a = e$, say $\alpha = ax\alpha''$, then
\[ \F \triangleleft a (ax\alpha'') = (\f_E \triangleleft a) (ax) (l \triangleleft a)_{ax} (\alpha'') = \f_E(a) \hat{l_a}(x) l_a|x (\alpha'') = \f_E(a) l_a(x\alpha''), \]
so two equivalent diagrams have the same action.
\end{proof}

Even if there is not a unique ``smallest'' diagram, the next Lemma shows that all such ``smallest'' diagrams only differ by essentially trivial modifications of the diagram.

\begin{lemma}
Given a diagram $\F$, it is equivalent to diagrams that are minimal with respect to the amount of their hyperedges and such diagrams only differ by performing permutations and expansions of trivially colored hyperedge (\cref{def.trivial.replacement.hypergraph})
In particular, if the replacement system is expanding then the minimal diagrams of each equivalence class differ by permutations and thus are finitely many.
\end{lemma}

\begin{proof}
The strategy is similar to that of \cite[Proposition 1.19]{BF19} for rearrangements.
Given a diagram $\F$, a hypergraph expansion $\Gamma$ is the domain hypergraph for a diagram with the same action of $\F$ if and only if such action has states in the self-similar-tuple on each cell in $\{ \llbracket e \rrbracket \mid e \in E_\Gamma \}$.
Among these hypergraph expansions, let $D$ be one with maximal cells (with respect to set inclusion).
Each non-trivially colored cell corresponds to a unique hyperedge, whereas a trivially colored cell $\llbracket e \rrbracket$ is the same as $\llbracket e w \rrbracket$ for any expansion of the trivially colored hyperedge $e$.
Thus, two choices of $D$ can only differ up to expansions of trivially colored hyperedges.
Once $D$ has been chosen, it is readily seen that the action of $\F$ uniquely determines the range hypergraph $R$, the bijections $\f_V$ and $\f_E$ and the label $l$.
The only aspect of the diagram that may differ is the map $\pi$.
However, $\pi$ is partially determined by $\f_E$ and $l$ because of condition (3) of \cref{def.diagram}, so different choices can only differ outside of the topological boundaries of cells.
Such distinct diagrams are all the same up to permutations, as needed.
\end{proof}

Note that in the previous Lemma we have actually shown that two diagrams with the same action are equivalent to a common minimal diagram, so they are themselves equivalent.
Together with \cref{lem.equivalence.works}, this means that two diagrams are equivalent if and only if they have the same action.

\subsubsection{Composition of diagrams}

Given two diagrams $\F$ and $\G$, up to expanding them, we can always assume that the domain hypergraph of $\F$ and the range hypergraph of $\G$ coincide.

Let $\F = (B,C,\f,l^{(\f)})$ and $\G = (A,B,\g,l^{(\g)})$ be two diagrams.
The composition $\F \G$ between two diagrams is then
\[ \F \G = (A, C, \f \circ \g, l^*), \text{ where } l_e^* \coloneq l_{\g_E(e)}^{(\f)} \circ l_e^{(\g)} \]

Using \cref{prop.pi.isomorphism.composition}, it is tedious but easy to verify that $(A, C, \f \circ \g, l^*)$ is truly a diagram.
Checking that the composition defined in this way corresponds to the composition of homeomorphisms of the limit space is routine.

\begin{proposition}
The collection of equivalence classes of diagrams defines a group under composition.
\end{proposition}

\begin{definition}
\label{def.ESS}
Given an almost expanding replacement system $\mathcal{R}$ and a compatible self-similar tuple $\mathbb{G}$, the group of \textbf{eventually-$\boldsymbol{\mathbb{G}}$ homeomorphisms} of the limit space of $\mathcal{R}$ is the group of those homeomorphisms that are represented by diagrams (as in \cref{def.diagram}).
We denote it by $\boldsymbol{E_\mathcal{R} \mathbb{G}}$.
\end{definition}

When not working with a specific tuple $\mathbb{G}$, we may refer to $E\mathbb{G}$ as an \textbf{eventually self-similar group of homeomorphisms} of the limit space, or \textbf{ESS} in short.

\begin{definition}
\label{def.rearrangement}
Given a replacement system, its \textbf{rearrangement group} is the ESS group $E\mathds{1}$, where $\mathds{1}$ is the self-similar tuple consisting of trivial groups, and we denote it by $\boldsymbol{RG_\mathcal{R}}$.
\end{definition}

In the case of expanding replacement systems of graphs, this definition coincides with that from \cite{BF19}.

\begin{remark}
In the case of connected limit spaces, the definition of eventually self-similar group of homeomorphims fits into the more general definition of quasisymmetries (see \cite{QS_BF}).
\end{remark}

\subsection{Sierpinski Triangle and Apollonian Gasket}
\label{sub.triangle.gasket.homeo.groups}

The homeomorphism group of the Sierpinski triangle is a finite group (the dihedral group on the triangle, see \cite[Theorem 2.3.6]{TriangleGasket}), which is readily seen to be a self-similar group generated by the two diagrams whose domain and range hypergraphs are the base hypergraph and its sole hyperedge is labeled, respectively, by $\rho$ and by $\phi$ from \cref{ex.color.automorphism}.

Theorem 3.4.5 of the dissertation \cite{TriangleGasket} shows that elements $r$, $c$ and $a$ depicted in \cite[Figure 3.3.1]{TriangleGasket} generate $\mathrm{Homeo}(AG)$.
Using as self-similar tuple (which is a group because there is a unique color) the group generated by the reflection $\phi$ and the rotation $\rho$ introduced in \cref{ex.color.automorphism}, two generators $r$ and $a$ are represented as diagrams in \cref{fig.apollonian.gasket.generators}.
The generator $c$ can also be represented as a diagram (although drawing it requires expanding the hyperedge $R2$ and the resulting figure has small details that are cumbersome to read), so the entire homeomorphism group of the Apollonian Gasket is an ESS group.


\section{Finiteness Properties of ESS Groups}
\label{sec.finiteness.properties}

Following \cite{Witzel}, here we define categories $\C_{\R,\mathbb{G}}$ whose fundamental groups are the ESS groups.
This is similar to the construction given in \cite[Subsection 5.5]{Witzel} for rearrangement groups, with the difference that the morphisms from an object to itself are extended to certain labeled $\pi$-hypergraph isomorphisms between hypergraph expansions of almost expanding replacement systems.

Throughout this section, if $\C$ is a (small) category, we will denote by $\mathrm{Ob}(\C)$ the set of its objects and, given $A,B \in \mathrm{Ob}(\C)$, we will denote by $\C(A,B)$ the set of morphisms from $B$ to $A$ (this is consistent with the notation from \cite{Witzel}).

The \textbf{fundamental group} $\pi_{1}(\mathcal{C},A)$ of a category is the group of maps from $A$ to itself which are a composition of morphisms in $\mathcal{C}$ and inverses of such morphisms (which may not belong to $\mathcal{C}$).
See \cite{Witzel} for more details.

\subsection{The Category for an ESS Group}

We first describe two useful notions related to replacement systems and ESS groups.

A \textbf{simple contraction} is the inverse of a simple expansion.
A \textbf{trivial} simple contraction is the inverse of a trivial simple expansion (recall \cref{def.trivial.replacement.hypergraph}).

We define \textbf{generalized diagrams} $(D,R,\f,l)$ by allowing $D$ and $R$ to be hypergraph expansions of replacement systems with different base hypergraphs.
For more details see \cite[Definition 3.1]{BF19}, which does this for rearrangement groups.

For the next definition, recall the notion of permutations of diagrams discussed in \cref{sub.diagram.equivalence}.

\begin{definition}
\label{def.our.category}
Let $\R$ be an almost expanding replacement system with set of colors $C = \{1, \dots, k\}$, base hypergraph $\Gamma_0$ and replacement hypergraphs $R_1, \dots, R_k$.
Let $\mathbb{G} = (G_1, \dots, G_k)$ be a self-similar tuple that is compatible with $\R$.
We define the category $\boldsymbol{\C_{\R,\mathbb{G}}}$ with the following data.
\begin{description}
    \item[Objects] $\mathrm{Ob}(\C_{\R,\mathbb{G}})$ is the set of all hypergraphs expansions of $\R$.
    \item[Morphisms] $\C_{\R,\mathbb{G}} = \mathcal{F}_\R \cup \mathcal{G}_{\R,\mathbb{G}}$, where
    \begin{itemize}
        \item $\mathcal{F}_\R$ is generated by the non-trivial simple hypergraph contractions based on $\R$;
        \item $\mathcal{G}_{\R,\mathbb{G}}$ consists of classes of generalized diagrams $(D,R,\f,l)$ differing from permutations, where $D$ and $R$ are the base hypergraphs of the domain and range replacement systems.
    \end{itemize}
\end{description}
\end{definition}

Hypergraph expansions have hyperedges coming from a fixed countable set, the language $\mathbb{L}_\R$, and the set of vertices is also countable, so $\C_{\R,\mathbb{G}}$ is a small category.
Also note that both $\mathrm{Ob}(\C_{\R,\mathbb{G}})$ and $\mathcal{F}_\R$ do not depend on the self-similar tuple $G$.

Throughout this section, $\R$ is going to be a fixed almost expanding replacement system and $\mathbb{G}$ a self-similar tuple that is compatible with $\R$.
Thus, to avoid unnecessary clutter, unless it is useful to contextualize, we will omit them from the notation and simply write $\C$, $\mathcal{F}$ and $\mathcal{G}$ in place of $\C_{\R,\mathbb{G}}$, $\mathcal{F}_\R$ and $\mathcal{G}_{\R,\mathbb{G}}$, respectively.

We recall that, given a category $\mathcal{X}$ and two subcategories $\mathcal{A}$ and $\mathcal{B}$, we say that $\mathcal{X}$ is an (internal) Zappa-Szép product of $\mathcal{A}$ and $\mathcal{B}$, and we write $\mathcal{X} = \mathcal{A} \bowtie \mathcal{B}$, when each element of $\mathcal{X}$ can be written as a product $ab$ for some $a \in \mathcal{A}$ and $b \in \mathcal{B}$.

\begin{proposition}
\label{prop.category.is.product}
$\C$ is an (internal) Zappa-Szép product $\mathcal{F}  \bowtie \mathcal{G}$.
\end{proposition}

\begin{proof}
Given any $g \in \mathcal{G}$ and a generator $f$ of $\mathcal{F}$, an element $gf$ is a simple hyperedge contraction $f$ followed by some $\pi$-isomorphism.
Say that $f$ is the contraction of a hyperedge $e$.
Then the $\pi$-isomorphism $g$ can be expanded to $g \triangleleft e$ as described in \cref{sub.diagram.equivalence}.
Then it is easy to check that $gf = f' (g \triangleleft e)$, where $f'$ is the contraction of the hyperedge $g_E(e)$.
\end{proof}

The ESS group $E_\R \mathbb{G}$ is the fundamental group $\pi_{1}(\C_{\R,\mathbb{G}},\Gamma_0)$: each element consists of a sequence of expansions (anti-contractions) followed by a labelled $\pi$-hypergraph isomorphism and then a sequence of contractions, which is a loop in the category.
The category is connected, so the base point does not matter.

\begin{remark}
Not including hyperedge contractions in \cref{def.our.category} clearly does not change the group.
In fact, one could affirm that the very inclusion of trivial expansions in almost expanding replacement systems could have been avoided entirely, instead not allowing trivial colors to expand at all.
However, doing so would cause the disappearance of isolated points from the symbol and limit spaces.
\end{remark}

\begin{remark}
\label{rmk.invertibles}
Note that the set $\C^\times$ of invertibles of $\C$ is precisely the subcategory $\mathcal{G}$.
We will thus use $\C^\times$ or $\mathcal{G}$ interchangeably depending on the context.

In the case in which $\mathbb{G}$ is trivial (i.e., for rearrangement groups), these are simply hypergraph isomorphisms, so in this case every invertible is in $\C(x,x)$ for some $x \in \mathrm{Ob}(\C)$.
However, when $\mathbb{G}$ allows for a nontrivial permutation of the boundary vertices there will be invertibles in $\C(x,y)$, where $x, y \in \mathrm{Ob}(\C)$ are $\pi$-isomorphic but possibly different hypergraphs.
\end{remark}

We ultimately want to apply Theorem 3.12 from \cite{Witzel}, which we state right below.
We will recall the relevant definitions in the upcoming subsections.

\begin{witzel}
\label{thm.witzel}
Let $\mathcal{X}$ be a factor-finite right-Ore category and let $\star \in \mathrm{Ob}(\mathcal{X})$.
Let $\mathcal{S}$ be a locally finite left-Garside family that is closed under taking factors.
Let $\rho \colon \mathrm{Ob}(\mathcal{X}) \to \mathbb{N}$ be a height function such that $\{x \in \mathrm{Ob}(\mathcal{X}) \mid \rho(x) \leq n\}$ is finite for every $n \in \mathbb{N}$.
Assume
\begin{enumerate}[leftmargin=2cm]
    \item[\textsc{(STAB)}] for all but finitely many $k \in \mathbb{N}$, $\mathcal{X}^\times(x,x)$ has type $F_n$ for all $x \in \rho^{-1}(k)$;
    \item[\textsc{(LK)}] there exists an $N \in \mathbb{N}$ such that $|E(x)|$ is $(n-1)$-connected for all $x$ with $\rho(x) \geq N$.
\end{enumerate}
Then $\pi_1(\mathcal{X},\star)$ has type $F_n$.
\end{witzel}

\subsection{Factor-finiteness and right-Ore}

We say that two contractions $f_1, f_2$ of $x \in \mathrm{Ob}(\C)$ are \textbf{parallel} if the two subhypergraphs of $x$ being contracted by $f_1$ and $f_2$ do not share hyperedges.
In this case, say $f_1 \in \mathcal{F}(y_1,x)$ and $f_2 \in \mathcal{F}(y_2,x)$, there exist $f_1' \in \mathcal{F}(z,y_2)$ and $f_2' \in \mathcal{F}(z,y_1)$ such that $f_1' f_2 = f_2' f_1 \in \mathcal{F}(z,x)$.
Essentially, this is because the two subhypergraphs can be contracted in whichever order one wishes.
The same can be defined for a finite set of contractions of $x$ by requiring that all pairs of subhypergraphs being contracted do not share hyperedges.

Conversely, given $f = f_1' f_2' \in \mathcal{F}(z,x)$, say $f_1' \in \mathcal{F}(z,y_2)$ and $f_2' \in \mathcal{F}(y_2,x)$, we say that $f_1'$ and $f_2'$ are \textbf{parallel} if the subhypergraph contracted by $f_1'$ does not include the hyperedge resulting from the contraction $f_2'$.
In this case, there exists $f_1 \in \mathcal{F}(y_1,x)$ that contracts the subhypergraph of $x$ corresponding to that of $y_2$ contracted by $f_1'$, so $f_1$ and $f_2'$ are parallel in the sense described in the paragraph above.
This definition can be extended to any finite sequence $f = f_1' \dots f_k'$ by requiring that the subhypergraph contracted by each $f_i'$ does not include any of the hyperedges resulting from the previous contractions $f_{i+1}', \dots, f_k'$.

It is clear that the two notions of parallelism that we define are equivalent in the sense that a set of parallel contractions $f_1, \dots, f_k$ of $x$ (the first definition) induces a sequence $f_1' \dots f_k'$ of parallel contractions (the second definition) that is unique once an order has been fixed.
Conversely, a sequence $f_1' \dots f_k'$ of parallel contractions starting from $x$ (second definition) induces a unique set of parallel contractions $f_1, \dots, f_k$ of $x$ (first definition).

Given morphisms $f,g$, we say that $f$ is a \textbf{factor} of $g$ if $g=lfr$ for some morphisms $l$ and $r$.
We say that $f$ and $g$ \textbf{have common right-multiple} if there exist morphisms $a$ and $b$ such that $af = bg$.

\begin{definition}
A category
\begin{itemize}
    \item is \textbf{factor-finite} if every morphism has only finitely many factors up to pre- and postcomposition by invertibles.
    \item is \textbf{cancellative} if, whenever $ef = eg$ or $fe = ge$, one has $f=g$.
    \item \textbf{has common right-multiples} if any two elements with the same target have a common right-multiple.
    \item is \textbf{right-Ore} if it is cancellative and has common right-multiples.
\end{itemize}
\end{definition}

By \cref{rmk.invertibles}, to verify that $\C$ is factor-finite we only need to show it for $\mathcal{F}$.
It is straightforward to see that hyperedge contractions do not have proper factors and that every morphism in $\mathcal{F}$ can be seen as a finite composition of hyperedge contractions (recall \cref{def.hypergraph.expansion}).
Every internal Zappa-Szép product $\mathcal{A} \bowtie \mathcal{B}$ naturally equips $\mathcal{A}$ with an action on $\mathcal{B}$, which is $b^a$ is the unique element such that $ba = a' b^a$ for some $a' \in \mathcal{A}$.
An action is said to be \textbf{injective} if, for all $a \in \mathcal{A}$, we have that $b_{1}^a = b_{2}^a$ implies $b_{1} = b_{2}$.

As already hinted in the proof of \cref{prop.category.is.product}, for our category $\C = \mathcal{F}  \bowtie \mathcal{G}$ the action of $\mathcal{F}$ on $\mathcal{G}$ is defined as follows on the generators:
if $f \in \mathcal{F}$ is a simple contraction $x \leftarrow x \triangleleft e$ and $g \in \mathcal{G}$ is a $\pi$-isomorphism $y \leftarrow x$, then $g^f$ is the $(\pi \triangleleft e)$-isomorphism $y \triangleleft g_E(e) \leftarrow x \triangleleft e$ obtained as described in \cref{sub.diagram.equivalence}.

Proposition 4.5(i) of \cite{Witzel} shows that, if $\mathcal{A}$ is right-Ore and the action of $\mathcal{A}$ on $\mathcal{B}$ is injective, then $\mathcal{A} \bowtie \mathcal{B}$ is right-Ore.
Let us verify these hypotheses.

\begin{proposition}
The category $\mathcal{F}$ is right-Ore.
\end{proposition}

\begin{proof}
$\mathcal{F}$ is cancellative because one can always ``delete'' the first or last from a sequence of contractions.
It has common right-multiples because, given two hypergraph expansions, it is always possible to find a common hypergraph expansion.
\end{proof}

\begin{proposition}
The action of $\mathcal{F}$ on $\mathcal{G}$ is injective.
\end{proposition}

\begin{proof}
If expanding an hyperedge $e$ of two $\pi$-isomorphisms $x \to y$ produces the same $\pi$-isomorphism, then the two $\pi$-isomorphisms were equal beforehand.
\end{proof}

Together with \cref{prop.category.is.product}, these last two propositions and the aforementioned Proposition 4.5(i) of \cite{Witzel} ultimately show that $\C$ is right-Ore.
It is factor-finite because $\mathcal{F}$ is and the morphisms of $\mathcal{G}$ are invertible.
In short, we gather these results in the following proposition.

\begin{corollary}
\label{cor.our.category.is.ore}
The category $\C$ is factor-finite right-Ore.
\end{corollary}

\subsection{Left-Garside Family for \texorpdfstring{$\C_{\R,\mathbb{G}}$}{CRG}}

To any hypergraph $x$ in $\mathrm{Ob}(\mathcal{F})$ we associate its \textbf{full expansion} $x^*$, i.e., the hypergraph obtained by expanding every hyperedge.
Let $\Delta \colon \mathrm{Ob}(\mathcal{F}) \to \mathcal{F}$ be the morphism $x \leftarrow x^*$ that consists of parallel contractions.
Let
\[ \mathrm{Div}(\Delta) = \{ g \in \mathcal{F} \mid gh = \Delta(x) \text{ for some } x \in \mathrm{Ob}(\mathcal{F}), h \in \mathcal{F} \}, \]
which consists of every composition of parallel contractions.
This set is clearly locally finite, since each hypergraph can only admit finitely many contractions.
Let
\[ \widetilde{\mathrm{Div}}(\Delta) = \{ h \in \mathcal{F} \mid gh = \Delta(x) \text{ for some } x \in \mathrm{Ob}(\mathcal{F}), g \in \mathcal{F} \}, \]
which consists of every composition of parallel contractions of $x^*$, for any $x \in \mathrm{Ob}(\mathcal{F})$.
It is clear that $\mathrm{Div}(\Delta)$ generates $\mathcal{F}$ and that $\widetilde{\mathrm{Div}}(\Delta) \subseteq \mathrm{Div}(\Delta)$.

Every $g \in \mathcal{F}(x,-)$ has a maximal left-factor of parallel contractions, which is the greatest common left-factor of $g$ and $\Delta(x)$.
Then the map $\Delta$ is a right-Garside map.
By \cite[Observation 1.7]{Witzel}, $\mathcal{S} := \mathrm{Div}(\Delta)$ is a left-Garside family closed under taking factors.
Thanks to \cite[Proposition 4.5 (iv)]{Witzel}, this translates to Zappa-Szép products, so by \cref{prop.category.is.product} we have the following.

\begin{proposition}
\label{prop.our.garside.family}
$\mathcal{S} = \mathrm{Div}(\Delta)$ is a locally finite left-Garside family in $\C$ that is closed under taking factors.
\end{proposition}

\subsection{Stabilizers in \texorpdfstring{$\C_{\R,\mathbb{G}}$}{CRG}}

In this subsection we translate condition \textsc{(STAB)} of \hyperref[thm.witzel]{Witzel's Theorem} into a sufficient condition on the groups of the self-similar tuple.

\begin{proposition}
\label{prop.finite.index.stabilizers}
For all $x \in \mathrm{Ob}(\C)$, the group $\C^\times(x,x)$ contains a finite-index copy of $\bigtimes \{ G_{c(e)} \mid e \in E_x \}$.
\end{proposition}

\begin{proof}
As noted in \cref{rmk.invertibles}, $\C^\times$ is none other than $\mathcal{G}$.
Then $\C^\times(x,x)$ consists of all $\pi$-automorphisms of the hypergraph $x$ together with a labeling of the hyperedges by elements of the groups of the self-similar tuple $\mathbb{G} = (G_1, \dots, G_k)$ that agrees with $\pi$.
Thus, there are as many copies of $G_i$ as there are hyperedges colored by $i$ in $x$ and the $\pi$-automorphisms of $x$ act by permuting the copies of the groups $G_1, \dots, G_k$ according to the permutation of hyperedges.
This means that $\C^\times(x,x)$ is a subgroup of the generalized wreath product
\[ \C^\times(x,x) \simeq H \leq \bigtimes \{ G_{c(e)} \mid e \in E_x \} \rtimes \Pi\Aut(x), \]
where $\Pi\Aut(x)$ is the group of all $\pi$-automorphisms of $x$.
The subgroup $H$ consists of the elements of $\bigtimes \{ G_{c(e)} \mid e \in E_x \} \rtimes \Pi\Aut(x)$ whose labeling agrees with $\pi$.

Now, $H$ clearly includes $\bigtimes \{ G_{c(e)} \mid e \in E_x \}$, so it is an intermediate group between $\bigtimes \{ G_{c(e)} \mid e \in E_x \}$ and $\bigtimes \{ G_{c(e)} \mid e \in E_x \} \rtimes \Pi\Aut(x)$.
Since the former has finite index in the latter, $\bigtimes \{ G_{c(e)} \mid e \in E_x \}$ must have finite index in $H$.
Then $\bigtimes \{ G_{c(e)} \mid e \in E_x \}$ has finite index in $H \simeq \C^\times(x,x)$.
\end{proof}

\begin{corollary}
\label{cor.finiteness.stabilizers}
If the groups of the self-similar tuple all have type $F_n$, then $\C^\times(x,x)$ has type $F_n$ for all $x \in \mathrm{Ob}(\C)$.
\end{corollary}

\begin{proof}
It is known that a group and its finite-index subgroups share their finiteness properties (see for example \cite[Corollary 7.2.4]{Geoghegan}).
Then, in light of \cref{prop.finite.index.stabilizers}, it suffices to show that $\bigtimes \{ G_{c(e)} \mid e \in E_x \}$ has type $F_n$.
The product of two groups of type $F_n$ has type $F_n$ (see for example Exercise 1 of Section 7.2 \cite{Geoghegan}), so $\bigtimes \{ G_{c(e)} \mid e \in E_x \}$ have type $F_n$ as soon as the groups of the self-similar type have type $F_n$, as needed.
\end{proof}

\subsection{The complexes}

A \textbf{height function} for a category $\mathcal{X}$ is a map $\rho \colon \mathrm{Ob}(\mathcal{X}) \to \mathbb{N}$ that satisfies the two following properties:
\begin{enumerate}
    \item if $\mathcal{X}(x,y)$ includes some invertible morphism then $\rho(x)=\rho(y)$;
    \item if $\mathcal{X}(x,y)$ includes some non-invertible morphism then $\rho(x) < \rho(y)$.
\end{enumerate}

For our category $\C$, we let $\rho(x)$ denote the number of hyperedges of $x$, which is readily seen to be a height function.

Following Witzel's notations from \cite{Witzel}, we consider the set of elementary elements: $\E = \C^\times \cup \C^\times \mathcal{S} \C^\times$.
Recall that $\C^\times$ consists of certain labeled $\pi$-isomorphisms of hypergraphs, as explained in \cref{rmk.invertibles}.
Moreover, $\C^\times \mathcal{S} \C^\times = \mathcal{S} \C^\times$, so ultimately $\E = \C^\times \cup \mathcal{S} \C^\times$ consists of the invertibles of $\C$, possibly followed by a composition of parallel contractions.

Define $E(x)$ as the set of equivalence classes in $\E(-,x) \setminus \E^\times(x,x)$ modulo the equivalence relation that relates $f$ and $g$ when there exists an $h \in \C^\times$ with $hg=f$.
We denote by $\bar{f} \in E(x)$ the class of $f \in \E(-,x) \setminus \E^\times(x,x)$.
Note that $\E(-,x) \setminus \E^\times(x,x)$ consists of elements $hg$, with $g \in \C^\times(-,x)$ and $h$ a composition of parallel contractions, where $h$ must be non-trivial when $g \in \C^\times(x,x)$, but can be non-trivial when $g \in \C^\times(y,x)$ for $x \neq y$.

We define a partial order on $E(x)$ by setting $\bar{f} \leq \bar{g}$ when there exists an $h \in \C$ such that $hg = f$.
To apply \hyperref[thm.witzel]{Witzel's Theorem} we need to give sufficient conditions for which $|E(x)|$ are $(n-1)$-connected.

As in \cite{SWZ19}, we will not work directly with $|E(x)|$.
We will soon introduce a simplicial complex $K_x$ whose barycentric subdivision is isomorphic to $|E(x)|$.

\subsection{\texorpdfstring{$\boldsymbol{\pi}$}{pi}-contractions}

Given a hypergraph $x \in \mathrm{Ob}(\C)$, we denote by $\boldsymbol{\C^\times \cdot x}$ the set of those hypergraphs $y$ such that $\C^\times(y,x) \neq \emptyset$.
Clearly, this partitions the set $\mathrm{Ob}(\C)$ into finite subsets of objects with the same height (number of hyperedges).

\begin{definition}
\label{def.pi.contraction}
A \textbf{simple $\boldsymbol{\pi}$-contraction} of a hypergraph $x \in \mathrm{Ob}(\C)$ is a simple contraction $f \in \mathcal{F}(-,y)$ for some $y \in \C^\times \cdot x$.
\end{definition}

In practice, a simple $\pi$-contraction of $x$ is a simple contraction $f$ of some hypergraph $y$ that can be obtained from $x$ via some diagram.
Once a $g \in \C^\times(y,x)$ has been fixed, $f$ can be seen as a contraction of a subhypergraph of $x$ that first needs to be ``adjusted'' by $g$.

For example, consider the replacement system $\mathcal{A}$ from \cref{ex.replacement.systems} and the self-similar tuple $\mathbb{G} = \left( G_{\text{\textcolor{blue}{b}}} = \langle \phi_\mathcal{A} \rangle = \{1,\phi_\mathcal{A}\}, G_{\text{\textcolor{red}{r}}} = 1 \right)$, with $\phi_\mathcal{A}$ as described in \cref{ex.color.automorphism}.
The graph expansion depicted in \cref{fig.airplane.contractions} features the following four simple $\pi$-contractions:
a simple contraction of the subgraph with edges $II$, $IF$, $IT$ and $IB$;
a simple contraction of the subgraph with edges $FI$, $FF$, $FT$ and $FB$;
two simple $\pi$-contractions of the subgraph with edges $II$, $FI$, $T$ and $B$ (in one $\lambda_{R_{\text{\textcolor{blue}{b}}}}(1)$ is identified with $\lambda_{II}(1)$, in the other with $\lambda_{FI}(1)$).
The $\pi$-contractions require reversing the orientation of the edges $II$ and $FI$, which can be done using $\phi_\mathcal{A}$.

\begin{figure}
\centering
\begin{tikzpicture}[scale=.83333,font=\small]
    \draw[edge,red,domain=5:175] plot ({1*cos(\x)}, {1*sin(\x)});
    \draw (90:1) node[above,red] {$T$};
    \draw[edge,red,domain=185:355] plot ({1*cos(\x)}, {1*sin(\x)});
    \draw (270:1) node[above,red] {$B$};
    \draw[edge,red,domain=5:175] plot ({.4*cos(\x)-2.25}, {.4*sin(\x)});
    \draw (-2.25,.4) node[above,red] {$IB$};
    \draw[edge,red,domain=185:355] plot ({.4*cos(\x)-2.25}, {.4*sin(\x)});
    \draw (-2.25,-.4) node[below,red] {$IT$};
    \draw[edge,red,domain=5:175] plot ({.4*cos(\x)+2.25}, {.4*sin(\x)});
    \draw (2.25,.4) node[above,red] {$FT$};
    \draw[edge,red,domain=185:355] plot ({.4*cos(\x)+2.25}, {.4*sin(\x)});
    \draw (2.25,-.4) node[below,red] {$FB$};
    \node[vertex] (l) at (-3.5,0) {};
    \node[vertex] (l1) at (-2.65,0) {};
    \node[vertex] (l2) at (-1.85,0) {};
    \node[vertex] (cl) at (-1,0) {};
    \node[vertex] (cr) at (1,0) {};
    \node[vertex] (r1) at (2.65,0) {};
    \node[vertex] (r2) at (1.85,0) {};
    \node[vertex] (r) at (3.5,0) {};
    \draw[edge,blue] (l1) to node[below]{$IF$} (l);
    \draw[edge,blue] (l2) to node[below]{$II$} (cl);
    \draw[edge,blue] (r1) to node[below]{$FF$} (r);
    \draw[edge,blue] (r2) to node[below]{$FI$} (cr);
\end{tikzpicture}
\caption{The graph expansion $\Gamma_0 \triangleleft I \triangleleft F$ of $\mathcal{A}$ that allows four simple $\pi$-contractions, with only a pair of parallel ones.}
\label{fig.airplane.contractions}
\end{figure}
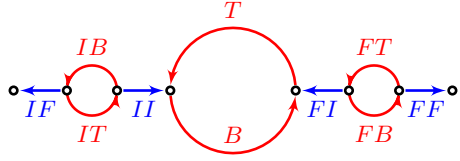

Two simple $\pi$-contractions $f_1$ and $f_2$ of $x$, say $f_i \in \mathcal{F}(-,y_i)$ with $y_i \in \C^\times \cdot x$, are \textbf{parallel} when there exist $g_i \in \C^\times(y_i,x)$ such that the two subhypergraphs of $x$ that are contracted in $f_1 g_1$ and $f_2 g_2$ do not share hyperedges.
More generally, this definition can be naturally extended to finite sets of $\pi$-contractions of $x$ by requiring the same condition on each pair $f_i g_i$ and $f_j g_j$.
For example, in the graph expansion depicted in  \cref{fig.airplane.contractions}, the two simple contractions are parallel, whereas each of the two $\pi$-contractions is not parallel to any other $\pi$-contraction.

\begin{remark}
\label{rmk.parallel.pi.contractions}
Given a set of parallel simple $\pi$-contractions, there exists a $g \in \C^\times(y,x)$ for a $y$ where all the contractions can be realized.
This is found by individually ``adjusting'' the subhypergraphs of $x$ that are being contracted.
Hence, in practice, when one wishes to find a family of parallel simple $\pi$-contractions of $x$, it suffices to consider the simple contractions of $y$.
\end{remark}

For any fixed $x \in \mathrm{Ob}(\C)$, consider the simplicial complex $\boldsymbol{K_x}$ built as follows.
\begin{itemize}
    \item Its vertices are the equivalence classes of simple $\pi$-contractions of $x$ up to left multiplication by invertibles.
    We will write $\boldsymbol{[h]}$ to denote a vertex of $K_x$, i.e., the equivalence class of $h$.
    \item Vertices $[f_1], \dots, [f_k]$ form a simplex when $f_1, \dots, f_k$ are parallel simple $\pi$-contractions.
\end{itemize}

\begin{remark}
If we restrict our attention on the cases in which the gluing relation is trivial (i.e., when the limit space is an edge shift), there are no non-trivial invertibles between different objects of the category, so simple $\pi$-contractions correspond to simple contractions.
For Scott-R\"{o}ver-Nekrashevych groups (i.e., when the edge shift is a full shift), our complexes are the same that were built in \cite{SWZ19}.
\end{remark}

A vertex of $E(x)$ is an element $sc$ with $s \in \mathcal{S}$ and $c \in \C^\times$.
Then $c$ is some diagram $y \leftarrow x$ and $s$ is a sequence of parallel simple contractions starting from $y$.
We define a map that sends $sc$ to the set of (equivalence classes of) simple contractions of which $s$ is composed.
These are parallel simple $\pi$-contractions of $x$, so the image of each vertex of $E(x)$ is a simplex of $K_x$.
It is clear that this map is a bijection and that it maps pairs of adjacent vertices of $E(x)$ to simplices that are one contained in the other, so we have the following.

\begin{proposition}
\label{prop.barycentric.subdivision}
For each $x \in \mathrm{Ob}(\C)$, the complex $|E(x)|$ is isomorphic to the barycentric subdivision of $K_x$, so the two are homeomorphic.
Moreover, the simplicial complex $K_x$ is flag.
\end{proposition}

\subsection{Condition for connectivity of the complexes}
\label{sub.connectivity}

We say that a simplex $\sigma$ of $K$ is a \textbf{$\boldsymbol{k}$-ground} if every vertex of $K$ is adjacent to all but at most $k$ vertices of $\sigma$.
A simplicial complex $K$ is \textbf{$\boldsymbol{(n,k)}$-grounded} if there is an $n$-simplex that is a $k$-ground.
We want to use the following theorem by Belk and Forrest, which is also referenced as Theorem 4.11 and applied in \cite{SWZ19}.

\begin{theorem}[Theorem 4.9 of \cite{BF19}]
\label{thm.grounded.implies.connected}
For all $m, k \in \mathbb{N}$, each $(mk, k)$-grounded flag complex is $(m-1)$-connected.
\end{theorem}

Denote by $d$ the maximum amount of hyperedges among the replacement hypergraphs of $\R$.

\begin{definition}
An hyperedge replacement system is \textbf{$\boldsymbol{m}$-contractive} for a compatible self-similar tuple $\mathbb{G}$ if all but finitely many hypergraph expansions admit at least $m$ parallel simple $\pi$-contractions based on $\mathbb{G}$.
It is \textbf{$\boldsymbol{\infty}$-contractive} if it is $m$-contractive for all $m \in \mathbb{N}$
\end{definition}

\begin{lemma}
Let $x \in \mathrm{Ob}(\C)$ and assume that $x$ admits at least $m$ parallel simple $\pi$-contractions.
Then the complex $K_x$ is $(m,d)$-grounded.
\end{lemma}

\begin{proof}
By assumption, $x$ admits at least $m$ parallel simple $\pi$-contractions $f_1, \dots, f_m$.
They form an $m$-simplex $\{ [f_1], \dots, [f_m] \}$, which we want to prove to be a $d$-ground.

If $\C^\times(x,y)$ is not empty, say $h \in \C^\times(x,y)$, then $x$ and $y$ have the same simple $\pi$-contractions up to conjugation by $h$, so $K_x$ and $K_y$ are isomorphic.
Then, without loss of generality, let us assume that $x$ is as described in \cref{rmk.parallel.pi.contractions}, i.e., the parallel simple $\pi$-contractions $f_1, \dots, f_m$ are simple contractions of $x$.

%
Each simple contraction $f$ of $x$ contracts a subhypergraph of $x$ consisting of at most $d$ hyperedges.
Since the simple contractions $f_1, \dots, f_m$ are parallel, each hyperedge of $x$ can be involved in at most one of them.
Thus, a simple $\pi$-contraction $f$ of $x$ is not parallel to at most $d$ simple contractions in $\{ [f_1], \dots, [f_m] \}$.
This means that $f$ is adjacent to all vertices of $\{ [f_1], \dots, [f_m] \}$ but at most $d$, as needed.
\end{proof}

Thanks to \cref{thm.grounded.implies.connected}, we have the following fact.

\begin{corollary}
\label{cor.contractive.implies.connected}
Let $x \in \mathrm{Ob}(\C)$ and assume that $x$ admits at least $m$ parallel simple $\pi$-contractions.
Then $K_x$ is $(\lfloor m/d \rfloor -1)$-connected.
\end{corollary}

Suppose that the hyperedge replacement system is $m$-contractive.
Then every $K_x$ is $(\lfloor m/d \rfloor -1)$-connected except for a finite amount of $x \in \mathrm{Ob}(\C)$.
Since for each $n \in \mathbb{N}$ there are only finitely hypergraphs of $\mathrm{Ob}(\C)$ with at most $n$ hyperedges, the set $\{x \in \mathrm{Ob}(\C) \mid \rho(x) \leq n\}$ is finite.
The rest of the hypotheses of \hyperref[thm.witzel]{Witzel's Theorem} were proved in \cref{cor.our.category.is.ore}, \cref{prop.our.garside.family} and \cref{cor.finiteness.stabilizers,,cor.contractive.implies.connected}, so we ultimately have our desired result.

\begin{theorem}
\label{thm.main}
If an almost expanding replacement system $\R$ is $m$-contractive and the groups of a compatible self-similar tuple $\mathbb{G}$ all have type $F_{\lfloor m/d \rfloor}$, then the ESS group $E_\R\mathbb{G}$ has type $F_{\lfloor m/d \rfloor}$.
In particular, if the replacement system is $\infty$-contractive and the groups of $\mathbb{G}$ have type $F_\infty$, then $E_\R\mathbb{G}$ has type $F_\infty$.
\end{theorem}

We will see some application of this theorem in the next subsection.
We will also see that the replacement system being $m$-contractive (respectively, $\infty$-contractive) is not a necessary condition for $E\mathbb{G}$ having type $F_{\lfloor m/d \rfloor}$ (respectively, $F_\infty$), with the Houghton groups in \cref{cor.example.known.results}(D) and the topological full groups of edge shifts in \cref{sub.Matui}, respectively.


\section{Applications of \texorpdfstring{\cref{thm.main}}{}}

The following are easy applications of \cref{thm.main} that recover known results:

\begin{corollary}
\label{cor.example.known.results}
\begin{enumerate}[label=\Alph*)]
    \item[]
    \item The Higman-Thompson groups from \cite{higman1974finitely} are based on $\infty$-contractive replacement systems, so they are $F_\infty$ (originally proved in \cite{BROWN198745}).
    \item Scott-R\"{o}ver-Nekrashevych groups from \cite{Scott,Rover,NekrashevychGroups} are based on the same replacement systems as the Higman-Thompson groups, so they inherit the same finiteness properties of the self-similar group (originally proved in \cite[Theorem 4.15]{SWZ19}).
    \item The Thompson-like groups $QF$, $QT$ and $QV$ from \cite{QV} are based on $\infty$-contractive hyperedge replacement systems, so they are $F_\infty$ (originally proved in \cite{QVfinitenessproperties}).
    \item For each $n \in \mathbb{N}_{\geq2}$, the Houghton group $H_n$ from \cite{Houghton1978TheFC} is based on an $n$-contractive replacement system (\cref{fig.houghton.replacement.system}), so it has type $F_{\lfloor n/2 \rfloor}$.
    This is not optimal, since \cite{BROWN198745} showed that $H_n$ has type $F_{n-1}$.
\end{enumerate}
\end{corollary}

It is not hard to see that the complexes $K_x$ arising from our construction are the same that were used in \cite{BROWN198745}.
What causes this loss of precision lies in the general computations from \cref{sub.connectivity}.

\begin{remark}
\label{rmk.non.example.known.results}
We also note the following non-examples.
\begin{enumerate}[label=\Alph*)]
    \item The basilica replacement system from \cite{BF19} is $2$-contractive but not $3$-contractive.
    In fact, \cite{TBnotfinpres} shows that the associated rearrangement group is finitely generated but not finitely presented.
    \item The homeomorphism group of the Sierpinski triangle is finite by \cite[Theorem 2.3.6]{TriangleGasket}, but the replacement system is not even $2$-contractive (the hypergraph expansions $\Gamma_0 \triangleleft X \triangleleft 3 \triangleleft \dots \triangleleft 3$ feature a unique $\pi$-contraction).
    \item The Apollonian gasket is only $1$-contractive (essentially for the same reason as the Sierpinski triangle), so our \cref{thm.main} does not apply.
\end{enumerate}
\end{remark}

The following subsections explore new results arising from \cref{thm.main}.

\subsection{Finiteness Properties for Rearrangement groups}
\label{sub.finiteness.rrg.gps}

By \cite[Lemma 4.13]{BF19}, if an expanding edge replacement system has finite branching and connected replacement graphs, then it is $\infty$-contractive.
Thus, when dealing with graphs, if the self-similar tuple is trivial, our \cref{thm.main} implies \cite[Theorem 4.1]{BF19}.
However, our \cref{thm.main} applies to certain rearrangement groups that do not satisfy the hypotheses of \cite[Theorem 4.1]{BF19}, as we will shortly see.

The following definition is from \cite{Finitary}.
It is also featured in \cite[Section 3.8.2]{NekSSG} in the case of full shifts.

\begin{definition}
A \textbf{finitary automorphism} is an automorphism $\phi$ of an initial edge shift $\Omega_s$ for which there exists an $N \in \mathbb{N}$ such that, for every $w \in \mathbb{L}_s$ of length $N$, the state $\phi | w$ is trivial.
A \textbf{finitary self-similar tuple} is a self-similar tuple whose every element is finitary.
\end{definition}

In practice, a finitary automorphism $\phi$ admits an $N \in \mathbb{N}$ such that $\phi$ rigidly permutes the level-$N$ cones, doing nothing else below level-$N$.
In fact, this implies that a finirary automorphism is always a rearrangement (\cref{def.rearrangement}).
In particular, we have the following fact.

\begin{proposition}
If $\mathbb{G}$ is a finitary self-similar tuple, then $E_\R\mathbb{G} = RG_\R$
\end{proposition}

This fact allows to apply \cref{thm.main} to certain rearrangement groups that do not satisfy the hypotheses of \cite[Theorem 4.1]{BF19}.
As explained in \cite[Remark 4.6]{BF19}, Belk and Forrest's Theorem does not apply to certain simple examples, such as $F \rtimes \mathbb{Z}_2$.
Our \cref{thm.main} essentially solves that issue (it can be applied to $F \rtimes \mathbb{Z}_2$ in the same way that we will apply it to other examples, see below).
However, it still does not yield a necessary condition for $F_\infty$:
for example, we will note in \cref{rmk.bad.shift} that there is a topological full group of an edge shifts which is $F_\infty$ but the replacement system associated to is not $\infty$-contractive.

As examples, in the remainder of this subsection we will prove that the airplane rearrangement group $T_A$ (introduced in \cite[Example 2.13]{BF19} and studied in \cite{airplane}) and the rearrangement groups $G_n$ of Wa\.zewski dendrites (from \cite{dendrite}) have type $F_\infty$.
The proofs are similar:
they both consist of finding a finitary self-similar tuple that allows to invert the edge orientation of blue edges and then show that, with these self-similar tuples, the groups are $\infty$-contractive.

\begin{figure}
\centering
\begin{tikzpicture}
    \node at (0,1.75) {$\Gamma_0$};
    \node[vertex] (b) at (0,0) {};
    \node[vertex] (i) at (180:1.5) {};;
    \node[vertex] (r2) at (125:1.5) {};
    \node[vertex] (rn-1) at (55:1.5) {};
    \node[vertex] (t) at (0:1.5) {};
    \draw[red,edge] (b) to node[below]{$1_0$} (i);
    \draw[red,edge,dotted] (b) to (r2);
    \draw[red] (90:.75) node{$i_0$};
    \draw[red,edge,dotted] (b) to (rn-1);
    \draw[red,edge] (b) to node[below]{$n_0$} (t);
    \begin{scope}[xshift=4.4cm]
    \node at (0,1.75) {$R_{\text{\textcolor{blue}{b}}}$};
    \node[vertex] (b) at (0,0) {};
    \node[vertex] (i) at (180:1.5) {}; \draw (0:-1.5) node[above]{$\lambda_R(1)$};
    \node[vertex] (r2) at (125:1.5) {};
    \node[vertex] (rn-1) at (55:1.5) {};
    \node[vertex] (t) at (0:1.5) {}; \draw (0:1.5) node[above]{$\lambda_R(2)$};
    \draw[blue,edge] (b) to node[below]{$1_I$} (i);
    \draw[red,edge,dotted] (b) to (r2);
    \draw[red] (90:.75) node{$i_I$};
    \draw[red,edge,dotted] (b) to (rn-1);
    \draw[blue,edge] (b) to node[below]{$n_I$} (t);
    \end{scope}
    \begin{scope}[xshift=8.8cm]
    \node at (0,1.75) {$R_{\text{\textcolor{red}{r}}}$};
    \node[vertex] (b) at (0,0) {};
    \node[vertex] (i) at (180:1.5) {}; \draw (0:-1.5) node[above]{$\lambda_R(1)$};
    \node[vertex] (r2) at (125:1.5) {};
    \node[vertex] (rn-1) at (55:1.5) {};
    \node[vertex] (t) at (0:1.5) {}; \draw (0:1.5) node[above]{$\lambda_R(2)$};
    \draw[blue,edge] (b) to node[below]{$1_E$} (i);
    \draw[red,edge,dotted] (b) to (r2);
    \draw[red] (90:.75) node{$i_E$};
    \draw[red,edge,dotted] (b) to (rn-1);
    \draw[red,edge] (b) to node[below]{$n_E$} (t);
    \end{scope}
\end{tikzpicture}
\caption{The Wa\.zewski dendrite replacement systems $\mathcal{D}_n$.}
\label{fig.dendrite.replacement.systems}
\end{figure}
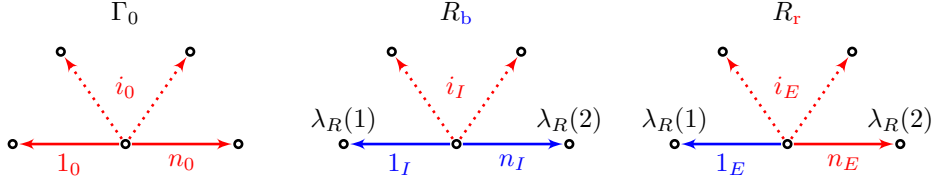

\subsubsection{Dendrite rearrangement groups have type \texorpdfstring{$F_\infty$}{F-infty}}
\label{sub.phin}

For each $n \geq 3$, consider the dendrite replacement system depicted in \cref{fig.dendrite.replacement.systems}.
Colors are not present in \cite{dendrite}, but it is easy to see that their inclusion does not change the limit space nor the group.
Note that an edge of a graph expansion of $\mathcal{D}_n$ is red if it is incident on a degree-$1$ vertex and is blue if both its boundary vertices have degree $n$.

Let $G_{\text{\textcolor{red}{r}}}$ be the trivial group and $G_{\text{\textcolor{blue}{b}}} = \langle \phi_n \rangle$, where $\phi_n$ is defined as
\[ \widehat{\phi_n} \colon 1_I \mapsto n_I,\, 2_I \mapsto 2_I,\, \dots, (n-1)_I \mapsto (n-1)_I,\, n_I \mapsto 1_I, \]
and with state $\phi_n | w$ trivial for each non-empty word $w$.
Clearly $\phi_n$ is a finitary automorphism of the color blue of $\mathcal{D}_n$, so $\mathbb{G}_{\mathcal{D}_n} \coloneq ( G_{\text{\textcolor{blue}{b}}}, G_{\text{\textcolor{red}{r}}} )$ is a finitary self-similar tuple that is compatible with $\mathcal{D}_n$ and thus $RG_{\mathcal{D}_n} = E_{\mathcal{D}_n} \mathbb{G}_{\mathcal{D}_n}$.
Note that $\langle \phi_n \rangle$ is a cyclic group of order $2$, so both groups of the self.similar tuple have type $F_\infty$.
We are only missing $\infty$-contractivity.

\begin{lemma}
\label{lem.dendrites}
Each dendrite replacement system $\mathcal{D}_n$ is $\infty$-contractive with the self-similar tuple $\mathbb{G}_{\mathcal{D}_n}$.
\end{lemma}

\begin{proof}
First observe that a simple $\pi$-contraction in this setting is just a simple contraction that ignores the orientation of blue edges.
Indeed, if two graph expansions only differ by the orientation of one or more blue edges, there is a diagram that has them as domain and range graphs, where each edge whose orientation needs reversing is labeled by $\phi_n$.

Colors aside, each graph expansion of $\mathcal{D}_n$ is a finite tree whose internal vertices have degree $n$ (where by \textit{degree} of a vertex we mean the amount of edges that are incident to it).
A graph expansion $\Gamma$ features precisely $(n-2)!$ $\pi$-contractions of $\Gamma$ for every internal vertex $v$ of $\Gamma$ (note that this fails when looking at ``standard'' contractions, which take into account the edge orientation).
In the context we just described, let us say that the $\pi$-contraction is \textit{centered} at $v$.
Then $\pi$-contractions of $\Gamma$ are parallel precisely when they are centered at distinct vertices which are not adjacent.
Evidently, the maximal amount of distinct vertices that are not adjacent grows to infinity as the number of vertices increases, so for each $m \in \mathbb{N}$ there are only finitely many graph expansions featuring at most $m$ parallel $\pi$-contractions.
This means that $\mathcal{D}_n$ is $m$-contractive for all $m \in \mathbb{N}$, i.e., it is $\infty$-contractive.
\end{proof}

\subsubsection{The airplane rearrangement group has type \texorpdfstring{$F_\infty$}{F-infty}}

For the airplane replacement system $\mathcal{A}$ (\cref{fig.airplane.replacement.system}), let $\phi_\infty \in \Omega_\text{\textcolor{blue}{b}}$ be defined by
\[ \widehat{\phi_\infty} \colon I \mapsto F,\, F \mapsto I,\, T \mapsto B,\, B \mapsto T, \]
and with state $\phi_\infty | w$ trivial for each non-empty word $w$.
Again, $\phi_\infty$ is a finitary automorphism of the color \textcolor{blue}{b}.
If we let $G_\text{\textcolor{blue}{b}} \coloneq \langle \phi_\infty \rangle$ and $G_\text{\textcolor{red}{r}}$ be trivial, then $\mathbb{G}_\mathcal{A} \coloneq ( G_\text{\textcolor{blue}{b}}, G_\text{\textcolor{red}{r}} )$ is a finitary self-similar tuple that is compatible with $\mathcal{A}$.
Once again, $RG_\mathcal{A} = E_\mathcal{A} \mathbb{G}_\mathcal{A}$ and the groups of the self-similar tuple are finite and thus have type $F_\infty$, so we are only missing $\infty$-contractivity.

\begin{lemma}
\label{lem.airplane}
The airplane replacement system $\mathcal{A}$ is $\infty$-contractive with the self-similar tuple $\mathbb{G}_\mathcal{A}$.
\end{lemma}

\begin{proof}
As done in \cref{lem.dendrites}, first note that a $\pi$-contraction here is either a ``standard'' red contraction (no orientation reversal allowed) or a blue contraction that ignores the orientation of blue edges.
This is, again, because $\phi_\infty$ allows to switch the orientation of blue edges.

Each blue $\pi$-contraction corresponds to a red cycle made of two sole red edges, while a red $\pi$-contraction corresponds to two adjacent red edges together with a blue edge that is attached to that unique red cycle.
Let us build a tree $T_\Gamma$ from a graph expansion $\Gamma$ of $\mathcal{A}$:
as vertices, it has all the red cycles of $\Gamma$ and all the vertices that are solely adjacent to a blue edge(let us call these \textit{cycles} and \textit{leaves});
two vertices are connected by an edge in $T_\Gamma$ if $\Gamma$ has a blue edge joining the corresponding cycles or leaves.
This tree comes equipped with a \textit{rotation system}, i.e., a cyclic ordering of the edges adjacent to each vertex, corresponding to the cyclic order of the blue edges attached to each red cycle.
Note that two leaves cannot be adjacent, whereas two cycles or a cycle and a leaf can be.
Moreover, the vertices that we called \textit{leaves} truly correspond to the leaves of $T_\Gamma$.
Now, blue $\pi$-contractions correspond to cycles of degree $2$ of $T_\Gamma$ and red $\pi$-contractions to its leaves.
Furthermore, $\pi$-contractions are parallel precisely when the blue $\pi$-contractions do not correspond to cycles that are adjacent in $T_\Gamma$ and the red $\pi$-contractions do not correspond to leaves of $T_\Gamma$ whose unique adjacent blue edges are consecutive in the rotation system around a cycle (a red and a blue $\pi$-contraction are always parallel, since we only allow contractions leading to elements of $\mathrm{Ob}(\C_{\mathcal{A},\mathbb{G}_\mathcal{A}})$).

Let $C$ be the total number of contractions of $\Gamma$, $V$ the amount of vertices of $T_\Gamma$, $V_1$ be the number of leaves, $\widetilde{V_1}$ that of leaves whose sole adjacent vertex has degree at least $3$, $V_2$ be the number of vertices of degree $2$ and $\widetilde{V_2}$ the number of vertices of degree $2$ that are not adjacent to a leaf.
Then $C = \widetilde{V_1} + V_2 = V_1 + \widetilde{V_2}$, so
\[ 2C = (\widetilde{V_1} + V_2) + (V_1 + \widetilde{V_2}) \geq V_1 + V_2 \geq V/2. \]
(The last inequality holds in any finite tree).
Then, ultimately, $C \geq V/4$.

Now, assume that $T_\Gamma$ has at most $V$ vertices.
Then it must admit at least $V/4$ contractions.
Each contraction can be non-parallel to at most two other contractions, so there must be $C/3 \geq V/12$ parallel contractions.
Since for each $m=V$ there are finitely many graph expansions of $\mathcal{A}$ with at most $V$ vertices, we have that only finitely many graph expansions can have less than $V/12$ parallel contractions, so ultimately $\mathcal{A}$ is $\infty$-contractive.
\end{proof}

\begin{corollary}
The airplane rearrangement group $T_A$ and the dendrite rearrangement groups $G_n$ have type $F_\infty$.
\end{corollary}

It was stated in \cite{BF19}, but not proved, that $T_A$ has type $F_\infty$, so this results fills the gap.
For dendrite rearrangement groups, this answers \cite[Question 3.7]{dendrite}.

In general, one can attempt to apply this technique to rearrangement groups whenever a replacement system $\R$ admits a finitary self-similar tuple $\mathbb{G}$ that features some element that permutes non-trivially the boundary vertices of a cell (such as $\phi_n$ and $\phi_\infty$ from our examples).

\subsection{ESS Groups of Wa\.zewski Dendrites}
\label{sub.dendrites}

By \cref{lem.dendrites}, dendrite replacement systems $\mathcal{D}_n$ are $\infty$-contractive if the self-similar tuple has $\langle \phi_n \rangle$ in its blue entry.
Thus, by \cref{thm.main}, the finiteness property of an ESS group $E_{\mathcal{D}_n} (\langle \phi \rangle, G_{\text{\textcolor{red}{r}}})$ of homeomorphisms of a dendrite $D_n$ is at least that of $G_{\text{\textcolor{red}{r}}}$.
This is also the case for the airplane by \cref{lem.airplane}, but dendrites provide particularly flexible examples because of the graph structure of their replacement graphs.
Indeed, for each $n \geq 2$, if $G$ is a self-similar group acting on the rooted $n$-regular tree, one can have a copy $G_{\text{\textcolor{red}{r}}}$ of $G$ acting on the cells of $D_{n+1}$ which correspond to words without instances of the symbol $1$.

As an example, consider the replacement system $\mathcal{D}_3$ and the self-similar tuple $\mathbb{G}=(\langle \phi_3 \rangle, G)$, where $\phi_3$ defined in \cref{sub.phin} and $G$ is a copy of the first Grigorchuk group (famously generated by involutions $\{a,b,c,d\}$) acting on the words in $\{2,3\}$.
Since the elements of $G$ do not permute the boundary points of red cells, the $\pi$-contractions are the same that we described in \cref{sub.finiteness.rrg.gps}, so the replacement system is $\infty$-contractive with the self-similar tuple $\mathbb{G}$.
\cref{fig.grig.diagrams} portrays a diagram and two of its expansions (the subscripts of symbols was dropped to improve clarity, as they do not provide information that is not already encoded).

\begin{figure}
\centering
\begin{subfigure}[t]{\textwidth}\centering
\begin{tikzpicture}[font=\footnotesize]
    \begin{scope}[xshift=-2.75cm]
    \node[vertex] (l) at (180:2) {};
    \node[vertex] (c) at (0:0) {};
    \node[vertex] (t) at (90:2) {};
    \node[vertex] (r) at (0:2) {};
    \node[vertex] (rc) at (0:1) {};
    \node[vertex] (rt) at ($(0:1)+(90:1)$) {};
    \draw[edge,red] (c) to node[above]{$1a$} (l);
    \draw[edge,red] (c) to node[above left]{$2$} (t);
    \draw[edge,blue] (rc) to node[below]{$31$} (c);
    \draw[edge,red] (rc) to node[left]{$32$} (rt);
    \draw[edge,red] (rc) to node[below]{$33b$} (r);
    \end{scope}
    \draw[-to] (-.15,0) -- node[above]{$\f$} (.15,0);
    \begin{scope}[xshift=2.75cm]
    \node[vertex] (l) at (180:2) {};
    \node[vertex] (lc) at (180:1) {};
    \node[vertex] (lb) at ($(180:1)+(270:1)$) {};
    \node[vertex] (c) at (0:0) {};
    \node[vertex] (t) at (90:2) {};
    \node[vertex] (r) at (0:2) {};
    \draw[edge,red] (c) to node[above left]{$\f(32)$} (t);
    \draw[edge,red] (c) to node[above]{$\f(33)$} (r);
    \draw[edge,blue] (lc) to node[above]{$\f(31)$} (c);
    \draw[edge,red] (lc) to node[right]{$\f(2)$} (lb);
    \draw[edge,red] (lc) to node[above]{$\f(1)$} (l);
    \end{scope}
    \node[align=center] at (6.3,0) {$\pi_{31} = (1\,2)$};
\end{tikzpicture}
\end{subfigure}
\\
\
\\
\begin{subfigure}[t]{\textwidth}\centering
\begin{tikzpicture}[font=\footnotesize]
    \begin{scope}[xshift=-2.75cm]
    \node[vertex] (l) at (180:2) {};
    \node[vertex] (lb) at ($(180:1)+(270:1)$) {};
    \node[vertex] (lc) at (180:1) {};
    \node[vertex] (c) at (0:0) {};
    \node[vertex] (t) at (90:2) {};
    \node[vertex] (r) at (0:2) {};
    \node[vertex] (rc) at (0:1) {};
    \node[vertex] (rt) at ($(0:1)+(90:1)$) {};
    \draw[edge,red] (lc) to node[above]{$13$} (l);
    \draw[edge,red] (lc) to node[left]{$12$} (lb);
    \draw[edge,blue] (lc) to node[above]{$11$} (c);
    \draw[edge,red] (c) to node[above left]{$2$} (t);
    \draw[edge,blue] (rc) to node[below]{$31$} (c);
    \draw[edge,red] (rc) to node[left]{$32$} (rt);
    \draw[edge,red] (rc) to node[below]{$33b$} (r);
    \end{scope}
    \draw[-to] (-.15,0) -- node[above]{$\f$} (.15,0);
    \begin{scope}[xshift=2.75cm]
    \node[vertex] (l) at (180:2) {};
    \node[vertex] (lc) at (180:1) {};
    \node[vertex] (lb) at ($(180:1)+(270:1)$) {};
    \node[vertex] (llc) at (180:1.5) {};
    \node[vertex] (llb) at ($(180:1.5)+(270:.5)$) {};
    \node[vertex] (c) at (0:0) {};
    \node[vertex] (t) at (90:2) {};
    \node[vertex] (r) at (0:2) {};
    \draw[edge,red] (c) to node[above left]{$\f(32)$} (t);
    \draw[edge,red] (c) to node[above]{$\f(33)$} (r);
    \draw[edge,blue] (lc) to node[above]{$\f(31)$} (c);
    \draw[edge,red] (lc) to node[below right]{$\f(2)$} (lb);
    \draw[edge,red] (llc) to node[above,xshift=-.25cm]{$\f(12)$} (l);
    \draw[edge,blue] (llc) to node[above]{$\f(11)$} (lc);
    \draw[edge,red] (llc) to node[below left]{$\f(13)$} (llb);
    \end{scope}
    \node[align=center] at (6.3,0) {$\pi_{31} = (1\,2)$};
\end{tikzpicture}
\end{subfigure}
\\
\
\\
\begin{subfigure}[t]{\textwidth}\centering
\begin{tikzpicture}[font=\footnotesize]
    \begin{scope}[xshift=-2.75cm]
    \node[vertex] (l) at (180:2) {};
    \node[vertex] (c) at (0:0) {};
    \node[vertex] (t) at (90:2) {};
    \node[vertex] (r) at (0:2) {};
    \node[vertex] (rc) at (0:1) {};
    \node[vertex] (rrc) at (0:1.5) {};
    \node[vertex] (rrt) at ($(0:1.5)+(90:.5)$) {};
    \node[vertex] (rt) at ($(0:1)+(90:1)$) {};
    \draw[edge,red] (c) to node[above]{$1a$} (l);
    \draw[edge,red] (c) to node[above left]{$2$} (t);
    \draw[edge,blue] (rc) to node[below]{$31$} (c);
    \draw[edge,red] (rc) to node[left]{$32$} (rt);
    \draw[edge,red] (rrc) to node[right]{$332a$} (rrt);
    \draw[edge,blue] (rrc) to node[below]{$331$} (rc);
    \draw[edge,red] (rrc) to node[below,xshift=.125cm]{$333c$} (r);
    \end{scope}
    \draw[-to] (-.15,0) -- node[above]{$\f$} (.15,0);
    \begin{scope}[xshift=2.75cm]
    \node[vertex] (l) at (180:2) {};
    \node[vertex] (lc) at (180:1) {};
    \node[vertex] (lb) at ($(180:1)+(270:1)$) {};
    \node[vertex] (c) at (0:0) {};
    \node[vertex] (t) at (90:2) {};
    \node[vertex] (rc) at (0:1) {};
    \node[vertex] (rt) at ($(0:1)+(90:1)$) {};
    \node[vertex] (r) at (0:2) {};
    \draw[edge,red] (c) to node[above left]{$\f(32)$} (t);
    \draw[edge,red] (rc) to node[right]{$\f(332)$} (rt);
    \draw[edge,blue] (rc) to node[below]{$\f(331)$} (c);
    \draw[edge,red] (rc) to node[below]{$\f(333)$} (r);
    \draw[edge,blue] (lc) to node[above]{$\f(31)$} (c);
    \draw[edge,red] (lc) to node[right]{$\f(2)$} (lb);
    \draw[edge,red] (lc) to node[above]{$\f(1)$} (l);
    \end{scope}
    \node[align=center] at (6.3,0) {$\pi_{31} = (1\,2)$};
\end{tikzpicture}
\end{subfigure}
\caption{A diagram of $E_{\mathcal{D}_n} (\langle \phi_3 \rangle, G)$ and two expansions.}
\label{fig.grig.diagrams}
\end{figure}
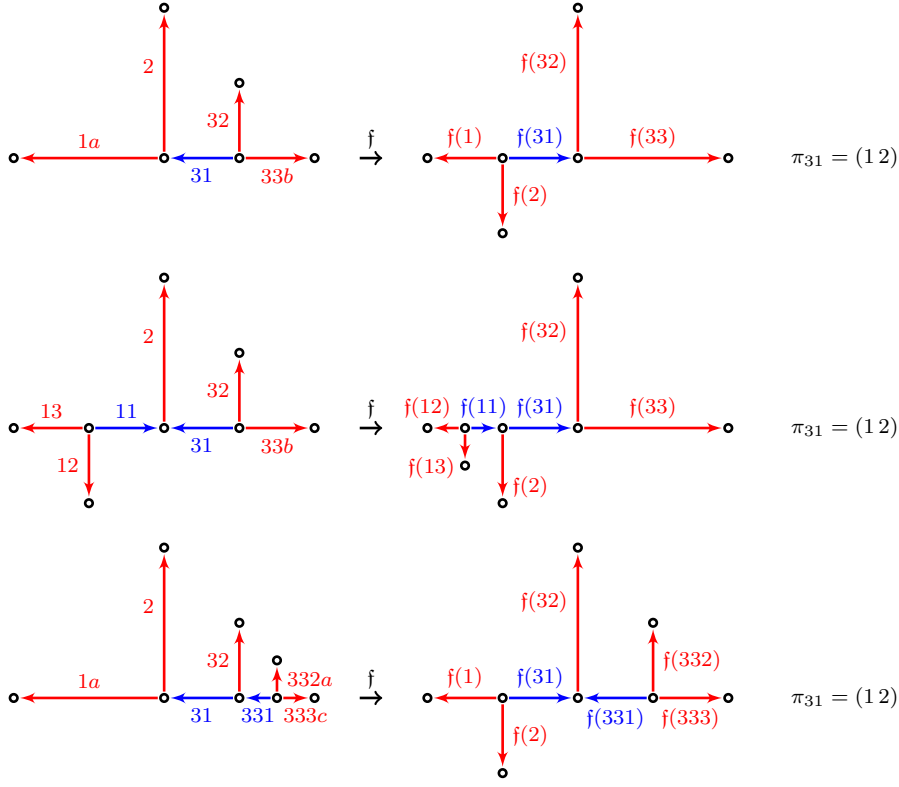

As an immediate consequence of \cref{thm.main}, we have the following fact. 

\begin{proposition}
The group $E_{\mathcal{D}_{3}} \mathbb{G}$ is finitely generated.
\end{proposition}

The analogous construction with Thompson group $V$, known as R\"{o}ver group (see \cite{Rover}), has type $F_\infty$ by \cite{RoverFinfty}, so it makes sense to ask the following questions.

\begin{question}
Is $E_{\mathcal{D}_3} \mathbb{G}$ finitely presented?
Is it $F_\infty$?
\end{question}

In general, using higher degree dendrite replacement systems $\mathcal{D}_n$, one can have any self-similar group on the $(n-1)$-regular rooted tree act on red edges.
Thus, for any $m \in \mathbb{N}$, using the self-similar groups provided by \cite[Theorem 6.13]{SWZ19}, one can build groups of type $F_m$ acting on dendrites.
However, determining whether these groups are sharply $F_m$ is beyond the scope of this paper.

\subsection{ESS Groups of Edge Shifts}
\label{sub.Matui}

In \cite{Matui} Matui showed that the topological full group of an irreducible (one-sided) edge shift has type $F_\infty$. Matui's proof starts by using results by Matsumoto \cite{Matsumoto10,Matsumoto13} to show that, up to a group isomorphism, we can make the following assumption on the edge shift.

\begin{assumption}
\label{ass.Matui}
For each color $i$, the replacement graph $R_i$ includes exactly an edge of every color $j \neq i$ and at least $2$ edges colored by $i$.
\end{assumption}

Edge shifts satisfying this assumption are $\infty$-contractive.
Indeed, let $L$ be the maximum amount of $i$-colored edges in each $R_i$.
Suppose that an expansion $\Gamma$ admits at most $m$ parallel contractions and fix an instance of those $m$ contractions.
Let $n$ be the number of total edges of $\Gamma$ and $r$ the number of edges that are not involved in the contractions.
Each contraction involves at most $L+k-1$ edges, so
\[ r \geq n - m(L+k-1). \]
The remaining edges cannot include more than $L$ edges of each color, or there would be another parallel contraction, violating the maximality of $m$, so
\[ r \leq kL. \]
Putting the two inequalities together, we have
\[ n \leq kL+m(L+k-1). \]
Hence, for every fixed $m$, there is a bound on the amount of edges of $\Gamma$, meaning that there are only finitely many graph expansions that allow at most $m$ reductions, for each $m \in \mathbb{N}$, as needed.

\begin{remark}
\label{rmk.bad.shift}
Note that the Matui's assumption is needed as it is straighforward to show that the edge shift described in \cref{ex.bad.shift} is not $\infty$-contracting.
\end{remark}

By virtue of the previous paragraph and exploiting \cref{thm.main}, we have a generalization of Matui's theorem in our context.

\begin{corollary}
Let $G$ be a a topological full group of an irreducible (one-sided) edge shift and let $\mathcal{R}$ be the edge shift satisfying \cref{ass.Matui}. If $\mathcal{G}$ is a compatible self-similar tuple where all the groups are of type $F_{n}$, then $E_{\mathcal{R}} \mathbb{G}$ is of type $F_{n}$.
\end{corollary}

In the case in which a self-similar groupoid from \cite{Deaconu} is a self-similar tuple (i.e., when there are no isomorphisms between different colors), this provides a criterion for addressing questions posed in \cite{Deaconu} after Corollary 5.5.

\begin{question}
As mentioned at the beginning of this subsection, Matsumoto's results \cite{Matsumoto10,Matsumoto13} allow us to substitute the replacement system with an $\infty$-contractive one without changing the given topological full group.
Does this also provide isomorphisms of ESS groups?
\end{question}


\section*{Acknowledgements}

The authors are thankful to Francesco Matucci for helpful advices, to Jim Belk and Bradley Forrest for kindly providing their image of the airplane limit space and to Rachel Skipper for useful discussions about finiteness properties.


\printbibliography[heading=bibintoc]

\end{document}